\documentclass[aos]{imsart}

\RequirePackage{amsthm,amsmath,amsfonts,amssymb}
\RequirePackage[authoryear]{natbib}
\RequirePackage[colorlinks,citecolor=blue,urlcolor=blue]{hyperref}
\RequirePackage{graphicx}

\startlocaldefs
\theoremstyle{plain}
\newtheorem{theorem}{Theorem}
\newtheorem{lemma}[theorem]{Lemma}
\newtheorem{proposition}[theorem]{Proposition}
\newtheorem{corollary}[theorem]{Corollary}
\newtheorem*{result}{Result}
\theoremstyle{remark}
\newtheorem{definition}{Definition}
\newtheorem{remark}{Remark}
\newtheorem{assumption}{Assumption}

\usepackage[ruled,linesnumbered]{algorithm2e}
\usepackage{bm}
\usepackage{subcaption}
\usepackage{mathtools}
\usepackage{booktabs}
\usepackage{todonotes}

\newcommand{\changed}{}

\newcommand{\mbb}[1]{\mathbb{#1}}
\newcommand{\mcal}[1]{\mathcal{#1}}
\newcommand{\dd}{\mathtt{d}}

\newcommand{\floor}[1]{\left\lfloor #1\right\rfloor}
\newcommand{\ceil}[1]{\left\lceil #1\right\rceil}

\newcommand{\Var}{\text{Var}}
\newcommand{\grad}{\nabla}
\newcommand*\quot[2]{{^{\textstyle #1}/_{\textstyle #2}}}
\newcommand{\Vol}{\text{Vol}}

\newcommand{\bmX}{\bm{X}}
\newcommand{\bmZ}{\bm{Z}}
\newcommand{\bmx}{\bm{x}}
\newcommand{\bmz}{\bm{z}}
\newcommand{\bmW}{\bm{W}}
\newcommand{\bmw}{\bm{w}}

\endlocaldefs

\begin{document}

\begin{frontmatter}
\title{Privacy Guarantees in Posterior Sampling under Contamination}
\runtitle{DP in Posterior under Contamination}

\begin{aug}
\author[A]{\fnms{Shenggang}~\snm{Hu}\ead[label=e1]{shenggang.hu@warwick.ac.uk}},
\author[B]{\fnms{Louis}~\snm{Aslett}\ead[label=e2]{louis.aslett@durham.ac.uk}},
\author[C]{\fnms{Hongsheng}~\snm{Dai}\ead[label=e3]{hongsheng.dai@newcastle.ac.uk}},
\author[C]{\fnms{Murray}~\snm{Pollock}\ead[label=e4]{murray.pollock@newcastle.ac.uk}},
\and
\author[A]{\fnms{Gareth} O.~\snm{Roberts}\ead[label=e5]{gareth.o.roberts@warwick.ac.uk}}
\address[A]{Department of Statistics, University of Warwick\printead[presep={,\ }]{e1,e5}}
\address[B]{Department of Mathematical Sciences, Durham University\printead[presep={,\ }]{e2}}
\address[C]{School of Mathematics, Statistics and Physics, Newcastle University\printead[presep={,\ }]{e3,e4}}
\end{aug}

\begin{abstract}
In recent years, differential privacy has been adopted by tech-companies and governmental agencies as the standard for measuring privacy in algorithms. 
In this article, we study differential privacy in Bayesian posterior sampling settings. 
We begin by considering differential privacy in the most common privatisation setting in which Laplace or Gaussian noise is injected into the output. 
In an effort to achieve better differential privacy, we consider adopting {\em Huber's contamination model} for use within privacy settings, and replace at random data points with samples from a heavy-tailed distribution ({\em instead} of injecting noise into the output). 
We derive bounds for the differential privacy level $(\epsilon,\delta)$ of our approach, without requiring bounded observation and parameter spaces, a restriction commonly imposed in the literature.
We further consider for our approach the effect of sample size on the privacy level and the rate at which $(\epsilon,\delta)$ converges to zero.
Asymptotically, our contamination approach is fully private with no information loss.
We also provide examples of inference models for which our approach applies, with theoretical convergence rate analysis and simulation studies.
\end{abstract}




\end{frontmatter}


\section{Introduction}
With the rapid growth of smart electronic devices in recent decades, the quantity of data generated and collected every day has increased by orders of magnitude. A byproduct of this expansion is increased awareness of personal privacy and concerns about privacy leakage as a consequence of utilising the collected data in studies.
In recent years, differential privacy has become of significant interest to computer scientists, statisticians and others for studying the properties and trade-offs in designing algorithms that protect the privacy of individuals, and is also being increasingly adopted by tech-companies, for instance, Google \citep{erlingsson2014rappor,aktay2020google}, Microsoft \citep{ding2017collecting,pereira2021us}, Apple  \citep{apple2017learning}, or government agencies \citep{machanavajjhala2008privacy, abowd2018us}.

Differential privacy, as defined in \cite{dwork2006calibrating,dwork2014algorithmic}, bounds the difference in response of a randomized algorithm $\mcal{A}:\mcal{X}^n\rightarrow\Omega$ when being supplied two neighbouring datasets of size $n$, i.e., one dataset is a modification of the other dataset by changing only one data point.
In particular, the response pattern should satisfy the following inequality for any measurable set $S\subset\Omega$ and any neighbouring datasets $\bm{X}$ and $\bm{Z}\in\mcal{X}^n$,
\begin{equation}\label{eq:dp}
   \mbb{P}(\mcal{A}(\bmX)\in S)\leq \exp(\epsilon)\mbb{P}(\mcal{A}(\bmZ)\in S) + \delta. 
\end{equation}
For simplicity, in this paper, we consider $\mcal{X}$ to be $\mbb{R}^d$ or a convex subset of $\mbb{R}^d$.
The algorithm is considered "more private" when $\epsilon$ and $\delta$ are smaller, in the sense that an arbitrarily large change in a single data point does not incur a significant change in the response pattern.

\subsection{Related Literature}
In most cases, differential privacy is achieved by purposefully injecting noise into the computation process, which perturbs the final outcome.
The injection of noise can happen in two places, either at the individual level before data collection or after data collection before releasing the result.
The most studied mechanism for privatising Bayesian inference is the latter, where the outcome is perturbed to ensure privacy.
Two main types of problems have been targeted: direct estimation problems \citep{foulds2016theory,wang2015privacy,zhang2016differential,bernstein2019differentially,kulkarni2021differentially} and sampling problems \citep{dimitrakakis2013bayesian,zhang2016differential,foulds2016theory,heikkila2019differentially,yildirim2019exact}.

The typical approach for private estimation problems is to directly inject Laplacian (or Gaussian) noise into intermediate steps or the final output \citep{wang2015privacy,foulds2016theory,zhang2016differential}.
\cite{bernstein2019differentially,kulkarni2021differentially} considered the injection of noise after moment approximations of the sufficient statistics to achieve privacy.
Gaussian/Laplacian mechanisms have the limitation that the magnitude of the noise scales with the sensitivity, defined as the maximal difference in the output of the target function with respect to neighbouring datasets.
In most applications, the sensitivity is unbounded unless the observation space is bounded.

Note that the definition of "Bayesian Differential Privacy" in \cite{triastcyn2019federated,triastcyn2020bayesian} differs from the previously mentioned work on differential privacy in Bayesian inference following \cite{dwork2006calibrating}'s definition.
In \cite{triastcyn2019federated,triastcyn2020bayesian}, the differing data point $z$ is treated as a random variable and marginalised, in contrast to (\ref{eq:dp}) where both $\bm{X}$ and $\bm{Z}$ are considered as arbitrary but fixed.

In this paper, we consider the differential privacy property of the Bayesian posterior sampling problem $\theta\sim\pi_n(\cdot|\bm{X})$ where $\pi_n$ is the posterior distribution given the dataset $\bm{X}\in\mcal{X}^n$ of size $n$, modelled by likelihood function $f(x;\theta)$, $\theta\in\Omega$, and prior $\pi_0(\theta)$.
Thereafter, we will call $\mcal{X}$ the observation (or data) space where the input data resides and $\Omega$ the parameter space where the model parameter $\theta$ takes values.

As noted in \cite{dimitrakakis2013bayesian}, Bayesian posterior sampling is itself a differentially private mechanism without additional noise under the condition that the log density is Lipschitz continuous in data with respect to some pseudometric $\rho$ for any fixed parameter $\theta$.
The level of privacy can be quantified by the pseudometric $\rho$ between the differing data points $\rho(\bmx,\bmz)$.
However, allowing the level of privacy to depend on $\rho(\bmx,\bmz)$ implies the algorithm can be arbitrarily non-private when the observation space is unbounded.
This limitation is shared by subsequent studies following \cite{dimitrakakis2013bayesian} on posterior sampling due to having a similar setup \citep{wang2015privacy,foulds2016theory,heikkila2019differentially,yildirim2019exact}.

An additional difficulty in bounding the privacy level of posterior sampling comes from the dependence of sensitivity on the parameter space.
The sensitivity of the posterior sampling problem is determined by the sensitivity of the log-posterior density function, which depends directly on the sampled parameter $\theta$.
Thus, the sensitivity could also be unbounded if the parameter space is unbounded \citep{zhang2016differential,foulds2016theory,heikkila2019differentially,yildirim2019exact, zhang2023dp}.

In contrast to adding noise at the outcome level, we consider Huber's contamination model \citep{huber2004robust} as a differential privacy mechanism.
A body of literature exists on quantifying the trade-off between privacy and efficient statistical inference under Huber's contamination model \citep{rohde2020geometrizing,cai2021cost,kroll2021density,li2022robustness}.
However, these analyses concern local differential privacy models, where privacy is assessed with respect to datasets of size $1$, whereas we study the differential privacy of the entire inference/sampling algorithm under Huber's contamination model, otherwise known as the central differential privacy model.

\begin{figure}[t]
    \centering
    \includegraphics[width=0.9\linewidth]{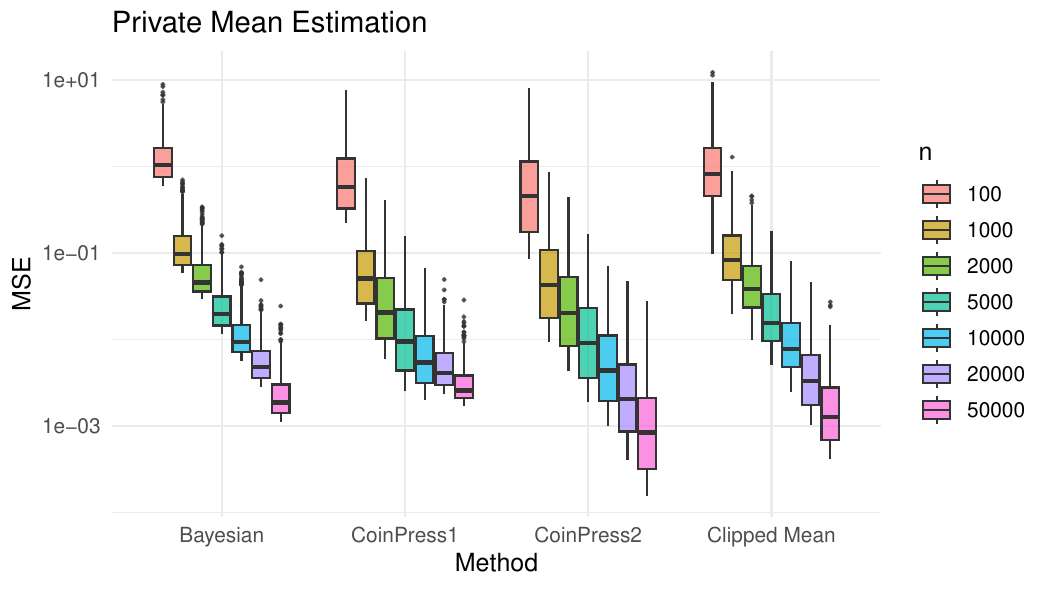}
    \caption{Mean Squared Error (MSE) Comparison for Private Mean Estimation Task. $n$ represents the size of the data set, ranging from $100$ to $50000$. MSE is plotted on log10 scale.}
    \label{fig:MSE_comparison_ct10}
\end{figure}

\subsection{Motivations}
Here, we briefly outline the motivation for choosing to privatise the posterior distribution instead of other steps in the inference process.
In practice, a broadly applicable approach to achieving differential privacy is to inject Gaussian/Laplace noise at an appropriate step.
The magnitude of the noise depends on the sensitivity of the quantity to be privatised.
In certain settings, such as logistic regression, the sensitivity of interest is naturally bounded.
However, in general, sensitivity is unbounded unless the analysis is restricted to bounded spaces.

When Gaussian/Laplace noise is injected in a straightforward manner, the injected noise usually scales polynomially with the diameter of the bounded domain. If the domain is conservatively chosen to be large, this scaling can result in poor performance.
Certain algorithms avoid the need for a bounded-domain assumption by applying appropriate clipping, e.g., \cite{biswas2020coinpress,huang2021instance} for private mean estimation.
However, such algorithms are often intricate to design and problem-specific.

In contrast, the Bayesian paradigm provides a general and well-studied foundation for statistical inference.
Inference proceeds via the posterior distribution, and sampling from the posterior can naturally be viewed as a “privatisation mechanism” due to its inherent randomness.
Although numerous works have investigated differential privacy guarantees for posterior sampling \citep{wang2015privacy,zhang2016differential,foulds2016theory,heikkila2019differentially,yildirim2019exact, zhang2023dp},  most of these results rely on strong assumptions about the likelihood function, which often entail bounded domains for either the dataset $\bmX$ or the parameter $\theta$.
Given the flexibility of the Bayesian framework, it is of significant interest to develop a general mechanism with theoretical guarantees for privatising posterior distributions without imposing restrictive assumptions.
This goal motivates the present work.

\subsection{Our contributions}
The novel contributions of this work are as follows:
\begin{itemize}
    \item We derive, in a probabilistic setting similar to the work of \cite{hall2013random}, the differential privacy properties of a general posterior sampling problem under Huber's contamination model \citep{huber2004robust} \textbf{without} assuming bounded parameter space or observation space, which are conditions that often appear explicitly or implicitly in the existing literature to ensure finite sensitivity;
    \item We characterise the rate of convergence for the differential privacy cost $(\epsilon,\delta)$ as a function of the number of data points $n$, which, to the best of our knowledge, has yet to be addressed;
    \item As an application, we show that the contaminated Bayesian inference may be used to solve private mean estimation problems with comparable performance to state-of-the-art frequentist methods (see Figure \ref{fig:MSE_comparison_ct10} and Section \ref{sec:simulation_result}).
\end{itemize}
Let $k_p$ denote the contaminated likelihood with contamination rate $p$, written as a mixture of the original likelihood function $f$ and a contamination density $g$,
$$
k_p(\bmx;\theta) := (1-p)f(\bmx;\theta) +pg(\bmx).
$$
We informally state here our main result from Section \ref{sec:main_noncompact}:
\begin{result}
Under mild assumptions, there exists a sequence of contamination probabilities $p_n\rightarrow0$ such that $\forall \epsilon,\delta>0$,
$$
\forall S\text{ measurable, }  \,\,\,\mbb{P}_{\pi_n}(S|\bm{X}) \leq \exp(\epsilon)\mbb{P}_{\pi_n}(S|\bm{Z}) + \delta
$$
fails with probability shrinking exponentially to $0$ as $n\rightarrow\infty$, where the probability is taken with respect to the neighbouring datasets $\bmX$ and $\bmZ$ (differ by only one entry) generated from the true contaminated density $k_{p_n}$.

In addition, for any $\theta$, the Fisher information $I_{p_n}(\theta)$ for the contaminated model converges entrywise to that of the uncontaminated model $I(\theta)$
$$
    I_{p_n}(\theta)\rightarrow I(\theta).
$$
\end{result}
The foregoing result demonstrates that, under mild regularity conditions, as the dataset size $n \to \infty$, sampling from the contaminated posterior achieves perfect differential privacy almost surely. 
Moreover, since the contamination rate satisfies $p_n \to 0$, the contaminated posterior distribution converges asymptotically to the true posterior. 
Consequently, the statistical efficiency of the proposed approach becomes asymptotically equivalent to that of the uncontaminated Bayesian model.

\subsection{Organisation of the Paper}
This paper is organised as follows. 
In Section \ref{sec:dp}, we introduce the setup of differential privacy and the data contamination mechanism, with relevant definitions needed for the main result.

Following that, we present our main results on the privacy level of $p$-contaminated posterior sampling in Section \ref{sec:main_results}.
Intermediate results for the main theorems are presented in Appendix \ref{sec:appx_main}, and the longer proofs appear in the supplementary material \citep{hu2023supp}.
We lift the restrictions of bounded parameter and observation spaces and derive in Section \ref{sec:main_noncompact} a privacy bound that shrinks with $n$, with accompanying discussions of the assumptions.

The link between the $(\epsilon,\delta)$ level and the contamination parameters $p$ and $g(\cdot)$ is implicit.
To help practitioners tune the model, we present, in Section \ref{sec:simulation}, a simulation-based approach to estimate $\epsilon$ (with $\delta$ fixed) for a given model setup, and to examine the privacy level of several regression models as the data size $n$ varies.
To illustrate the performance, we compare our approach with two state-of-the-art frequentist methods on a private mean estimation problem, in Section \ref{sec:pme}.
On the theoretical side, we show in Section \ref{sec:examples} and in the supplementary material (Appendix C) that our assumptions are satisfied for many regression models when $n$ is large enough.
Finally, we conclude the paper with a short discussion in Section \ref{sec:discussion}.


\section{Preliminaries}\label{sec:dp}
In this section, we introduce the basic ideas and definitions needed to construct our framework for establishing differential privacy in Bayesian posterior sampling.

\subsection{Differential Privacy in Bayesian Inference}
In a typical Bayesian inference setting, we have access to a dataset of size $n$, $\mcal{D}\in \mcal{X}^n$, where $\mcal{X}$ is the observation space and the data are modelled by a density $f(\bmx;\theta):\mcal{X}\times \Omega \rightarrow \mbb{R}_{\geq0}$ with parameter $\theta\in\Omega$.
The parameter $\theta$ is a random variable with prior density $\pi_0(\theta)$ on the measurable space $(\Omega, \mcal{B}(\Omega))$.
Then the posterior distribution for $\theta$ given the dataset $\mcal{D}$ can be expressed as
\begin{equation}\label{eq:simple_posterior}
    \pi_n(\theta|\mcal{D}) = \frac{\prod_{i=1}^n f(\bmx_i;\theta) \pi_0(\theta)}{\int_{\Omega} \prod_{i=1}^n f(\bmx_i;\theta) \pi_0(\theta) \dd\theta}.
\end{equation}
Throughout this paper, we assume the following:
\begin{itemize}
    \item The likelihood functions $f(\bmx;\theta)$ and $k_p(\bmx;\theta)$ are jointly continuous in $\bmx$ and $\theta$, and the contamination likelihood $g(\bmx)$ has full support over the observation space;
    \item The observation space $\mcal{X}$ and the parameter space $\Omega$ are convex subsets of Euclidean space, but need not be bounded or of the same dimension.
\end{itemize}
For now, we also assume that the model is weakly consistent at $\theta^*$ (see, e.g., Definition 6.1 of \cite{ghosal2017fundamentals}).
In later sections, this weak consistency assumption will be implied by other assumptions.
\begin{definition}[Neighbouring Datasets]\label{def:neighbour}
    Two datasets $\bm{X}$ and $\bm{Z}$ are neighbours if and only if $\exists \mcal{D}$ and $\bmx,\bmz\notin\mcal{D}$ such that $\bmX=\mcal{D}\cup\{\bmx\}$, $\bmZ=\mcal{D}\cup\{\bmz\}$.
\end{definition}
{\changed
We adapt the following definition from \cite{hall2013random}.
\begin{definition}[$(\epsilon_n,\delta_n,p_n)$-Random Differential Privacy]\label{def:rdp}
The direct sampling from the posterior distribution in (\ref{eq:simple_posterior}) is considered to be $(\epsilon_n,\delta_n,p_n)$-random differentially private if for any two neighbouring datasets $\bm{X},\bm{Z}\in\mcal{X}^n$, with i.i.d. entries generated from $\mbb{P}_{\theta^*}$,
\begin{equation}\label{eq:privacy_def}
    \mbb{P}_{\theta^*}\left( \forall S\in\mcal{B}(\Omega), \, \mbb{P}_{\pi_n}(S|\bm{X}) \leq \exp({\epsilon_n}) \mbb{P}_{\pi_n}(S |\bm{Z}) + \delta_n\right) \geq 1-p_n,
\end{equation}
where $\epsilon_n,\delta_n,p_n\geq0$ are decreasing sequences of constants.
\end{definition}
\begin{remark}
    Setting $p_n=0$ does not imply $(\epsilon_n,\delta_n)$-DP, because that would require the inner inequality to be satisfied uniformly for all possible pairs of datasets.
    Such a condition usually cannot be satisfied uniformly.
    For instance, when fitting a Bayesian posterior using a Gaussian likelihood of variance $1$ and a Gaussian prior $\mcal{N}(\mu_0,1)$ on the mean $\mu$.
    Let $\bm{X},\bm{Z}$ be neighbouring datasets with $\mcal{D}=\bm{X}\cap \bm{Z}$, $\bm{X} = \mcal{D}\cup\{x\}$, $\bm{Z}=\mcal{D}\cup\{z\}$, where $\mcal{D}:=\{x_1,\dots,x_n\}$.
    Then
    \begin{equation*}
        \log\left(\frac{\pi(\mu|\bm{X})}{\pi(\mu|\bm{Z})}\right) = -\frac{1}{2(n+2)}(x^2-z^2)+(x-z)\left(\mu-\frac{1}{n+2}\sum_{i=1}^{n}(x_i+\mu_0)\right),
    \end{equation*}
    which is unbounded in $\mu$, $x$ and $z$.
\end{remark}
}

\subsection{p-Contaminated Data}
Here, we consider the injection of noise before the inference stage, where for any data point $\bm{y}$ in the dataset $\mcal{D}$, we replace it with $\mcal{C}(\bm{y})$ by
\begin{align*}
    \mbb{P}(\mcal{C}(\bm{y}) = \bm{y}) &\,= 1-p, \\
    \mbb{P}(\mcal{C}(\bm{y})= \bmx) &\, = p, 
\end{align*}
where $0<p<1$, $\bmx\sim g(\cdot)$, and $g(\cdot)$ is a density function with a heavier tail than $f(\bmx;\theta),\forall\theta$.
Thus, the likelihood function of the contaminated dataset is given by
$k_p(\bmx;\theta) = (1-p) f(\bmx;\theta) + pg(\bmx)$.

\begin{remark}
We may consider two distinct DP setups with the same contamination strategy.
When we discuss the differential privacy of a randomised algorithm $\mcal{A}$, see (\ref{eq:dp}), it is important to identify what random process is involved in $\mcal{A}$, in this case, whether contamination is part of $\mcal{A}$, i.e., contamination at runtime, or not part of $\mcal{A}$, i.e., contamination before $\mcal{A}$.
In our framework, we adopt the latter configuration due to its simplicity.

Nevertheless, one could alternatively consider the former approach, in which uncontaminated neighbouring datasets $\bmX$ and $\bmZ$ are contaminated at runtime to yield $\mcal{C}(\bmX)$ and $\mcal{C}(\bmZ)$. In this setting, $\mcal{C}(\bmX)$ and $\mcal{C}(\bmZ)$ are random variables whose realisations are not necessarily neighbouring datasets.
It is not hard to verify that if a pre-contaminated sampling algorithm is $(\epsilon,\delta)$-DP, then contamination at runtime will also satisfy $(\epsilon,\delta)$-DP.
However, because $\mcal{C}(\bmX)$ and $\mcal{C}(\bmZ)$ are not necessarily neighbouring datasets when contamination occurs at runtime, the resulting analysis becomes substantially more complicated.
Either way, one should treat the contamination density $g(\cdot)$ and the contamination rate $p$ as public knowledge to avoid having model identifiability issues \citep{mu2023huber}.
\end{remark}
Hereafter, we will assume the dataset to be \textbf{precontaminated} and denote $\bmX$ as the result of the contamination. 
Thus, the hypothetical neighbouring dataset $\bmZ$ obtained by changing one entry of the already contaminated dataset $\bmX$ carries the same contamination as $\bmX$ except possibly for one entry.

\subsection{Posterior Decomposition}
Here we quickly outline our strategy to bound $\epsilon$ and $\delta$. Let 
$$d(\bmx;\theta):= \frac{k_p(\bmx;\theta)}{k_p(\bmx;\theta^*)},$$ 
where $\theta^*$ is the true parameter. 
If $\bmX$ is a dataset with $n$ data points $x_i,i=1,\dots,n$, then for $A\subset\Omega$, denote
$$
m_n(A,\bmX) = \int_A \prod_{i=1}^n d(\bmx_i;\theta) \pi_0(\theta)\dd\theta,
$$ 
where the dependence on $\theta^*$ is omitted, since $\theta^*$ is considered constant in the model.
Now, we can write the posterior distribution function $\pi_n(\cdot|\bmX)$ in terms of $m_n(\cdot,\bmX)$,
\begin{equation} \label{eq:posterior_mass}
    \mbb{P}_{\pi_n}(A|\bmX) = \frac{\int_A \prod_{i=1}^n d(\bmx_i;\theta)\pi_0(\dd \theta)}{\int_\Omega \prod_{i=1}^n d(\bmx_i;\theta)\pi_0(\dd \theta)} = \frac{m_n(A,\bmX)}{m_n(\Omega,\bmX)}.
\end{equation}
Let $A_n:=\{\theta\in\Omega: h(\theta,\theta^*)\leq\phi_n\}$, for a positive sequence $\phi_n\rightarrow0$, where $h$ denotes the Hellinger distance which will be defined in Definition \ref{def:hellinger}.
$A_n$ may be alternatively defined using other distance metrics such that $\mathbb{P}_{\pi_n}(A_n^c|\bm{X})\rightarrow0$.
We will prove our main theorems for Hellinger distance, but use $L_{\infty}$ distance in the simulation.
Then
$$
\mbb{P}_{\pi_n}(S|\bm{X}) = \mbb{P}_{\pi_n}(S\cap A_n|\bm{X}) + \mbb{P}_{\pi_n}(S\cap A_n^{c}|\bm{X}),
$$
where we aim to bound $\mbb{P}_{\pi_n}(S\cap A_n^c|\bmX)$ by $\delta$, so that $\epsilon$ is only required to be upper bounded within a compact set $A_n$.

\begin{proposition}\label{prop:posterior_decomp}
Let $\bmX$ and $\bmZ$ be neighbouring datasets, i.e., $\bm{X}=\mcal{D}\cup\{\bmx\}$, $\bm{Z}=\mcal{D}\cup\{\bmz\}$, $\bmx,\bmz\notin \mcal{D}$.
Let $A_n$ be defined as above, then for any $S\subset\Omega$,
\begin{equation}
    \mbb{P}_{\pi_n}(S|\bm{X})\leq  \eta_n\left[\eta_n + \frac{m_n(A_n^c, \bm{Z})}{m_n(\Omega,\bm{X})}\right] \mbb{P}_{\pi_n}(S|\bm{Z}) + \frac{m_n(A_n^c, \bm{X})}{m_n(\Omega,\bm{X})}, \label{eq:dp_decompose}
\end{equation}
where
\begin{equation}\label{eq:sensitivity_eta}
    \eta_n =\sup_{\bmx,\bmz\in\mcal{X}} \sup_{\theta\in A_n} \frac{d(\bmx;\theta)}{d(\bmz;\theta)}.
\end{equation}
\end{proposition}
If we can find upper bounds for $m_n(A_n^c,\bmX)$, $m_n(A_n^c,\bmZ)$ and $\eta_n$, and lower bound for $m_n(\Omega,\bm{X})$, then we can treat (the upper bounds of) $\eta_n(\eta_n + \tfrac{m_n(A_n^c, \bm{Z})}{m_n(\Omega,\bm{X})})$ as $e^{\epsilon}$ and $\frac{m_n(A_n^c, \bm{X})}{m_n(\Omega,\bm{X})}$ as $\delta$.
\begin{proof}[Proof for Proposition \ref{prop:posterior_decomp}]
    Note that using the notation of (\ref{eq:posterior_mass}),
    $$
        \mbb{P}_{\pi_n}(S\cap A_n^c|\bmX) \leq \mbb{P}_{\pi_n}(A_n^c|\bmX) = \frac{m_n(A_n^c, \bm{X})}{m_n(\Omega,\bm{X})},
    $$
    and
    \begin{align}
        \mbb{P}_{\pi_n}(S\cap A_n|\bmX) \leq &\, \sup_{\bmx,\bmz\in \mcal{X}}\left[\sup_{\theta\in A_n} \frac{d(\bmx;\theta)}{d(\bmz;\theta)} \right]\frac{\int_{A_n\cap S} \prod_{i=1}^n d(\bmz_i,\theta)\pi_0(\dd \theta)}{\int_\Omega \prod_{i=1}^n d(\bmx_i,\theta)\pi_0(\dd \theta)}\nonumber\\
        \leq &\, \sup_{\bmx,\bmz\in \mcal{X}}\left[\sup_{\theta\in A_n} \frac{d(\bmx;\theta)}{d(\bmz;\theta)} \right]\frac{\int_{A_n\cap S} \prod_{i=1}^n d(\bmz_i;\theta)\pi_0(\dd \theta)}{\int_\Omega \prod_{i=1}^n d(\bmx_i;\theta)\pi_0(\dd \theta)}\nonumber\\
        \leq &\, \sup_{\bmx,\bmz\in \mcal{X}}\left[\sup_{\theta\in A_n} \frac{d(\bmx;\theta)}{d(\bmz;\theta)}\right] \frac{\int_\Omega \prod_{i=1}^n d(\bmz_i;\theta)\pi_0(\dd \theta)}{\int_\Omega \prod_{i=1}^n d(\bmx_i;\theta)\pi_0(\dd \theta)} \mbb{P}_{\pi_n}(S\cap A_n|\bmZ)\label{eq:epsilon_split}\\
        \leq &\, \sup_{\bmx,\bmz\in \mcal{X}}\left[\sup_{\theta\in A_n} \frac{d(\bmx;\theta)}{d(\bmz;\theta)}\right] \frac{m_n(A_n,\bmZ)+m_n(A_n^c,\bmZ)}{m_n(\Omega,\bmX)} \mbb{P}_{\pi_n}(S\cap A_n|\bmZ) \nonumber\\
        \leq &\, \eta_n \left(\eta_n \mbb{P}_{\pi_n}(A_n|\bm{X})+ \frac{m_n(A_n^c,\bmZ)}{m_n(\Omega,\bmX)}\right) \mbb{P}_{\pi_n}(s\cap A_n|\bmZ),\nonumber
    \end{align}
    where the last step comes from drawing the supremum of the fraction $\tfrac{d(\bm{z};\theta)}{d(\bm{x};\theta)}$ out of the integral to transform $m_n(A_n,\bmZ)$ into $m_n(A_n,\bmX)$.
\end{proof}

\subsection{Hellinger Distance and Bracketing Entropy}
When the parameter space is non-compact, we need some control over the space complexity of density functions to compute the tail bound for the posterior distribution, which will act as an upper bound for $\delta$.
In empirical processes, covering numbers and bracketing numbers are often used to quantify the size and complexity of a metric space.
\begin{definition}[Hellinger Distance]\label{def:hellinger}
    Let $p,q$ be two density functions defined on a measurable space $(\mbb{R}^{d},\mcal{B}(\mbb{R}^{d}))$, then the squared Hellinger distance between $p$ and $q$ is defined as 
    $$
    h^2(p,q):= \frac{1}{2}\left\|p^{1/2}-q^{1/2}\right\|_2^2 = \frac{1}{2}\int_{\mbb{R}^{d}} \left(\sqrt{p(\bmx)}-\sqrt{q(\bmx)}\right)^2\dd \bmx.
    $$
\end{definition}

Throughout this paper, the \textbf{Hellinger distance} will be used as the \textbf{preferred metric} for computing \textbf{bracketing numbers}, unless specifically stated.
We hereafter omit the inclusion of metric choice in the notations.

When the context is clear, we will use $h(\theta_1,\theta_2)$ to denote the distance between two contaminated density functions $h(k_p(\cdot;\theta_1), k_p(\cdot;\theta_2))$, and replace the function space with the parameter space $\Omega$. 

\vspace{1em}
\noindent
Let $\Xi$ denote the set of non-negative integrable functions on $\mbb{R}^{d}$ and $\mcal{P}\subset\Xi$ be the set of density functions of our interest.
\begin{definition}[$r$-Hellinger Bracketing]
    A set of function pairs $\{[L_j,U_j], j=1,\dots,N\}\subset \Xi\times\Xi$ is a \textbf{$r$-Hellinger bracketing} of the space $\mcal{P}$ if 
    $$
        h(L_j, U_j)\leq r \text{ for all } j\in\{1,\dots,N\},
    $$
    and 
    $$ 
    \forall p\in\mcal{P},\,\, \exists j\in\{1,\dots,N\} \text{ such that }  L_j(\bmx)\leq p(\bmx)\leq U_j(\bmx), \forall \bmx\in \mbb{R}^{d}.
    $$
    An $r$-bracket is a pair $[L_j,U_j]$ with $h(L_j,U_j)\leq r$.
\end{definition}
\begin{remark}
    The Hellinger distance is a metric on density functions, but it also generalises easily to a metric on the space of non-negative integrable functions for the purpose of the above definition.
\end{remark}
\begin{definition}[Hellinger Bracketing Number]
    The \textbf{$r$-Hellinger bracketing number} is the size of the smallest $r$-bracketing, defined as $$N_{[]}(r,\mcal{P}):=\min\{|B| : B \text{ is an $r$-Hellinger bracketing of $\mcal{P}$}\}.$$
    The \textbf{$r$-Hellinger bracketing entropy} is the log of the bracketing number $$H_{[]}(r, \mcal{P}) := \log N_{[]}(r,\mcal{P}).$$
\end{definition}

\section{Main Results} \label{sec:main_results}

This section is devoted to presenting our main results and discussing the required assumptions.

\subsection{Main Assumptions}
First, we define some necessary assumptions.

For $A\subset\Omega$, denote the set of density functions $f(\bmx;\theta)$ indexed by $\theta\in A$ by $\mcal{F}(A):=\{f(\cdot;\theta):\theta\in A\}$.
Similarly, $\mcal{K}_p(A):=\{k_p(\cdot;\theta);\theta\in A\}$.
\begin{assumption}\label{as:bracket_entropy}
     There exists a positive sequence $\phi_n\rightarrow0$ with $n\phi_n^2\rightarrow\infty$, constants $c>0$ with a sequence of expanding subsets $\Omega_1\subset\Omega_2\subset\dots\subseteq \Omega$, $\Omega_n\rightarrow\Omega$, such that
\begin{equation}\label{eq:bracket_entropy}
    \int_{\phi_n^2/2^{10}}^{\phi_n} H^{1/2}_{[]}(u,\mcal{F}(\Omega_n))\dd u \leq c\sqrt{n} \phi_n^2.
\end{equation}
\end{assumption}
Assumption \ref{as:bracket_entropy} measures the space complexity of the density functions $f(\cdot;\theta)$. 
Note that $H_{[]}(u,\mcal{F}(\Omega_n))\leq c^2n\phi_n^2,\forall u\in(\phi_n^2/2^{10},\phi_n)$ implies (\ref{eq:bracket_entropy}), thus this condition can be verified with the following lemma.

\begin{lemma}\label{lemma:cover_bound}
    Let $\Omega_1\subset\Omega_2\subset\dots\subset\Omega$ be a sequence of subsets such that $\Omega_n\rightarrow\Omega$ and $\Omega$ is a convex subset of $\mbb{R}^d$. 
    Let $R_n$ denote the radius of $\Omega_n$, which grows at most polynomially with $n$.
    Suppose that the set of density functions $\mcal{F}(\Omega_n)$ is locally Lipschitz in the sense that
    $$
    \left|\sqrt{f_\theta(\bmx)}-\sqrt{f_\vartheta(\bmx)}\right| \leq M^{(n)}_{\vartheta,u}(\bmx)\|\theta-\vartheta\|_\infty
    $$
    for some function $M^{(n)}_{\vartheta,u}:\mcal{X}\rightarrow\mbb{R}$, $\forall \bmx\in\mcal{X}$, $\theta,\vartheta\in\Omega_n$, $\|\theta-\vartheta\|_\infty \leq u$.
    Let
    $$
        M_{u}^{(n)}:=\sup_{\vartheta\in\Omega_n} \|M^{(n)}_{\vartheta,u}\|_2,
    $$
    which is a constant depending on $R_n$.
    If there exists a constant $q\in(0,1)$ such that $\forall u\in(0,\phi_0)$ for some constant $\phi_0$,
    \begin{equation}\label{eq:loglog-bound}
        \frac{\log\log M_{u}^{(n)}}{\log n} \leq q,
    \end{equation}
    then there exists a positive sequence $\phi_0>\phi_n\rightarrow 0$ with $\phi_n > n^{-\tfrac{1-q}{4}}$ such that (\ref{eq:bracket_entropy}) holds.
\end{lemma}
\begin{remark}
    It is usually enough to consider $\Omega_n$ as the $L_2$ neighbourhood of radius $R_n$ around $\theta^*$, with $R_n\rightarrow\infty$ no faster than polynomial in $n$.
    If $R_n$ grows exponentially in $n$, $(\ref{eq:bracket_entropy})$ might not hold for any $\phi_n$.
    
    Later, in the examples, we show that $M_u^{(n)}$ grows with $u$, so it usually suffices to check (\ref{eq:loglog-bound}) for only $u=\phi_0$.
\end{remark}

\begin{definition}\label{def:neighbourhood}
For $\alpha\in(0,1]$, define
$$
\rho_{p,\alpha}(\theta,\vartheta):=\frac{1}{\alpha}\int\left[\left(\frac{k_p(\bmx;\theta)}{k_p(\bmx;\vartheta)}\right)^{\alpha}-1\right]k_p(\bmx;\theta)\dd \bmx.
$$
For $t>0$, $p\in(0,1)$, define the $\rho_{p,\alpha}$ neighbourhood around $\theta^*$ by 
$$
S_{p,\alpha}(t):=\{\theta : \rho_{p,\alpha}(\theta^*,\theta)\leq t\}.
$$
\end{definition}

\begin{assumption}\label{as:prior_center}
    The positive sequence $\phi_n\rightarrow0$ satisfies
    \begin{equation}\label{eq:prior_centre}
        \mbb{P}_{\pi_0}(S_{0,\alpha}(t_{n}))\geq c_3 e^{-2nt_{n}},
    \end{equation}
    for some constants $\alpha\in(0,1]$, $c_3>0$,
    where $t_n:=c_1\phi_n^2/8$, 
    $
        \tfrac{n t_{n}}{\log n}> \max\left\{\tfrac{1}{\alpha}, \tfrac{c_1}{8c_2}\right\}
    $, $\forall n>0$, and $c_2$ is a constant from Lemma \ref{lemma:wong1995probability}.
\end{assumption}

\begin{assumption}\label{as:prior_tail}
    The positive sequence $\phi_n\rightarrow0$ and $\Omega_n\rightarrow\Omega$ satisfy
    \begin{equation}\label{eq:prior_tail}
        \mbb{P}_{\pi_0}(\Omega_{n}^c)\leq c_4\exp(-2c_1 n\phi_{n}^2),
    \end{equation} 
    for some constant $c_4>0$.
\end{assumption}

Assumptions \ref{as:prior_center} and \ref{as:prior_tail} control the prior concentration around the centre and in the tails, respectively.
The first three assumptions are needed to control the $\delta$ term in (\ref{eq:privacy_def}).
Note that Assumptions \ref{as:bracket_entropy}-\ref{as:prior_tail} feature the same symbol $\phi_n$ for the sequence.
In the proof of the main result, we require the same sequence $\phi_n$ to satisfy these three assumptions simultaneously.

\begin{definition}\label{def:sup_grad}
    Denote the supremum of the directional derivative of $f(\bmx;\theta)$ with respect to $\theta$ at $(\bmx,\theta)$ by $\tilde{\grad}_{\theta} f(\bmx;\theta)$,
    $$\tilde{\grad}_{\theta} f(\bmx;\theta) = \sup_{\|\vec{u}\|_2=1} \lim_{\xi\downarrow0} \frac{f(\bmx;\theta+\xi\vec{u}) - f(\bmx;\theta)}{\xi}.$$ 
\end{definition}

\begin{assumption}\label{as:ratio_bound}
    All the directional derivatives of the density function $f(\bmx;\theta)$ with respect to $\theta$ exist and are bounded, and the contamination density $g$ has tails no lighter than $\tilde{\grad}_\theta f(\bmx;\theta)$ in any direction, i.e., for any fixed $\theta\in\Omega$,
    $$
        \sup_{\bmx\in\mcal{X}} \left\|\frac{\tilde{\grad}_{\theta} f(\bmx;\theta)}{g(\bmx)}\right\|_2 <\infty.
    $$
\end{assumption}
\begin{remark}
    We assume a weaker condition than total differentiability to accommodate the use of absolute value $|\cdot|$ in density functions, e.g., the Laplace distribution.

    When $g$ is continuous and has full support over the support of $\tilde{\grad}_{\theta} f$, the assumption holds trivially whenever $g$ has suitably heavy tails.
\end{remark}

\begin{assumption}\label{as:identifiability}
For every $\psi>0$, 
\begin{equation}\label{cond:inf_hellinger}
    \inf_{\|\theta-\theta^*\|_2>\psi} h(f(\cdot;\theta),f(\cdot;\theta^*)) > 0.
\end{equation}
\end{assumption}
\begin{remark}
    Assumption \ref{as:identifiability} is similar to one of the conditions in Theorem 9.13 of \cite{wasserman2004all}, but stronger, since we stated it in terms of Hellinger distance rather than KL divergence. The condition directly implies the identifiability of the model. See also Condition 2 in the lecture notes \cite{wasserman2020intermediate}.
\end{remark}

The following Proposition helps quantify the rate of convergence for $\epsilon_n$ within a Hellinger neighbourhood of $\theta^*$ by employing Cauchy's mean value theorem on the likelihood ratio $d(\bmx;\theta)$.
\begin{proposition}\label{prop:hellinger_convergence}
If Assumption \ref{as:identifiability} holds, then for any sequence $\phi_n\rightarrow0$, and any decreasing sequence $1> p_n\rightarrow 0$, there exists a sequence $\psi_n\rightarrow0$ such that
\begin{equation}\label{cond:hellinger}
    h(k_{p_n}(\cdot;\theta),k_{p_n}(\cdot;\theta^*))\leq \phi_n \implies \|\theta-\theta^*\|_2\leq \psi_n. 
\end{equation}
\end{proposition}

\subsection{DP under Bayesian Sampling} \label{sec:main_noncompact}
\begin{theorem} \label{thm:e-d_privacy}
Suppose that there exists a decreasing sequence $\phi_n>0$ that satisfies the conditions in Assumptions \ref{as:bracket_entropy}-\ref{as:prior_tail}, and suppose also that Assumptions \ref{as:ratio_bound}, \ref{as:identifiability} hold.
For a fixed $p_0\in(0,1)$, there exist positive constants (or a sequence of constants) $\{T_n\},c_1,c_2,\dots,C_1,C_2,\dots<\infty$ such that for any sequence $p_0 \geq p_n \downarrow 0$,
\begin{align}\label{eq:dp_main}
    \mbb{P}^*_{\theta^*}\Bigl(\forall S\in\mcal{B}(\Omega),&\,\mbb{P}_{\pi_n}(S|\bm{X}) \leq  e^{\epsilon_n}\mbb{P}_{\pi_n}(S|\bm{Z}) + \delta_n \Bigr) \\
     \geq &\,\begin{dcases}
    1- C_1e^{-nc_5 \phi^2_n}, & \text{for } nc_2\phi_n^2\geq \log 2,\\
    1- (C_2+C_3(\log \phi_n)^2)e^{-nc_5 \phi^2_n}, & \text{otherwise}
\end{dcases} \nonumber
\end{align}
with
$$\epsilon_n:= \varepsilon_n + \log(\delta_n+\exp(\varepsilon_n)),\quad \varepsilon_n := 2(1-p_n)T_n\psi_n/p_n,\quad \delta_n:= \tfrac{4}{c_3}e^{-\frac{1}{2}c_1 n\phi_n^2},$$ 
where $\mbb{P}^*_{\theta^*}$ is the data generating (outer) measure for the pair of neighbouring datasets of size $n$, $\bmX$ and $\bmZ$, with entries generated from the true likelihood $k_{p_n}(\cdot;\theta^*)$. 
In addition, $t_n = c_1\phi_n^2/8$, $\psi_n$ is defined in Proposition \ref{prop:hellinger_convergence}, and $c_5 =  \min\{c_1\alpha/8,c_2\}$.
In other words,
$$
\forall S\in\mcal{B}(\Omega),\,\,\mbb{P}_{\pi_n}(S|\bm{X}) \leq \exp(\epsilon_n)\mbb{P}_{\pi_n}(S|\bm{Z}) + \delta_n
$$
holds with probability tending to one as $n\rightarrow\infty$.
\end{theorem}
We define in Definition \ref{def:rdp} the type of DP proved here, which is adapted from \cite{hall2013random}, but in our case, the probability of failing $(\epsilon_n,\delta_n)$-DP diminishes exponentially with $n$.
The main conclusion of Theorem \ref{thm:e-d_privacy} is that as $n\rightarrow\infty$, we may choose a sequence $p_n\rightarrow0$ such that $\epsilon_n,\delta_n\rightarrow0$ and for $n$ large enough the failure probability falls into the first case of (\ref{eq:dp_main}), i.e., $<C_1\exp(-nc_5\phi_n^2)\rightarrow0$.

\begin{proof}[Sketch Proof]
    For simplicity, we will write $\mbb{P}$ instead of $\mbb{P}^*_{\theta^*}$ when the context is clear.
    Proposition \ref{prop:posterior_decomp} provides insight into how to upper bound the $\epsilon_n$ and $\delta_n$. 
    Recall that
    \begin{equation*}
    \mbb{P}_{\pi_n}(S|\bm{X})\leq  \eta_n\left[\eta_n + \frac{m_n(A_n^c, \bm{Z})}{m_n(\Omega,\bm{X})}\right] \mbb{P}_{\pi_n}(S|\bm{Z}) + \frac{m_n(A_n^c, \bm{X})}{m_n(\Omega,\bm{X})},
    \end{equation*}
    where
    \begin{equation*}
        \eta_n =\sup_{\bmx,\bmz\in\mcal{X}} \sup_{\theta\in A_n} \frac{d(\bmx;\theta)}{d(\bmz;\theta)}.
    \end{equation*}
    Under Assumptions \ref{as:ratio_bound} and \ref{as:identifiability}, we can show that $\eta_n$ is always finite with $\eta_n\rightarrow1$.
    We can also characterise its convergence rate (see Lemma \ref{lemma:ratio_bound}).

    Then we prove the following:
    \begin{itemize}
        \item In Lemma \ref{lemma:denominator_bound}, we prove under Assumption \ref{as:prior_center} that
        $$
        \mbb{P}\left(m_n(\Omega,\bmX)\leq \frac{c_3}{2}e^{-4nt_n}\right)\leq 2e^{-2n\alpha t_n}.
        $$
        \item By Lemma \ref{lemma:wong1995probability}, Assumption \ref{as:bracket_entropy} implies that 
        $$
        \mbb{P}\left(m_n(A_n^c\cap \Omega_n, \bmX)\geq e^{-c_1n\phi_n^2}\right) \leq C^{**}\exp(-c_2n\phi_n^2),
        $$ 
        for some constant $C^{**}$ we specify in the lemma.
        The proof is provided in the supplementary material. 
        \item Then Lemma \ref{lemma:numerator_bound} shows that 
        $$
        \mbb{P}\left(m_n(A_n^c\cap \Omega_n^c, \bmX) \geq e^{-c_1n\phi_n^2}\right) \leq c_4e^{-c_1n\phi_n^2},
        $$ and consequently 
        $$
        \mbb{P}\left(m_n(A_n^c, \bmX)\geq 2e^{-c_1n\phi_n^2}\right) \leq C^{**}e^{-c_2n\phi_n^2}+c_4e^{-nc_1\phi_n^2}.
        $$
        \item The relevant proofs are contained in the Appendix \ref{sec:appx_main}, apart from the proof of Lemma \ref{lemma:wong1995probability} in the supplement.
    \end{itemize}
    Since we assume that $\bmZ$ also follows the same data generating measure, $m_n(A_n^c,\bmZ)$ shares the same upper bound as $m_n(A_n^c, \bmX)$.
    Now, every term in the inequality has a probabilistic upper bound; combining these bounds completes the proof.
\end{proof}
Lemmas \ref{lemma:entropy_order}, \ref{lemma:divergence_order}, and Proposition \ref{prop:hellinger_convergence} are intermediate results such that whenever the Assumptions \ref{as:bracket_entropy} and \ref{as:prior_center} hold with respect to uncontaminated likelihood $f(\cdot;\theta)$, the same assumptions also hold for the actual contaminated likelihood $k_p(\cdot;\theta)$, and in practice, one only needs to check the assumptions for the case $p=0$.

\begin{theorem}[Fisher Information] \label{thm:fisher_information}
Let $p_n\rightarrow0$ be a sequence of positive real numbers.
Let $I_{p_n}(\theta)$ denote the Fisher information matrix for a single data point with respect to the density $k_{p_n}(\cdot;\theta)$ and $I(\theta)$ denote the Fisher information matrix with respect to $f(\cdot;\theta)$.
Then 
$$
 I_{p_n}(\theta^*)\rightarrow I(\theta^*).
$$
\end{theorem}
\begin{proof}
The proof is in Appendix B in the supplement.
\end{proof}
Under Assumptions \ref{as:bracket_entropy}-\ref{as:identifiability}, Theorem \ref{thm:e-d_privacy} holds for any decreasing $p_n\rightarrow0$, $p_n\leq p_0<1$, especially for $n$ large enough,
$$
\mbb{P}\Bigl(\mbb{P}_{\pi_n}(S|\bm{X}) \leq \exp(\epsilon_n)\mbb{P}_{\pi_n}(S|\bm{Z}) + \delta_n \Bigr) \geq 1- Ce^{-nc_5 \phi_n^2}, 
$$
where
$$\epsilon_n:= \varepsilon_n + \log(\delta_n+\exp(\varepsilon_n)),\quad \varepsilon_n := 2(1-p_n)T_n\psi_n/p_n,\quad \delta_n:= \tfrac{4}{c_3}e^{-\frac{1}{2}c_1 n\phi_n^2},$$ 
By picking $p_n\rightarrow0$ such that $\psi_n/p_n\rightarrow0$ as $n\rightarrow\infty$, both $\delta_n$ and $\epsilon_n$ converge to zero.
Asymptotically, the algorithm achieves differential privacy for arbitrarily small $(\epsilon,\delta)$.
Finally, since $p_n\rightarrow 0$, the convergence in the Fisher information follows from Theorem \ref{thm:fisher_information}.

Our results address the case where samples are drawn exactly from the posterior distribution.
For Markov Chain Monte Carlo algorithms \citep{robert1999monte} that do not produce i.i.d. samples, our results indicate the differential privacy level upon reaching stationarity.
Furthermore, \cite{minami2016differential} can be used to derive a tight bound on the privacy loss incurred when generating an approximate sample, by bounding the total variation distance between the distribution of the approximate sample and the target posterior distribution. 
There is a growing body of literature on the convergence rates of MCMC algorithms, e.g., \cite{durmus2017nonasymptotic}, \cite{andrieu2022comparison}, \cite{andrieu2023weak}, which can be leveraged to complement our analysis.

\subsection{Remarks on Assumptions}

\begin{remark}
Lemmas \ref{lemma:wong1995probability} and \ref{lemma:numerator_bound} jointly provide a probabilistic bound for the posterior mass on $A_n^c$. 
This result could alternatively be established via Theorem 8.11 (or Theorem 8.9) of \cite{ghosal2017fundamentals}, under a similar set of assumptions.
We adopt the proof techniques from \cite{shen1994convergence,wong1995probability,shen2001rates} to obtain a more concise representation of the bounds along with exponentially shrinking failure probability, in contrast to polynomial decay obtained from the results of \cite{ghosal2017fundamentals}.
The assumptions used by \cite{ghosal2017fundamentals} are, however, arguably less restrictive with respect to the prior and size of $n$, which we discuss in more detail in Appendix D.1.
\end{remark}

\begin{remark}
One sufficient condition for Assumption \ref{as:bracket_entropy} to hold is the following:
$$
H_{[]}(u,\mcal{F}(\Omega_n))\leq c^2n\phi_n^2, \,\,\, \forall u\in(\phi_n^2/2^{10},\phi_n).
$$
By Lemma \ref{lemma:cover_bound}, the above holds for most families of likelihood functions $f(\cdot;\theta)$ when $n$ is large enough.
However, due to the constant $c$ being quite small in magnitude, see Lemma B.2 in the supplement, one typically requires unrealistically large $n$ for Theorem \ref{thm:e-d_privacy} to hold.
Despite having derived its convergence rate in a non-asymptotic fashion, the result is better deemed as an asymptotic one.
We will show in the next section using empirical simulations that Theorem \ref{thm:e-d_privacy} provides insight into how $\epsilon_n$ and $\delta_n$ converge.
\end{remark}

\noindent 
Both Assumptions \ref{as:prior_center} and \ref{as:prior_tail} are conditions on the prior distribution.
Let $\bar{S}(r):=\{\theta\in\Omega:\|\theta^*-\theta\|_2\leq r\}$ denote the $L_2$-ball of radius $r$ around $\theta^*$.
One can show that under mild conditions, the set $S_{0,\alpha}(t_n)\cap \bar{S}(c_r)$ contains an $L_2$-ball with radius proportional to $t_n$ for some constant $c_r<\infty$.

\begin{lemma}\label{lemma:local_set}
    Let $S_{0,\alpha}(t_n):=\{\theta:\rho_{0,\alpha}(\theta^*,\theta)\leq t_n\}$ defined as in Definition \ref{def:neighbourhood}.
    Suppose that there exist $\alpha\in(0,1]$, and $c_r<\infty$ such that 
    $$
        \max_{\theta,\tilde{\theta}\in \bar{S}(c_r)}\int \left[\frac{f(\bmx;\theta^*)}{f(\bmx;\theta)}\right]^{\alpha+1} \|\tilde{\bm{\grad}}_{\theta} f(\bmx;\tilde{\theta})\|_2\dd x < M_{c_r},
    $$
    where $M_{c_r}$ is a constant depending on $c_r$, and $\tilde{\bm{\grad}}$ is the supremum directional derivative defined in Definition \ref{def:sup_grad}.
    Then $\bar{S}(t_n/M_{c_r})\subseteq S_{0,\alpha}(t_n)$.
\end{lemma}

\begin{remark}
The condition holds for most likelihood functions, since $\|\tilde{\bm{\grad}}_\theta f(\bmx;\theta)\|_2$ is typically integrable.
When $\theta$ does not control the dispersion, $f(\bmx;\theta^*)$ and $f(\bmx;\theta)$ have matching rates of decay and hence the expression is integrable.
Otherwise, one can typically choose $c_r$ small enough such that the expression remains integrable.
\end{remark}

\begin{remark}
One may use a similar argument to show that the Hellinger distance and the $L_2$ distance are locally related (see Appendix C), allowing $\phi_n$ and $\psi_n$ to be chosen to decay at the same rate.
Then, the rate of $p_n$ can be chosen freely as long as $\psi_n/p_n\rightarrow0$.
\end{remark}

Note that we may choose $\Omega_n$ in Assumption \ref{as:prior_tail} as an expanding $L_2$-ball. 
Then both Assumptions \ref{as:prior_center}, \ref{as:prior_tail} reduce to computing the prior mass on the $L_2$-ball $\bar{S}(r)$, i.e., $\mbb{P}_{\pi_0}(\bar{S}(r))$ for some (sequence of) $r$.
In general, Assumption \ref{as:prior_center} should hold whenever there is sufficient mass around the true parameter $\theta^*$, i.e., the prior's support covers $\theta^*$, and Assumption \ref{as:prior_tail} should hold whenever the prior has an exponentially decaying tail.

When the prior has a polynomial tail, Assumption \ref{as:prior_tail} requires the set $\Omega_n$ to expand at least exponentially fast with $n$.
For instance, the horseshoe prior readily satisfies Assumption \ref{as:prior_tail} when $\Omega_n$ expands with radius $e^n$.
Following the proof of Lemma \ref{lemma:cover_bound}, one can derive a similar result for $\Omega_n$ with radius growing at most $O(e^n)$ fast, which still holds asymptotically for a wide range of common likelihood functions, e.g., Gaussian.

\begin{lemma}\label{lemma:gaussian_prior_mass}
    Let $\bar{S}(r)$ be defined as above, and let $\pi_0(\cdot)$ be a standard Gaussian prior.
    We have two separate lower bounds for the prior mass
    $$
    \mbb{P}_{\pi_0}(\bar{S}(r)) \geq \begin{dcases}
        \frac{\exp(-\lambda/2)}{\Gamma(d/2)} \gamma(d/2, r^2/2), & \text{ for } r \text{ small,}\\
        \frac{1}{\Gamma(d/2)}\gamma(d/2, R^2/2), &  \text{ for } r>\|\theta^*\|_2.
    \end{dcases}
    $$
    where $\lambda=\|\theta^*\|_2^2$, $R=r-\|\theta^*\|_2$, and
    $\gamma(\alpha,r) = \int_0^r u^{\alpha-1}\exp(-u)\dd u$
    is the incomplete gamma function.
\end{lemma}
The above lemma adapts trivially to Gaussian priors with different mean and covariance structures, or to spike and slab priors with Gaussian slabs.
\begin{remark}\label{rmk:gaussian_prior_mass}
    For $r$ small, we use the first inequality to verify Assumption \ref{as:prior_center}
    \begin{align*}
        \frac{\gamma(\alpha,r)}{\Gamma(\alpha)} =\,\,& \frac{1}{\Gamma(\alpha)} \int_0^r w^{\alpha-1}e^{-w}\dd w, \\
        \geq\,\, & \frac{e^{-r}}{\Gamma(\alpha)} \int_0^r w^{\alpha-1}\dd w \geq \frac{1}{2\pi} e^{\alpha}\alpha^{-\alpha-\frac{1}{2}} r^{\alpha}e^{-r},
    \end{align*}
    where the last step is due to Stirling's approximation (see, for instance, (3.9) page 24 of \cite{artin1964gamma}).
    Thus $\mbb{P}_{\pi_0}(\bar{S}(r))$ shrinks polynomially with $r$, which is always slower than the exponential shrinkage in Assumption \ref{as:prior_center} for sufficiently large $n$.
\end{remark}
\begin{remark}\label{rmk:gaussian_prior_tail}
    For $r$ large, to verify Assumption \ref{as:prior_tail}, we use the second inequality with the approximation $\Gamma(\alpha)-\gamma(\alpha,{r})\sim r^{\alpha-1}e^{-r}$, 
    and hence,
    $$
    1-\mbb{P}_{\pi_0}(\bar{S}(r)) \leq \frac{\Gamma(d/2)-\gamma(d/2, R^2/2)}{\Gamma(d/2)} \sim \frac{1}{\Gamma(d/2)} r^{d/2}\exp(-(r-\|\theta^*\|_2)^2/2),
    $$
    where $\sim$ indicates the leading order term.
\end{remark}

\section{Empirical Assessment of Differential Privacy} \label{sec:simulation}
It is nontrivial to directly link the $(\epsilon,\delta)$ of the model with the amount of contamination being applied.
In other words, it can be hard for practitioners to directly derive the amount of contamination needed to achieve a certain degree of $(\epsilon,\delta)$ by consulting Theorem \ref{thm:e-d_privacy} alone.
This section is devoted to providing guidance on empirically estimating $(\epsilon,\delta)$ for a given model.
The contamination level can then be adjusted based on the empirical estimation, e.g., increase the contamination rate $p_n$ if $\epsilon$ is not low enough.

We begin by introducing our pipeline for estimating DP, followed by some empirical simulations for a few settings of common regression models.

\subsection{Estimation Procedure}
Estimating $\epsilon$ and $\delta$ mostly follows from the result of Proposition \ref{prop:posterior_decomp}.
Recall (\ref{eq:epsilon_split}) from Proposition \ref{prop:posterior_decomp},
$$
\mbb{P}_{\pi_n}(S\cap A_n|\bmX) \leq \sup_{\bmx,\bmz\in \mcal{X}}\left[\sup_{\theta\in A_n} \frac{d(\bmx;\theta)}{d(\bmz;\theta)}\right] \frac{m_n(\Omega,\bmZ)}{m_n(\Omega,\bmX)} \mbb{P}_{\pi_n}(S \cap A_n|\bmZ),
$$
where $m_n(\Omega,\bmX)=\int_{\Omega}\prod_{i=1}^n d(\bmx_i;\theta)\pi_0(\theta)\dd\theta$, and
$$
    d(\bmx;\theta):= \frac{k_{p_n}(\bmx;\theta)}{k_{p_n}(\bmx;\theta^*)}=\frac{(1-p_n)f(\bmx;\theta)+p_ng(\bmx)}{(1-p_n)f(\bmx;\theta^*)+p_ng(\bmx)}.
$$
Conditioned on a dataset $\mcal{D}$ of size $n-1$, we may compute the empirical $\hat{\epsilon}_n$ associated with $\mcal{D}$ by solving
$$
\exp(\hat{\epsilon}_n) \geq \sup_{\bmx,\bmz\in \mcal{X}}\left[\sup_{\theta\in A_n} \frac{d(\bmx;\theta)}{d(\bmz;\theta)}\right] \frac{m_n(\Omega,\mcal{D}\cup\{\bmz\})}{m_n(\Omega,\mcal{D}\cup\{\bmx\})} = \sup_{\bmx,\bmz\in \mcal{X}}\left[\sup_{\theta\in A_n}\frac{d(\bmx;\theta)}{d(\bmz;\theta)} \right]\frac{\mathbb{E}_{\pi_{n-1}}[d(\bmz;\theta)]}{\mathbb{E}_{\pi_{n-1}}[d(\bmx;\theta)]},
$$
where $\pi_{n-1}$ is the posterior distribution on $\mcal{D}$.
The first fraction can be dealt with separately for $\bmx$ and $\bmz$, by solving
\begin{equation}\label{eq:eta_opt}
    \sup_{\bmx\in \mcal{X}}\sup_{\theta\in A_n} d(\bmx;\theta) \text{ and } \inf_{\bmz\in\mcal{X}}\inf_{\theta\in A_n} d(\bmz;\theta).
\end{equation}
Since $\pi_{n-1}$ is usually only available up to a constant, we may approximate the expectations $\mathbb{E}_{\pi_{n-1}}$ through importance sampling and solve the approximated optimization problems
\begin{equation}\label{eq:algorithm_epsilon_constant}
    \sup_{\bmx} \sum_{i=1}^m w_id(\bmx;\theta_i), \quad \inf_{\bmz} \sum_{i=1}^m w_id(\bmz;\theta_i),
\end{equation}
where $(\theta_i,w_i)$ are weighted samples for $\pi_{n-1}(\cdot|\mcal{D})$ with $\theta_i$ drawn from some proposal distribution, and $w_i$ being the importance weight associated with $\theta_i$.
For instance, the Laplace approximation of the posterior distribution $\mcal{N}(\theta_{\text{MAP}}, \bm{H}^{-1})$ may be used as the proposal distribution by computing the MAP estimator and the Hessian $\bm{H}$ of $-\log(\pi_{n-1})$ (see Section 4.6.8.2, pg. 152-153 of \cite{murphy2022probabilistic}).
\begin{remark}
There are a few places where the use of $\theta^*$ is preferred.
An estimate $\hat{\theta}^*$ would suffice when $\theta^*$ is unknown. 
\end{remark}
For $\delta$, we will estimate $\mathbb{P}_{\pi_n}(A_n^c|\bmX)$ as an upper bound for $\delta_n$.
Again, an analytic solution is usually unavailable, so a numerical scheme is required.
Here we recommend using the Laplace approximation again to implement an importance sampling scheme by sampling $\theta_i\sim \mcal{N}(\theta_{\text{MAP}}, \bm{H}^{-1})$ and computing importance weight $w_i:=\pi_n(\theta_i|\bm{X})/f_{\mcal{N}}(\theta_i;\theta_{\text{MAP}},\bm{H}^{-1})$, where $f_{\mcal{N}}$ is the Gaussian density function.
We may estimate $\hat{\delta}_n$ by summing the weights $w_i$ of the importance samples $\theta_i\in A_n^c$,
\begin{equation}\label{eq:algorithm_delta}
    \hat{\delta}_n:= \sum_{\theta_i\in A_n^c} w_i.
\end{equation}

In the proof, we specified the size of the neighbourhood through the sequence $\phi_n$.
However, this may pose difficulty in controlling $\delta$.
For empirical estimation, we instead fix the maximum size of $\delta$ and choose $\phi_n$ to be the minimal value for which $\hat{\delta}_n<\delta$ in (\ref{eq:algorithm_delta}). 
Through $\phi_n$, we construct the set $A_n$ to be a neighbourhood around $\theta^*$, e.g., $L_\infty$-ball around $\theta^*$. 
In the simulations, we choose the contamination rate $p_n$ through
\begin{equation}\label{eq:rate_constants}
    p_n:=n^{1/8}.
\end{equation}
Now, $\epsilon_n$ may be estimated by Algorithm \ref{alg:empirical_estimation_full}.
\begin{remark}
    In order for Assumption \ref{as:ratio_bound} to be satisfied for regression settings, one needs to impose a constraint on the size of covariates.
    In theory, we may allow the size of $\bm{W}$ to grow with $n$; for simplicity, we assume in all the simulations that $\|\bm{W}\|_\infty \leq 1$.
\end{remark}

\begin{algorithm}[t]
\caption{An empirical procedure to estimate $\epsilon$ given $\delta$ and model setup.}
\label{alg:empirical_estimation_full}
 \SetAlgoLined
 \SetKwInOut{Input}{input}
 \SetKwInOut{Output}{output}
\Input{Data set size $n$; Contamination rate $p_n$; $\delta_n>0$; Likelihood function $f$; Contamination density $g$; prior $\pi_0$;  Number of repeats $K$; Output quantile $q$; Number of particles $m$; True parameter $\theta^*$;}
\For{$k=1,\dots,K$}
{
    \uIf{No dataset is present}
    {
        Generate true observations $\bm{X}:=\{\bmx_1,\dots,\bmx_n\}$ according to $\bmx_i\sim f(\cdot;\theta^*)$\;
    }
    Contaminate the observations $\bmX$, denoted $\tilde{\bmX}$, by replacing each $\bmx_i$ with probability $p_n$ by a draw from $g(\cdot)$\;
    Compute the Laplace approximation of the contaminated posterior distribution $\pi_n(\cdot;\tilde{\bmX})$ \;
    Sample $\theta_1,\dots,\theta_m$ from the Laplace approximation\label{step:laplace}\;
    Compute normalized importance weights $w_1,\dots,w_m$ with respect to $\pi_n(\cdot;\tilde{\bmX})$\;
    Choose the minimal $\phi_n$ such that the total weight of particles outside $A_n:=B_{\phi_n}(\hat{\theta}^*, \infty)$ is less than $\delta_n$\, where $A_n$ is the $L_{\infty}$-ball around $\hat{\theta}^*$, i.e., a hypercube of width $2\phi_n$.\;
Solve 
$$
    \eta_n := \sup_{\bmx,\bmz\in \mcal{X}}\sup_{\theta\in A_n} \frac{d(\bmx;\theta)}{d(\bmz;\theta)}
$$
by solving (\ref{eq:eta_opt}) with respect to $A_n, p_n$\;
Set $\mcal{D}$ to be $\tilde{\bmX}$ with the last element discarded\;
Compute the normalized weights $\tilde{w}_i$ {\changed for each $\theta_i$ (from Step \ref{step:laplace}) } with respect to $\pi_{n-1}(\cdot;\mcal{D})$\label{step:w_tilde}\;
Solve for $$
\alpha := \sup_{\bmx} \sum_{i=1}^m \tilde{w}_id(\bmx;\theta_i), \quad \beta:= \inf_{\bmz} \sum_{i=1}^m \tilde{w}_id(\bmx;\theta_i)
$$ using $\theta_1,\dots,\theta_m$ and $\tilde{w}_1,\dots,\tilde{w}_m$\;
Set $\hat{\epsilon}_n^{(k)}:= \log\eta_n+\log\alpha - \log\beta$\;
}
Sort $\epsilon_n^{(k)}$ in ascending order and output the $q$-th percentile.\;
\Output{Estimated $\hat{\epsilon}_n$ with $(100-q)/100$ probability of failing ($\epsilon,\delta$)-DP.}
\end{algorithm}

\subsection{Posterior Sampling for Private Mean Estimation}\label{sec:pme}
In this section, we consider the private mean estimation problem, a setting that has been extensively investigated in the differential privacy literature, see, e.g., \cite{biswas2020coinpress,covington2021unbiased,huang2021instance,duchi2023fast}, among others. 

Let $\bm{X}:=\{x_1,\dots,x_n\}$ denote a dataset of independent and identically distributed observations, where $x_i\sim\mcal{N}_{[l,u]}(\theta^*,\sigma^2)$ are drawn from a truncated normal distribution supported on the interval $[l,u]$ with standard deviation $\sigma=8$.
In the simulation study, the true mean is set to $\theta^* = 30$, and the objective is to estimate $\theta^*$ under differential privacy constraints.

This subsection compares the performance of the following methods under an equivalent privacy budget:
\begin{enumerate}
    \item Bayesian: draw $\hat{\theta}_{B}\sim \pi_n(\cdot|\bmX)$;
    \item CoinPress algorithm \citep{biswas2020coinpress};
    \item Private Quantile Estimation + Clipped Mean Estimator \citep{huang2021instance};
\end{enumerate}
For brevity, we omit the naive approach of injecting Laplace or Gaussian noise directly into the sample mean, as such methods perform significantly worse—by several orders of magnitude—under a conservative choice of interval size ($u-l = 600$).

For the Bayesian contamination model, we consider a contamination rate of $p_n:=n^{-1/8}$ where $n$ is the sample size. 
The dataset $\bm{X}$ is contaminated by replacing each observation with probability $p_n$ by a draw from a scaled-and-translated Student T distribution $\tilde{x}_i\sim T_{\nu}(\theta^*, \sigma)$ truncated to $[l,u]$.
The density of $T_{\nu}(\theta^*,\sigma)$ is given by
$$
    g(\tilde{x}_i)\propto \left(1+\frac{(\tilde{x}_i-\theta^*)^2}{\nu\sigma^2}\right)^{-\tfrac{\nu+1}{2}}\mathbb{I}_{l\leq \tilde{x}_i\leq u}, \quad \nu=5.
$$
We employ a relatively diffuse prior $\pi_0(\cdot) \sim \mathcal{N}(40, 40^2)$ for $\theta$ and estimate $\hat{\epsilon}_n$ with fixed $\delta_n = \tfrac{1}{10n}$ using Algorithm \ref{alg:empirical_estimation_full}.
The estimated $\hat{\epsilon}_n$ values for varying $n$ are presented in Table \ref{tab:epsilon}. 
Sampling from the posterior is done through an ordinary random-walk Metropolis algorithm.
We repeat the estimation for $1000$ simulated datasets, use the $99\%$ quantile of the resulting $\hat{\epsilon}_n$, and consider the Bayesian contamination mechanism to be $(\epsilon_n,\delta_n)$-DP with probability at least $99\%$.

{\changed 
The estimated $\hat{\epsilon}_n$ from the Bayesian approach is used as the privacy budget for the other two algorithms.
}
However, both the CoinPress algorithm and the Clipped Mean Estimator provide privacy guarantees in zero-Concentrated DP (zCDP) \citep{dwork2016concentrated}. 
{\changed 
The values $\hat{\epsilon}_n$ from Table \ref{tab:epsilon} and $\delta_n=\tfrac{1}{10n}$ are used to solve for an appropriate $\rho_n$ budget for $\rho$-zCDP, by using the result that $\tfrac{\epsilon^2}{2}$-zCDP implies $(\epsilon\sqrt{\log(1/\delta)},\delta)$-DP, $\forall\delta>0$, \citep{bun2016concentrated}.
Although this translation from $(\hat{\epsilon}_n,\delta_n)$-DP to $\rho_n$-zCDP imposes a stricter privacy restriction on frequentist approaches, it is a natural choice, e.g., the same translation would be adopted when comparing with the Gaussian mechanism under the $(\epsilon,\delta)$-DP framework.
}
\begin{table}[t]
    \caption{Estimated $\hat{\epsilon}_n$ values for different sample sizes $n$.}
    \label{tab:epsilon}
    \centering
    \begin{tabular}{crrrrrrr}
    \toprule
    $n$ & 100 & 1,000 & 2,000 & 5,000 & 10,000 & 20,000 & 50,000 \\
    \midrule
    $\hat{\epsilon}_n$ & 2.85 & 0.94 & 0.72 & 0.46 & 0.32 & 0.26 & 0.18 \\
    \bottomrule
    \end{tabular}
\end{table}

To evaluate performance, we simulate a fresh dataset $\bm{X}$ and apply each of the three mechanisms to obtain estimates of $\theta^*$. 
For each dataset, the private mean estimation is repeated $1000$ times to compute the bias, variance, and mean squared error (MSE) of each estimator. 
This process is then repeated over $1000$ datasets to generate the box plots shown in Figure \ref{fig:MSE_comparison_ct10}.

Figure \ref{fig:MSE_comparison_ct10} includes two versions of the CoinPress algorithm, which differ only by the number of iterations they take, $3$ and $10$, respectively. 
The same privacy budget is assigned to both versions.
Intuitively, we see that when the number of iterations is not large enough, the performance of the estimator ceases to improve with $n$.
Overall, a suitably tuned CoinPress algorithm outperforms other alternatives; however, the Bayesian approach achieves comparable performance to the Clipped Mean Estimator for $n \geq 1000$, with only a modest performance gap relative to CoinPress.
It is worth emphasising that a posterior sample inherently contains richer information than a mean estimate.
Also, we note that the variance of the MSE for the frequentist estimators is, in general, larger than that of the Bayesian estimator, which means that although the minimum MSE of the Bayesian estimator might not be as good as the frequentist methods, the performance is more consistent.

\begin{remark}
While the three methods are evaluated on the same problem, their applicability differs.
The CoinPress algorithm and the Clipped Mean algorithm are specifically tailored for private mean estimation, with the former assuming Gaussian-distributed data and the latter being model-free.
Both essentially circumvent the need for a user-defined clipping bound, but this feature might not generalise to other problems.
In contrast, the Bayesian approach naturally extends to a wider range of estimation tasks without altering the overall framework.

A potential drawback of introducing privacy through contamination is the absence of a straightforward, non-probabilistic relationship between the contamination rate and the resulting privacy level.
By contrast, frequentist approaches, or more generally, methods that inject Gaussian or Laplace noise into a quantity with bounded sensitivity, offer an explicit correspondence between the privacy budget and the level of perturbation.
\end{remark}

\subsection{Simulation Setups and Results}\label{sec:simulation_result}

Let $\bmX=\{X_1,\dots,X_n\}\in\mathbb{R}^n$ be the set of observations and $\bmW\in[-1,1]^{n\times d}$ be the covariate matrix associated with $\bmX$, i.e., the $i$-th row of $\bmW$ is associated with the $i$-th observation $X_i$.
By default, we assume the first column of $\bmW$ is a vector of $1$.
Let $\theta\in\mathbb{R}^d$ denote the parameter, and the model likelihood is given by a density function $f(X_i;\theta,\bmw_i)=f(X_i,\theta^\top \bmw_i)$, such that $\int_{\mathbb{R}}f(x;\theta,\bmw)\dd x=1$ for any pair of $\theta,\bmw\in\mathbb{R}^d$.
Let $\theta^*$ denote the true parameter value, which in practice may be replaced by an estimate $\hat{\theta}$, e.g., using the maximum a posteriori estimator.

For the cases where contamination is required, we choose a contamination density $g(x)$ and replace each $X_i$ with probability $p_n$ by a random sample from $g(\cdot)$.
Note that the contamination density is allowed to depend on the covariates $\bmw$, see the example below for the linear regression case.

When contamination is present, the Bayesian posterior will be computed with respect to the contaminated likelihood 
$$k_{p_n}(x;\theta,\bmw)=(1-p_n)f(x;\theta,\bmw)+p_n\,g(x).$$
The prior $\pi_0$ on $\theta$ is always set as a zero-mean Gaussian prior with uncorrelated dimensions and standard deviation $10$.

The above estimation procedure is applied to the following three models:
\begin{enumerate}
    \item \textit{Bayesian Linear Regression}:
    
    \noindent 
    The regression model is given by 
    $
        X_i = \theta^{*\top}\bmw_i + e_i,
    $
    where $e_i\sim \mcal{N}(0,1)$.
    The likelihood function $f(X;\theta,\bmw)$ is given by
    $$
        f(x;\theta,\bmw) = \frac{1}{\sqrt{2\pi}}\exp\left(-(x-\theta^\top\bmw)^2/2\right).
    $$
    The contamination density is chosen as $g(X_i; \bmw_i)\sim T_5(\theta^{*\top}\bmw_i)$, a shifted Student's t-distribution with degrees of freedom $5$, given by the following expression,
    $$
        g(x) = \frac{\Gamma(3)}{\lambda\sqrt{3\pi}\Gamma(\tfrac{5}{2})}\left(1+\frac{x^2}{\lambda}\right)^{-3},
    $$
    where $\lambda$ is the dispersion parameter, which may vary according to the dimension of $\theta$.
    In this setup, a heavy-tailed contamination is required in order to remove the compactness constraint on the observation space.

    \item \textit{Bayesian Logistic Regression}:
    
    \noindent 
    The regression model is given by 
    $
        X_i \sim \text{Ber}\left((1+e^{-\theta^{*\top}\bmw_i})^{-1}\right),
    $
    where $X\in\{0,1\}$ with the likelihood function 
    $$
        f(x;\theta,\bmw) = \left(1+e^{-\theta^{\top}\bmw_i}\right)^{-x}\left(1+e^{\theta^{\top}\bmw_i}\right)^{-1+x}.
    $$
    Since the observation space is compact, $\epsilon$ is bounded in this setup.
    However, the simulation suggests that contamination should still be adopted.
    To contaminate the observations in this case, we replace each observation $X_i$ with probability $p_n$ by a draw from $\text{Ber}(\tfrac{1}{2})$.
    Thus, the contaminated density is
    $
        k_{p_n}(x;\theta,\bmw) = (1-p_n)f(x;\theta,\bmw)+0.5p_n.
    $

    \item \textit{Bayesian Cauchy Regression}: 

    \noindent 
    The regression model is given by
    $
        X_i = \theta^{*\top}\bmw_i + e_i
    $
    where $e_i\sim \text{Cauchy}(0)$ with the likelihood function
    $$
        f(x;\theta,\bmw) = \frac{1}{\pi\left[1+\left(x-\theta^\top \bmw\right)^2\right]}.
    $$
    For contamination, we use another Cauchy distribution with a different scale 
    $$
        g(x)= \frac{1}{\lambda\pi\left[1+x^2/\lambda^2\right]}
    $$
\end{enumerate}
For each setting, we apply Algorithm \ref{alg:empirical_estimation_full} to estimate $\epsilon_n$ using simulated datasets with sizes ranging from $1{,}000$ to $500{,}000$ and record the $99$th percentile of the estimated $\epsilon_n$.
In all simulations, $\delta_n$ is fixed to $(10n)^{-1}$ for numerical stability.
The estimated quantiles of $\hat{\epsilon}_n$ for each setting are plotted in Figure \ref{fig:dim_compare} as solid lines with markers.
The $\hat{\epsilon}_n$ are not log-scaled to better contrast the difference in magnitude between curves.
We have included the same plots, but on a log scale, to better illustrate that $\hat{\epsilon}_n$ decreases polynomially with $n$, which is consistent with Theorem \ref{thm:e-d_privacy}. 
We therefore fit the empirical decay rate of $\hat{\epsilon}_n$ and extrapolate the trend to larger sample sizes (up to $10^7$), shown as semi-transparent lines in the figure (without markers).

In the case of logistic regression, the likelihood ratio $d(x;\theta,\bmw)$ is uniformly bounded $\forall x$ even without contamination.
Nevertheless, Figure \ref{subfig:lgr} shows that introducing contamination yields a substantial reduction in $\epsilon_n$, approximately by a factor of four.

The behaviour differs for Cauchy regression (Figure \ref{subfig:cr}), where the likelihood ratio $d(x;\theta,\bmw)$ is likewise uniformly bounded $\forall x$.
In this case, the uncontaminated setup achieves a lower estimated $\epsilon_n$, which is likely due to a faster contraction in the posterior distribution. 
Here, a smaller neighbourhood $A_n$ is shown to be more effective in controlling the likelihood ratio than the additional randomness introduced via contamination.

We also examine the effect of parameter dimensionality by estimating $\epsilon_n$ for both models with parameter dimensions $5$ and $10$.
As expected, higher-dimensional parameter spaces incur a greater privacy cost, requiring larger sample sizes to achieve the same level of privacy.
{
\changed
This issue is not intrinsic to our method but arises because the privacy leakage of an output is roughly proportional to its dimension, making this an intrinsically challenging setting.
Our methodology does have the limitation that $\epsilon$ is difficult to control for small $n$, as the method achieves privacy through contamination more effectively when the contamination rate is moderate-to-low.
This situation is analogous to applying insufficient noise to the output and trading privacy for statistical efficiency in the Laplace mechanism case.
When the sensitivity scales linearly with the dimension of the function being privatised, achieving the same level of privacy requires the noise to also scale linearly with the dimension. This places our method close to a regime in which nearly the entire dataset must be contaminated in order to provide privacy. 
}

\begin{remark}
    Differentially private Bayesian inference may also be achieved by introducing contamination at alternative stages of the inference procedure. Various such approaches have been proposed in the literature, e.g., \cite{bernstein2019differentially, kulkarni2021differentially, sheffet2019old}.
    {\changed 
    Although these methods provide non-probabilistic privacy guarantees, they typically rely on restrictive assumptions, such as bounded observation or parameter spaces, which our framework explicitly avoids. 
    Moreover, the strength of their privacy guarantees is highly sensitive to the size of the imposed bounds and can deteriorate rapidly as these bounds increase.
    In contrast, our proposed framework is more general, as it contaminates the model at the data-level and does not require boundedness assumptions. 
    Consequently, a fair comparison between these approaches is difficult, and we therefore omit a direct comparison.
    }
\end{remark}

\begin{figure}[t]
    \centering
    \begin{subfigure}{0.8\linewidth}
        \centering
        \caption{}\label{subfig:lr}
        \includegraphics[width=\linewidth]{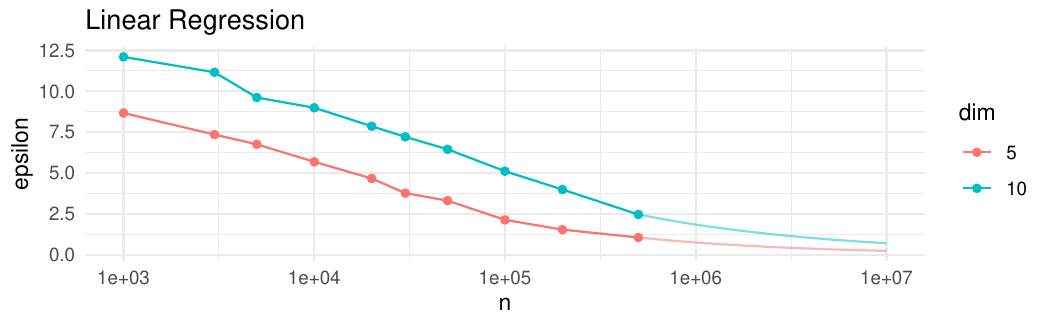}
    \end{subfigure} \hfill
    \begin{subfigure}{0.8\linewidth}
        \centering
        \caption{}\label{subfig:lgr}
        \includegraphics[width=\linewidth]{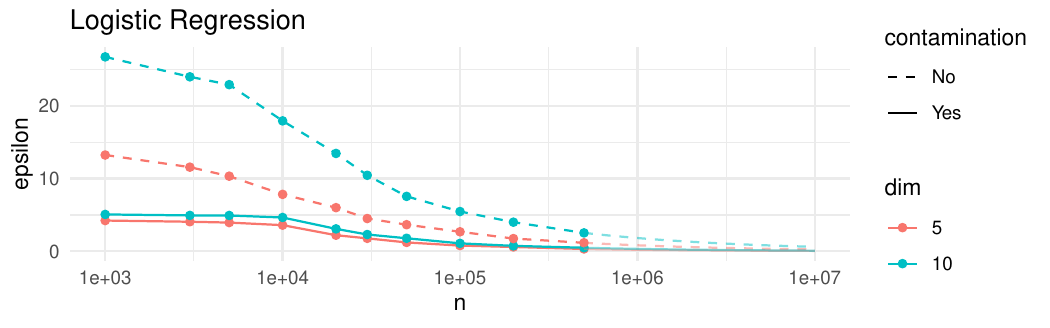}
    \end{subfigure}
    \begin{subfigure}{0.8\linewidth}
        \centering
        \caption{}\label{subfig:cr}
        \includegraphics[width=\linewidth]{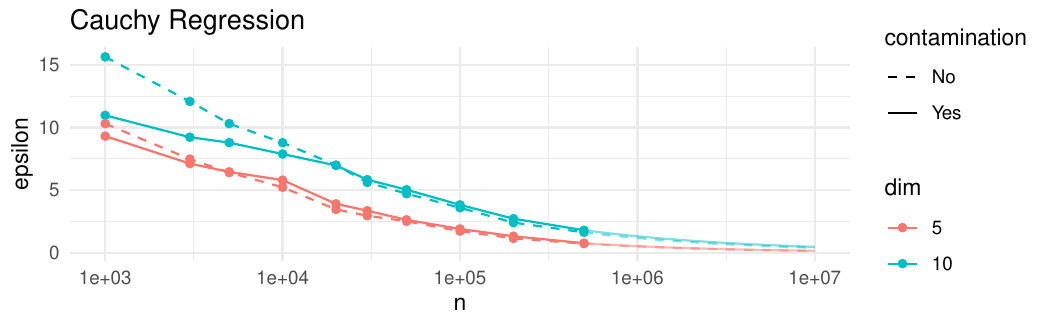}
    \end{subfigure}
    \caption{Plots of estimated $\epsilon_n$ under the setting of Linear Regression, Logistic Regression and Cauchy Regression as the dataset size $n$ varies. The pale lines without nodes are extrapolated estimations.}
    \label{fig:dim_compare}
\end{figure}

\section{Examples}\label{sec:examples}
We have seen in the previous section that empirically the approximate $\hat{\epsilon}_n$ follows the decreasing trend as in Theorem \ref{thm:e-d_privacy}.
In this section, we give a few examples where Theorem \ref{thm:e-d_privacy} is applicable theoretically. 
The verification of Assumptions \ref{as:bracket_entropy} to \ref{as:prior_tail} across different models shares a few common procedures.
For Assumption \ref{as:ratio_bound} to hold for regression models (with covariates), one needs to control the leverage of the covariates $\bm{W}$.
It suffices to assume $\|\bm{W}\|_\infty \leq n^{\iota}$, for some $\iota\in(0,1)$ that depends on the choice of $\phi_n$.
However, for simplicity, we will still assume $\|\bm{W}\|_\infty\leq 1$, as we did in the simulations in Section \ref{sec:simulation}.
With $\|\bm{W}\|_\infty\leq 1$, Assumption \ref{as:ratio_bound} can be trivially satisfied by correctly choosing a heavy-tailed contamination density $g$, and Assumption \ref{as:identifiability} is usually easy to check for linear models.

Due to space constraints, the detailed analysis is deferred to Appendix C.
Typically, we can set $\phi_n$ and $\psi_n$ in Proposition \ref{prop:hellinger_convergence} to decay at the same rate, with $\phi_n,\psi_n\propto n^{-1/q}$ and $p_n=n^{-1/2q}$, for some $q>4$, which gives
$$
\epsilon_n \sim n^{-1/2q},\qquad \delta_n\sim \exp(-\tfrac{1}{2}c_1 n^{1-2/q}),
$$
where $c_1$ is a fixed constant.

\section{Discussion}\label{sec:discussion}
Differential privacy mechanisms in statistical inference can be roughly divided into two categories:
1. Local models, where data are privatised before collection and inference is conducted on the perturbed dataset; 
2. Central models, where data are collected and processed without noise, but the result is privatised before release.

The "local model", or "local differential privacy" as defined in \cite{erlingsson2014rappor}, quantifies the level of differential privacy provided by the mechanism on "datasets" with only one data point.
Thus, contamination can be done individually before the dataset is gathered. Huber's contamination model is a popular approach to local differential privacy.
Our work, however, considers differential privacy from a \textbf{central model} perspective despite using a local privatisation mechanism.
Notably, the local differential privacy level only depends on the contamination probability $p_n$, but $p_n\rightarrow0$ as $n\rightarrow\infty$.
From the local model perspective, our setup will be asymptotically non-private as $n\rightarrow\infty$.
However, from the central model perspective, with a trusted curator, our approach can ensure full differential privacy almost surely with an asymptotically zero contamination rate.

Huber's contamination model for robust estimation by contaminating the dataset with a heavy-tailed distribution also introduces differential privacy to posterior sampling, similar to central models that directly inject Laplace/Gaussian noise into the output.
By contaminating the dataset with a heavy-tailed distribution, we can lift the assumption on bounded observation space.
We further extend the result, through a probabilistic argument, to lift the bound on the parameter space and upper bound the rate of convergence for $\epsilon$ and $\delta$.

Through simulation, a practitioner can estimate the level of differential privacy without having to hand-compute all the constants in the theorem.
Despite having several non-trivial assumptions for the theorem, we show in Section \ref{sec:examples} and Appendix C how the assumptions could be verified mathematically and how the convergence rate of $\epsilon$ and $\delta$ can be derived under a few common regression setups.

The analysis presented in this work relies on the assumption that i.i.d. samples are drawn exactly from the target distribution. 
The bounds may be extended to non-exact samplers by applying the results of \cite{minami2016differential}, provided that the algorithm admits a provable convergence guarantee. 
The study of convergence properties of Markov chain Monte Carlo methods is an active area of research (e.g., \cite{durmus2017nonasymptotic,andrieu2022comparison,andrieu2023weak}). 
Integrating such convergence results into our framework represents a promising avenue for future work.

Our main result (Theorem \ref{thm:e-d_privacy}) holds uniformly for all dataset sizes $n$, provided the underlying assumptions are satisfied.
Nevertheless, this should be interpreted as an asymptotic result since $n$ has to be large for some assumptions to hold in general.
This may be improved by using the results of \cite{ghosal2017fundamentals}, for which we include a detailed discussion in Appendix D.1.

Furthermore, posterior contraction results are also established in broader Bayesian inference frameworks, including misspecified models and infinite-dimensional parameter spaces.
While the specific assumptions employed here may not directly apply in those contexts, analogous differential privacy bounds can be derived under appropriately modified assumptions, provided a suitable contamination density can be constructed. We defer such generalisations to future work.

\begin{appendix}
\section{Proofs for Section \ref{sec:main_noncompact}} \label{sec:appx_main}

\begin{lemma}[Theorem 1 \citep{wong1995probability}]\label{lemma:wong1995probability}
Let $\bmX_1,\dots,\bmX_n$ be i.i.d. random variables following density $k_p(\cdot;\theta^*)$. Let $d(\bmx,\theta)$ denote the likelihood ratio between $\theta$ and $\theta^*$ at $\bmx$, $d(\bmx,\theta)=k_p(\bmx;\theta)/k_p(\bmx;\theta^*)$. 
Then $\forall \phi_n>0$, $\Omega_n\subset\Omega$, if
\begin{equation}\label{eq:entropy_bound_wong}
   \int_{\phi_n^2/2^{10}}^{\phi_n} H_{[]}^{1/2}\left(u,\mcal{K}_p(\Omega_n)\right)\dd u \leq c\sqrt{n}\phi_n^2,
\end{equation}
then 
$$
\mbb{P}^*_{\theta^*}\left(\sup_{\theta\in A_n^c\cap \Omega_n} \prod_{i=1}^n d(\bmx_i,\theta)\geq \exp({-c_1n\phi_n^2}) \right) \leq (1+L(\phi_n))(1+N(\phi_n))\exp({-c_2n\phi_n^2}),
$$
where $\mbb{P}^*_{\theta^*}$ denotes the outer measure with respect to $\bmX_i\sim k_p(\cdot;\theta^*)$, $A_n:=\{\theta\in\Omega : h(\theta^*,\theta)\leq \phi_n\}$, and $c,c_1,c_2$ are positive constants independent of $n$ and $\phi_n$.

If in addition
$
\exp\left(- c_2n\phi_n^2\right) \leq \frac{1}{2},
$
then
$$
\mbb{P}^*_{\theta^*}\left(\sup_{\theta\in A_n^c\cap \Omega_n} \prod_{i=1}^n d(\bmx_i,\theta)\geq \exp({-c_1n\phi_n^2}) \right) \leq 4\exp({-c_2n\phi_n^2}).
$$
\end{lemma}

\begin{proof}
    See Appendix B.2 in the supplementary material. In our choice of constants, we have $(1+L(\phi_n))(1+N(\phi_n))\leq 20\max\{-\log_2(\phi_n),(\log_2(\phi_n))^2\}$.
\end{proof}

\begin{lemma}\label{lemma:entropy_order}
    Let $p\in [0,1)$, $A\subset \Omega$, $\mcal{F}(A):=\{f(\cdot;\theta):\theta\in A\}$ and $\mcal{K}_p(A):=\{k_p(\cdot;\theta);\theta\in A\}$.
    Suppose that $1>p\geq p'$, then
    \begin{equation}
        H_{[]}\left(u\sqrt{\frac{1-p}{1-p'}}, \mcal{K}_{p}(A)\right) \leq H_{[]}(u,\mcal{K}_{p'}(A)).
    \end{equation}
    Setting $p'=0$, we have
    \begin{equation}
        H_{[]}\left(u, \mcal{K}_p(A)\right) \leq  H_{[]}\left(u\sqrt{1-p}, \mcal{K}_p(A)\right) \leq H_{[]}(u,\mcal{F}(A)).
    \end{equation}
\end{lemma}

\begin{lemma}[\cite{shen2001rates}]\label{lemma:shen_denominator}
Let $t_n$ be a sequence of positive numbers and $\alpha\in(0,1]$.
For any fixed $p\in[0,1)$,
$$
\mbb{P}\left(m_n(\Omega, \bm{X}) \leq \frac{1}{2}\mbb{P}_{\pi_0}(S_{p,\alpha}(t_n))e^{-2nt_n}\right) \leq 2e^{-n\alpha t_n}.
$$
\end{lemma}
\begin{proof}
    See Lemma 2 of \cite{shen2001rates}; the proof holds for any fixed $p\in[0,1)$.
\end{proof}

\begin{lemma}\label{lemma:divergence_order}
For $\alpha\in(0,1]$, $S_{p,\alpha}(t)=\{\theta\in\Omega : \rho_{p,\alpha}(\theta^*,\theta)\leq t\}$, then $S_{0,\alpha}(t)\subseteq S_{p,\alpha}(t)$.
\end{lemma}

\begin{lemma}\label{lemma:denominator_bound}
Under Assumption \ref{as:prior_center}, we have for any $p\in[0,1)$
    $$
    \mbb{P}\left(m_n(\Omega,\bm{X}) \leq \frac{c_3}{2}e^{-4nt_n} \right) \leq 2e^{-n\alpha t_n}.
    $$
\end{lemma}
\begin{proof}
By Lemma \ref{lemma:divergence_order} and Assumption \ref{as:prior_center}, 
$$
\mbb{P}_{\pi_0}\Bigl(S_{p,\alpha}(t_n)\Bigr) \geq \mbb{P}_{\pi_0}\Bigl(S_{0,\alpha}(t_n)\Bigr)\geq c_3 e^{-2nt_n}.$$
Then applying Lemma \ref{lemma:shen_denominator} with the above inequality gives the bound.
\end{proof}

\begin{lemma}\label{lemma:numerator_bound}
Recall that $A_n:=\{\theta\in\Omega: h(\theta,\theta^*)\leq\phi_n\}$.
Under Assumptions \ref{as:bracket_entropy} and \ref{as:prior_tail}, we have for any $p\in(0,1)$, $n>0$,
$$
    \mbb{P}\left(m_n(A_n^c,\bm{X}) \geq 2e^{-c_1n\phi_n^2}\right) \leq C(n,\phi_n)e^{-c_2 n\phi_n^2} + c_4 e^{-n c_1\phi_n^2}.
$$
where 
$$
C(n,\phi_n) = \begin{dcases}
4, & \text{for } nc_2\phi_n^2\geq \log 2.\\
    20\max\{-\log_2(\phi_n), (\log_2(\phi_n))^2\}, & \text{otherwise} 
\end{dcases}
$$
\end{lemma}

\begin{lemma}\label{lemma:ratio_bound}
    Under Assumptions \ref{as:ratio_bound} and \ref{as:identifiability}, there exists a sequence $T_n<\infty$ such that for any sequence $p_n\rightarrow0$
    $$
    \sup_{x\in\mbb{R}^d}\sup_{\theta\in A_n}|\log d(\bmx;\theta)| \leq \frac{1-p_n}{p_n}T_n\psi_n.
    $$
    where $d(\bmx;\theta) = k_{p_n}(\bmx;\theta)/k_{p_n}(\bmx;\theta^*)$.
    Thus in (\ref{eq:dp_decompose}), $\eta_n = \exp({2(1-p_n)T_n\psi_n/p_n})$.
\end{lemma}

\begin{proof}[Proof for Theorem \ref{thm:e-d_privacy}]
By Lemmas \ref{lemma:denominator_bound} and \ref{lemma:numerator_bound}, and recall that $t_n=c_1\phi_n^2/8$ defined in Assumption \ref{as:prior_center},
\begin{align*}
\mbb{P}\left(\frac{m_n(A_n^c,\bm{X})}{m_n(\Omega,\bm{X})} \geq \frac{4}{c_3}e^{-\frac{1}{2} c_1 n \phi_n^2}\right)  \leq &\, 
 \mbb{P}\left(\frac{m_n(A_n^c,\bm{X})}{m_n(\Omega,\bm{X})} \geq \frac{4}{c_3}e^{- c_1 n \phi_n^2 + 4nt_n}\right) \\
\leq &\, 2e^{-nc_1\alpha \phi_n^2/8} + C(n,\phi_n)e^{-c_2 n\phi_n^2} + c_4 e^{- c_1 n \phi_n^2}.
\end{align*}
Recall from (\ref{eq:dp_decompose}) that for any measurable $S\subset\Omega$
\begin{equation*}
    \mbb{P}_{\pi_n}(S|\bm{X}) \leq \eta_n\left[\eta_n + \frac{m_n(A_n^c,\bm{Z})}{m_n(\Omega,\bm{X})}\right]\mbb{P}_{\pi_n}(S|\bm{Z})+\frac{m_n(A_n^c,\bm{X})}{m_n(\Omega,\bm{X})},
\end{equation*}
and from Lemma \ref{lemma:ratio_bound} that
$
\eta_n \leq e^{2(1-p_n)T_n\psi_n/p_n}.
$
Let $\delta_n= \tfrac{4}{c_3}e^{-\frac{1}{2}c_1 n\phi_n^2}$, $\varepsilon_n = 2(1-p_n)T_n\psi_n/p_n$ and 
$\epsilon_n= \varepsilon_n + \log(\delta_n+\exp(\varepsilon_n)).$
Note that since $\bm{X}$ and $\bm{Z}$ have the same cardinality, applying Lemma \ref{lemma:numerator_bound} again to $m_n(A_n^c,\bm{Z})$, we have
\begin{align*}
  \mbb{P}\Bigl( \mbb{P}_{\pi_n}(S|\bm{X}) \leq e^{\epsilon_n} \mbb{P}_{\pi_n}(S|\bm{Z}) + \delta_n \Bigr)
  \geq &\, 1- 2\mbb{P}\left(\frac{m_n(A_n^c,\bm{X})}{m_n(\Omega,\bm{X})} \geq \frac{4}{c_3}e^{-\frac{1}{2}c_1 n\phi_n^2} \right) \\
  \geq &\, 1- 4e^{-nc_1\alpha \phi_n^2/8} - 2C(n,\phi_n) e^{-c_2n\phi_n^2} - 2c_4e^{-c_1 n \phi_n^2} \\
  \geq &\, 1- C'(n,\phi_n)e^{-nc_5 \phi_n^2},
\end{align*}
where $c_5 = \min\{c_1\alpha/8,c_2\}$ and $C'(n,\phi_n) = 4+2C(n,\phi_n)+2c_4$.
Note that by definition of $C(n,\phi_n)$ from Lemma \ref{lemma:numerator_bound}, $C'(n,\phi_n)=4+40(\log_2 n)^2+2c_4$.
On the other hand, if $n$ is large enough, then $C'(n,\phi_n)=12+2c_4$ is independent of $n$, and the proof is complete.
\end{proof}

\begin{proof}[Proof for Theorem \ref{thm:fisher_information}]
\begin{align*}
    I_{p_n}(\theta^*) & = \mbb{E}_{k_{p_n}}\left[(\nabla_{\theta}\log k_{p_n}(\bmx;\theta^*)) (\nabla_{\theta}\log k_{p_n}(\bmx;\theta^*))^{\top} \right] \\
    & = (1-p_n)^2 \int \frac{\left[\nabla_{\theta}f(\bmx;\theta^*)\right]\left[\nabla_{\theta}f(\bmx;\theta^*)\right]^{\top}}{k_{p_n}(\bmx;\theta^*)f(\bmx;\theta^*)} f(\bmx;\theta^*)\dd x.
\end{align*}
Let $H(\bmx) := \left[\nabla_{\theta}f(\bmx;\theta^*)\right]\left[\nabla_{\theta}f(\bmx;\theta^*)\right]^{\top}$ and let $[H(\bmx)]_{i,j}$ denote its $(i,j)$-th entry.
Consider the sets $A_{i,j}^+:=\{x: [H(\bmx)]_{i,j} > 0\}$ and similarly $A_{i,j}^-:=\{x:[H(\bmx)]_{i,j}<0\}$. Then we have
$$
0<\left|\frac{[H(\bmx)]_{i,j}}{k_{p_n}(\bmx;\theta^*)f(\bmx;\theta^*)} \right| \leq \frac{[H(\bmx)]_{i,j}}{(1-p_n)f^2(\bmx;\theta^*)}, \text{ for } x\in A_{i,j}^+,
$$
and
$$
0<\left|\frac{[H(\bmx)]_{i,j}}{k_{p_n}(\bmx;\theta^*)f(\bmx;\theta^*)}\right| \leq -\frac{[H(\bmx)]_{i,j}}{(1-p_n)f^2(\bmx;\theta^*)}, \text{ for } x\in A_{i,j}^-.
$$
Note that $A_{i,j}^+$ and $A_{i,j}^-$ are both measurable with
$$
[I(\theta^*)]_{i,j} = \int_{A_{i,j}^+\cup A_{i,j}^-} \frac{[H(\bmx)]_{i,j}}{f^2(\bmx;\theta^*)} f(\bmx;\theta^*)\dd x.
$$
Thus, applying the dominated convergence theorem on both $A_{i,j}^+$ and $A_{i,j}^-$,
\begin{align*}
    \lim_{n\rightarrow\infty} \left[I_{p_n}(\theta^*)\right]_{i,j} &\, = \lim_{n\rightarrow\infty} (1-p_n)^2 \left(\int_{A_{i,j}^+} \frac{[H(\bmx)]_{i,j}}{k_{p_n}(\bmx;\theta^*)f(\bmx;\theta^*)}f(\bmx;\theta^*)\dd x\right. \\
    &\,  \qquad\qquad\qquad\qquad \left.+ \int_{A_{i,j}^-} \frac{[H(\bmx)]_{i,j}}{k_{p_n}(\bmx;\theta^*)f(\bmx;\theta^*)}f(\bmx;\theta^*)\dd x  \right)\\
    &\, = \int_{A_{i,j}^+\cup A_{i,j}^-} \frac{[H(\bmx)]_{i,j}}{f^2(\bmx;\theta^*)}f(\bmx;\theta^*)\dd x= [I(\theta^*)]_{i,j}.
\end{align*}
\end{proof}

\section{Proofs for Section 3}

\subsection{Results in Main Text}

\begin{definition}[$u$-cover under $\|\cdot\|_\infty$]
    A $u$-cover of the set $\Omega$ is a set $\{\theta_1,\dots,\theta_n\}\subset\mbb{R}^d$ such that for any $\theta\in\Omega$, there exists $i\in\{1,\dots,n\}$ such that $\|\theta-\theta_i\|_\infty\leq u$.
\end{definition}
\begin{definition}[Covering Number]
    The $u$-covering number of $\Omega$ under $\|\cdot\|_\infty$ is 
    $$
    N(u,\Omega):=\min \{N\in\mbb{N} : \exists\text{ a $u$-cover under $\|\cdot\|_\infty$ of size }N \}.
    $$
\end{definition}

\begin{definition}[$u$-separated points under $\|\cdot\|_\infty$]
    A set of points $\{\theta_1,\dots,\theta_k\}\subset \Omega$ is $u$-separated if $\|\theta_i-\theta_j\|_\infty > u$, $\forall i\neq j$.
\end{definition}
\begin{definition}[Packing Number]
    The $u$-packing number of $\Omega$ under $\|\cdot\|_\infty$ is 
    $$
    D(u,\Omega):=\max \{N\in\mbb{N} : \exists\text{ a set of $N$ points that is $u$-separated under $\|\cdot\|_\infty$} \}.
    $$
\end{definition}

\subsubsection{Proof for Lemma 2} 
\begin{proof}
    The first half of this proof is based on several results commonly seen in the empirical processes literature (with some variation), see, for instance, Section 2.8 of \cite{van1996weak}.

    Recall that the (8) in Assumption 1 requires
    \begin{equation*}
    \int_{\phi_n^2/2^{10}}^{\phi_n} H^{1/2}_{[]}(u,\mcal{F}(\Omega_n))\dd u \leq c\sqrt{n} \phi_n^2.
    \end{equation*}
    To deduce (8), we first prove the following bound on $N_{[]}$, 
    \begin{equation}\label{eq:bracketing_bound}
    N_{[]}(u M^{(n)}_{u},\mcal{F}(\Omega_n)) \leq C_d\text{Vol}(\tilde{\Omega}_n)u^{-d},
    \end{equation}
    where $\tilde{\Omega}_n:=\{\theta\in\mbb{R}^d: \exists\phi\in\Omega_n,\|\theta-\phi\|_\infty\leq u/2\}$.
    
    Let $\{\theta_1,\dots,\theta_k\}$ be a $u$-cover of $\Omega_n$ with respect to $\|\cdot\|_\infty$.
    Consider the following functions, for $i=1,\dots,k$,
    $$
    \tilde{L}_i(\bmx):= \max\left\{\sqrt{f_{\theta_i}(\bmx)}-u M^{(n)}_{\theta_i,u}(\bmx),0 \right\}, \,\,\, \tilde{U}_i := \sqrt{f_{\theta_i}(\bmx)} +u M^{(n)}_{\theta_i,u}(\bmx),
    $$
    and $L_i(\bmx):=\tilde{L}_i(\bmx)^2$, $U_i:=\tilde{U}_i(\bmx)^2$.
    Since 
    $$
    |\sqrt{f_\theta(\bmx)} - \sqrt{f_{\theta_i}(\bmx)}| \leq M^{(n)}_{\theta_i,u}(\bmx)\|\theta-\theta_i\|_\infty,
    $$
    we have, for any $\theta$ such that $\|\theta-\theta_i\|_\infty\leq u$,
    $$
    \tilde{L}_i(\bmx) \leq \sqrt{f_\theta(\bmx)} \leq \tilde{U}_i(\bmx).
    $$
    Moreover, recall that $M^{(n)}_{u}:=\sup_{\vartheta\in\Omega_n}\|M^{(n)}_{\vartheta,u}\|_2$,
    \begin{align*}
        h^2(L_i,U_i) &= \frac{1}{2}\iint \left(\sqrt{L_i}-\sqrt{U_i}\right)^2 \dd \bmx\\
        & \leq \iint u^2 M^{(n)}_{\theta_i,u}(\bmx)^2\dd \bmx \leq  (uM^{(n)}_{u})^2.
    \end{align*}
   Thus the set of $[L_i,U_i]$ forms a $u M^{(n)}_{u}$-bracketing for $\mcal{F}(\Omega_n)$ under the Hellinger distance $h(\cdot,\cdot)$, hence,
   $$
        N_{[]}(u M^{(n)}_{u},\mcal{F}(\Omega_n)) \leq N(u,\Omega_n).
   $$

   To complete (\ref{eq:bracketing_bound}), we recall that $N(u,\Omega_n)\leq D(u,\Omega_n)$ always holds.
   If $\{\theta_1,\dots,\theta_k\}$, $k:=D(u,\Omega_n)$, is a set of $u$-separated points, then the balls of radius $u/2$ centred at $\{\theta_1,\dots,\theta_k\}$ are disjoint and contained in $\tilde{\Omega}_n:=\{\theta\in\Omega: \exists \theta'\in\Omega_n,\|\theta-\theta'\|_2\leq u/2\}$.
   Thus $D(u,\Omega_n)$ is upper bounded by the ratio between the volume of the space $\tilde{\Omega}_n$ and the volume of a ball of radius $u/2$.
   Let $v_d$ denote the volume of a unit ball, and define $C_d:=2^d/v_d$, then
   $$
    N_{[]}(u M^{(n)}_{u},\mcal{F}(\Omega_n)) \leq N(u,\Omega_n) \leq D(u,\Omega_n) \leq C_d\Vol(\tilde{\Omega}_n) u^{-d}.
   $$
   Recall that $R_n$ is the radius of $\Omega_n$.
   $\phi_n$ is a decreasing sequence, let $\phi_0$ denote an upper bound for $\{\phi_n\}$ then $u\in(0,\phi_0)$ and $\Vol(\tilde{\Omega}_n)=O((R_n+u/2)^d)<O((R_n+\phi_0)^d)$, thus,
   \begin{equation}\label{eq:entropy_expanded}
       H_{[]}(u,\mcal{F}(\Omega_n)) \leq \log C^* + d\log(R_n+\phi_0) + d\log M_u^{(n)} - d\log u.
   \end{equation}
   Since $H_{[]}(u,\mcal{F}(\Omega_n)$ decreases with $u$ but $cnu^2$ increases with $u$, it suffices to check 
   $$
        H_{[]}(\phi_n^2/2^{10}, \mcal{F}(\Omega_n)) \leq  c^2n\phi_n^2
   $$
   to prove that
   $$
    H_{[]}(u,\mcal{F}(\Omega_n))\leq c^2n\phi_n^2, \quad \forall u\in(\phi_n^2/2^{10},\phi_n).
   $$
   Due to the restriction that $n\phi_n^2\rightarrow\infty$, we have $\phi_n^{-1}=o(n^{1/2})$, thus every term on the right-hand side of (\ref{eq:entropy_expanded}) are of order $\log n$, apart from $\log M_u^{(n)}$.
   By assumption, there exists $q\in(0,1)$ such that $\forall u\in(0,\phi_0)$,
   $$
        \frac{\log\log M_u^{(n)}}{\log n}< q,
   $$
   then $\log M_u^{(n)} = O(n^q)$, and we can choose $\phi_n> n^{-\tfrac{1-q}{2}}$ such that
   \begin{align*}
       H_{[]}(\phi_n^2/2^{10}, \mcal{F}(\Omega_n)) \leq &\, \log C^* + d\log(R_n+\phi_0)+d\log M_{\phi_n^2/2^{10}}^{(n)} - 4d\log \phi_n \\
       < &\, c^2n^q \leq c^2n\phi_n^2.
   \end{align*}
   for $n$ sufficiently large, which implies (8).
\end{proof}

\subsubsection{Proof for Lemma 6}
\begin{proof}
    Recall that
    $$
        \rho_{0,\alpha}(\theta^*,\theta) := \frac{1}{\alpha}\int\left[\left(\frac{f(\bmx;\theta^*)}{f(\bmx;\theta)}\right)^{\alpha}-1\right]f(\bmx;\theta^*)\dd \bmx.
    $$
    By the assumption that $f(\bmx;\theta)$ has all directional derivatives in $\theta$ and the definition of $\tilde{\bm{\grad}}$ from Definition 7,
    $$
        \left|\left(\frac{f(\bmx;\theta^*)}{f(\bmx;\theta)}\right)^{\alpha}-1 \right| \leq \max_{\theta,\tilde{\theta}\in \bar{S}(c_r)} \alpha\left(\frac{f(\bmx;\theta^*)}{f(\bmx;\theta)}\right)^{\alpha-1}\frac{f(\bmx;\theta^*)}{f(\bmx;\theta)^2}\|\tilde{\bm{\grad}}_{\theta}f(\bmx;\tilde{\theta})\|_2\|\theta-\theta^*\|_2.
    $$
    Therefore, if $\|\theta-\theta^*\|_2\leq c_r$, then
    \begin{align*}
        \rho_{0,\alpha}(\theta^*,\theta) \leq &\, \left(\max_{\theta,\tilde{\theta}\in \bar{S}(c_r)}\int \left[\frac{f(\bmx;\theta^*)}{f(\bmx;\theta)}\right]^{\alpha+1}\|\tilde{\bm{\grad}}_{\theta}f(\bmx;\tilde{\theta})\|_2\dd \bmx\right) \|\theta-\theta^*\|_2 \\
        \leq &\, M_{c_r}\|\theta-\theta^*\|_2.
    \end{align*}
    In other words, $\bar{S}(t_n)\subseteq S_{0,\alpha}(M_{c_r}t_n)\cap \bar{S}(c_r)$, and $\bar{S}(t_n/M_{c_r})\subseteq S_{0,\alpha}(t_n)$.
\end{proof}

\subsubsection{Proof for Lemma 7}
\begin{proof}
Recall that $\bar{S}(r):=\{\theta\in\Omega:\|\theta^*-\theta\|_2\leq r\}$ and the prior distribution on $\theta$ is $\pi_0(\theta)\sim \mcal{N}(0,I_d)$.
Let $\theta_i$ denote the $i$th dimension of vector $\theta$, then
$$
\mbb{P}_{\pi_0}(\bar{S}(r)) = \int_{\|\theta^*-\theta\|_2\leq r} (2\pi)^{-d/2}\exp\left(-\frac{1}{2}\sum_{i=1}^d\theta_i^2\right)\dd \theta_1\cdots\dd\theta_d.
$$
\begin{enumerate}
    \item For the case when $r$ is small, let $w_i=\theta_i-\theta_i^*$. Then $w_i\sim\mcal{N}(\theta^*_i,1)$.
    Thus $\sum_i w_i^2$ is a non-central $\chi^2$ distribution with $d$ degrees of freedom and non-centrality parameter $\lambda=\|\theta^*\|_2^2$, with probability density function given by \citep{patnaik1949non}
    $$
        f_{\chi^2}(x;d,\lambda) = \frac{1}{2}\exp\left(-\frac{x+\lambda}{2}\right)\left(\frac{x}{\lambda}\right)^{(d-2)/4}I_{d/2-1}(\sqrt{\lambda x}),
    $$
    where $I_{\nu}(x)$ is the modified Bessel function of the first kind,
    $$
    I_{\nu}(x) = (x/2)^{\nu} \sum_{j=0}^\infty \frac{(x^2/4)^j}{j!\Gamma(\nu+j+1)}.
    $$
    Now,
    \begin{align*}
        \mbb{P}_{\pi_0}(\bar{S}(r)) =&\, \mbb{P}\left(\chi^2(d,\lambda)\leq r^2\right)\\
        = &\, \int_0^{r^2} \frac{1}{2}\exp\left(-\frac{x+\lambda}{2}\right)\left(\frac{x}{\lambda}\right)^{(d-2)/4}I_{d/2-1}(\sqrt{\lambda x}) \dd x \\
        \geq &\, \int_0^{r^2} \frac{1}{2\Gamma(d/2)}\exp\left(-\frac{x+\lambda}{2}\right)\left(\frac{x}{\lambda}\right)^{(d-2)/4} \left(\frac{\sqrt{\lambda x}}{2}\right)^{d/2-1}\dd x\\
        = &\, \frac{2^{-d/2}\exp(-\lambda/2)}{\Gamma(d/2)} \int_0^{r^2}x^{d/2-1}\exp(-x/2)\dd x \\
        =&\, \frac{\exp(-\lambda/2)}{\Gamma(d/2)}\gamma(d/2,r^2/2).
    \end{align*}
    \item For the case where $r>\|\theta^*\|_2$, note that
    $$
    \{\theta\in\Omega: \|\theta\|_2^2 \leq (r-\|\theta^*\|_2)^2\} \subseteq \{\theta\in\Omega: \|\theta^*-\theta\|_2^2\leq r^2\}. 
    $$
    Denote $R = r-\|\theta^*\|_2$,
    \begin{align*}
        \mbb{P}_{\pi_0}(\bar{S}(r)) \geq &\, \mbb{P}\left(\chi^2(d,0) \leq R^2\right) \\
        = &\, \frac{1}{\Gamma(d/2)}\gamma(d/2,R^2/2).
    \end{align*}
\end{enumerate}
\end{proof}

\subsubsection{Proof for Proposition 3} 

\begin{proof}
If Assumption 5 holds, then the same assumption holds for the set of density functions $k_p(\bmx;\theta)$ for any $p\in(0,1)$. 
It suffices to prove the proposition holds for $f$ given Assumption 5; the same argument can be used to prove this proposition holds for any $k_p$, $p\in(0,1)$. 

\noindent
Recall that by Assumption 5, $\forall \psi> 0$, 
$$
    \inf_{\|\theta-\theta^*\|_2>\psi} h(f(\cdot;\theta),f(\cdot;\theta^*))>0.
$$
Denote
$$
    A(\psi) := \inf_{\|\theta-\theta^*\|_2>\psi} h(f(\cdot;\theta),f(\cdot;\theta^*)),
$$
and $M = \sup_{\psi} A(\psi)>0$.
Note that $A(\psi)$ is non-decreasing and right-continuous since the infimum is taken over the set excluding $\|\theta-\theta^*\|=\psi$.
Then $\forall \phi<M$, $\exists\psi$ such that $\psi=\min\{\psi':\phi\leq A(\psi')\}$.
Therefore, $\forall \phi_n\rightarrow0$, $\exists\psi_n$ non-increasing such that $\phi_n\leq A(\psi_n)$ and we may choose $\psi_n:=\min\{\psi:\phi_n\leq A(\psi)\}$.
Suppose for contradiction that $\psi_n\not\rightarrow0$, then $\forall m>0$, $\exists N>0$ such that $\forall n>N$,
$$
    \inf\{\psi: \phi_n\leq A(\psi)\} > m.
$$
Thus, $\forall\psi'\leq m$, we have $A(\psi')<\phi_n$, but since $\phi_n\rightarrow0$, we have
$$
\inf_{\|\theta-\theta^*\|_2>m} h(f(\cdot;\theta),f(\cdot;\theta^*))=0,
$$
which contradicts the assumption.
Hence $\psi_n\rightarrow0$ and the proof is completed.

\end{proof}

\subsection{Proofs Related to Lemma 8} 
Recall that $d(\bmx;\theta):=k_p(\bmx;\theta)/k_p(\bmx;\theta^*)$ where $\theta^*$ is the true parameter.
For $\tau>0$ to be chosen later, let $$\tilde{Z}_{\theta}(\bmx):=\max\{-\tau,\log d(\bmx,\theta)\},$$ 
and for functions $u:\bmX\mapsto u(\bmX)$, define the empirical process
$$
\nu_n(u):= n^{-1/2}\sum_{i=1}^n u(\bmX_i)-\mbb{E}u(\bmX_i).
$$
\begin{theorem}[One-sided Bernstein's Inequality]
Let $X_1,\dots,X_n$ be i.i.d. random variables, then the following holds:
\begin{itemize}
    \item If $|X_i-\mbb{E}X_i|\leq T,\forall i\in\{1,\dots,n\}$ and $\Var(X_i)\leq v$, then
    \begin{equation}\label{eq:bernstein-bdd}
        \mbb{P}^*\left(n^{-1/2}\sum_{i=1}^n (X_i-\mbb{E}X_i) > M\right) \leq \exp\left(-\frac{M^2}{2(v+MT/3n^{1/2})}\right);
    \end{equation}
    \item If $\mbb{E}|X_i|^j\leq j!b^{j-2}v/2$ for any $j\geq 2$, then
    \begin{equation}\label{eq:bernstein-moment}
        \mbb{P}^*\left(n^{-1/2}\sum_{i=1}^n (X_i-\mbb{E}X_i) > M\right) \leq \exp\left(-\frac{M^2}{2(v+bM/n^{1/2})}\right),
    \end{equation}
\end{itemize}
where $\mbb{P}^*$ denotes the outer measure with respect to the measure of $\bmX_i$s.
\end{theorem}
\begin{proof}See \cite{bennett1962probability}.\end{proof}

\begin{lemma}[Lemma 7 \citep{wong1995probability}]\label{lemma:wong7}
    Let $\bmX_1,\dots,\bmX_n$ be i.i.d. random variables following density $k_p(\cdot;\theta^*)$. For any $t>0$, $0<k<1$, $M>0$, let
    $$
        \Psi(M,t^2,n) :=\frac{M^2}{4(8c_0t^2+M/n^{1/2})},
    $$
    where 
    $$c_0:=\frac{\exp(\tau/2)-1-\tau/2}{(1-\exp(-\tau/2))^2},\quad M:=kn^{1/2}t^2,$$
    and $\tau>0$ is a constant to be chosen.
    Let 
    $$
    s(t):=\frac{kM}{16\sqrt{2n}\exp(\tau/2)}=\frac{k^2t^2}{2^{4.5}\exp(\tau/2)}.
    $$
    If 
    \begin{equation}\label{eq:wong7entropy}
        \int_{s(t)/4}^t H_{[]}^{1/2}(u,\mcal{K}_p(\Omega_n))\dd u\leq \frac{\sqrt{3}k^{3/2}M}{2^{10}\exp(\tau/2)} = \frac{\sqrt{3}k^{5/2}n^{1/2}t^2}{2^{10}\exp(\tau/2)},
    \end{equation}
    then
    $$
        \mbb{P}^*\left(\sup_{\substack{h(\theta,\theta^*)\leq t\\\theta\in\Omega_n}} \nu(\tilde{Z}_{\theta})\geq M \right)\leq (1+N(t))\exp(-(1-k)\Psi(M,t^2,n)),
    $$
    where
    $$
    N(t):=\min\{x\in\mbb{N}:2^x s(t)>t\}.
    $$
\end{lemma}

\begin{proof}
    The proof follows from a similar chaining argument as in \cite{shen1994convergence}.
    In this proof only, we will also consider the bracketing entropy computed in $L_2$-norm, denoted by $H_{[]}(u,\mcal{F},\|\cdot\|_2)$. To make further distinction, we will denote the Hellinger bracketing entropy by $H_{[]}(u,\mcal{F},h)$.

    Recall that for $\theta\in\Omega_n$, we define $\tilde{Z}_{\theta}(\bmx):=\max\{-\tau,\log d(\bmx,\theta)\}$.
    Let $\mcal{G}:=\mcal{K}_p(\Omega_n)$ denote the space of density functions of interest, and let $\tilde{\mcal{Z}}:=\{\tilde{Z}_{\theta}: \theta\in\Omega_n\}$ be the space of lower-truncated log-likelihood ratios.
    Note that the two spaces are linked, in the sense that the Hellinger bracketing entropy of $\mcal{G}$ is related to the $L_2$-bracketing entropy of $\tilde{\mcal{Z}}$ by
    $$
        H_{[]}(2\sqrt{2}\exp(\tau/2)\epsilon,\tilde{\mcal{Z}},\|\cdot\|_2)\leq H_{[]}(\epsilon,\mcal{G},h),
    $$
    see Lemma 3 of \cite{wong1995probability} or note that
    \begin{equation}\label{eq:metric_link}
        \| \tilde{Z}_{\theta_1} - \tilde{Z}_{\theta_2}\|_2^2 \leq 4\exp(\tau)\|k_p^{1/2}(\cdot;\theta_1) - k_p^{1/2}(\cdot;\theta_2)\|_2^2 = 8\exp(\tau)h^2(\theta_1,\theta_2). 
    \end{equation}
    Provided that (which we will prove directly later)
    \begin{equation}\label{eq:low_entropy_as}
        H_{[]}(t,\mcal{G},h)\leq \frac{k}{4}\Psi(M,t^2,n),
    \end{equation}
    we can construct a sequence of bracketings on the space $\mcal{G}$ and $\tilde{\mcal{Z}}$.
    Let $t\geq \delta_0 > \delta_1>\cdots>\delta_N>0$, and let $\mcal{F}_0,\dots\mcal{F}_N$ denote the set of Hellinger bracketings of $\mcal{G}$ with bracket size $\delta_0,\dots,\delta_N$ respectively.
    We may choose $\delta_j$ such that 
    $$
    \delta_0:=\max\{\delta\leq t :  H_{[]}(\delta,\mcal{G},h)\leq \tfrac{k}{4}\Psi(M,t^2,n)\},
    $$
    and for $j=0,\dots,N-1$,
    $$
        \delta_{j+1} := \max\{s(t), \sup\{\delta\leq \delta_j/2: H_{[]}(\delta,\mcal{G},h)\geq 4H_{[]}(\delta_j,\mcal{G},h)\}\},
    $$
    with $N:=\min\{j:\delta_j\leq s(t)\}$.
    Furthermore, let 
    $$\tilde{\delta}_0 := \left(\sup\{\tilde{\delta} \leq 2\sqrt{2}\delta_0\exp(\tau/2): H_{[]}(\tilde{\delta},\tilde{\mcal{Z}},\|\cdot\|_2)=H_{[]}(\delta_0,\mcal{G},h)\}\right)^{-},
    $$ 
    and $\tilde{\delta}_j:=2\sqrt{2}\delta_j\exp(\tau/2)$, then
    $$ 
    H_{[]}(\tilde{\delta}_j,\tilde{\mcal{Z}},\|\cdot\|_2)\leq H_{[]}(\delta_j, \mcal{G},h).
    $$
    Note that the entropy function $H_{[]}$ is decreasing and right-continuous.
    We want to evaluate $H_{[]}(\tilde{\delta}_0,\tilde{\mcal{Z}},\|\cdot\|_2)$ at the left limit $(\tilde{\delta}_0)^{-}$ if a jump happens at $\tilde{\delta}_0$, so that $H_{[]}(\tilde{\delta}_0,\tilde{\mcal{Z}},\|\cdot\|_2)=H_{[]}(\delta_0,\mcal{G},h)$.
    
    The space $\tilde{\mcal{Z}}$ can be obtained by transforming elements of $\mcal{G}$ through mapping $f\mapsto \max\{-\tau, \log(f/k_p(\cdot;\theta^*))\}$, we may assume the choice of bracketing in one space induces another for the case $j=0$.
    
    For each $\theta\in\Omega_n$, there exists a bracket $[\tilde{Z}^{L}_{j,\theta},\tilde{Z}^{U}_{j,\theta}]$ for every $j$ such that
    $\tilde{Z}^{L}_{j,\theta} \leq \tilde{Z}_\theta \leq  \tilde{Z}^{U}_{j,\theta}$ and
    $\|\tilde{Z}^{L}_{j,\theta}-\tilde{Z}^{U}_{j,\theta}\|_2\leq \tilde{\delta}_j$.
    Define a sequence of shrinking envelope functions for each $\theta$ by
    $$
        u_j(\theta) := \min_{i\leq j} \tilde{Z}^{U}_{i,\theta},\quad l_j(\theta):= \max_{i\leq j} \tilde{Z}^{L}_{i,\theta}.
    $$
    Then $l_j(\theta)\leq \tilde{Z}_\theta \leq u_j(\theta)$, and $\|u_j-l_j\|_2\leq \tilde{\delta}_j$. 
    For notational simplicity, we hide the dependency of $u_j(\theta),l_j(\theta)$ on $\theta$ and write them directly as $u_j$ and $l_j$ as functions of $\bmX_i$.
    Finally, define sequences $a_1>\cdots>a_N$, $\eta_0,\dots,\eta_N$ by
    \begin{equation}\label{eq:eta,a}
        a_j := \frac{\sqrt{n}\tilde{\delta}_{j-1}^2}{\eta_{j-1}}, \quad \eta_j:= \frac{4\tilde{\delta}_j}{\sqrt{3k}}\left(\sum_{i\leq j+1}H_{[]}(\tilde{\delta}_i,\tilde{\mcal{Z}},\|\cdot\|_2)\right)^{1/2}, \quad \eta_N:=0.
    \end{equation}
    We may decompose $\tilde{Z}_\theta$ on a partition on the sample space $B_0,\dots,B_N$
    \begin{align}
        \tilde{Z}_\theta =&\,  u_0 + \sum_{j=0}^N (u_j I_{B_j} - u_0I_{B_j}) + \left(\tilde{Z}_{\theta}-\sum_{j=0}^N u_j I_{B_j}\right) \nonumber\\
        = &\, u_0 + \sum_{j=1}^N \sum_{i=1}^j (u_i-u_{i-1})I_{B_j} + \sum_{j=0}^N (\tilde{Z}_{\theta}-u_j)I_{B_j} \nonumber\\
        = &\, u_0 + \sum_{j=1}^N(u_j-u_{j-1})I_{\cup_{i\geq j}B_i} + \sum_{j=0}^N (\tilde{Z}_{\theta}-u_j)I_{B_j}, \label{eq:log_decomp_wong}
    \end{align}
    where 
    \begin{align*}
        & B_0:=\{(u_0-l_0)\geq a_1\},\\
        & B_j := \{(u_j-l_j)\geq a_{j+1},\,\,\, (u_i-l_i)<a_{i+1},i=1,\dots,j-1\},\\
        & B_N := \left(\cup_{j=0}^{N-1} B_j\right)^c.
    \end{align*}
    Note that
    \begin{align*}
        \sum_{j=0}^N \eta_j = &\, \sum_{j=0}^{N-1} 4\tilde{\delta}_j(3k)^{-1/2}\left(\sum_{i\leq j+1}H_{[]}(\tilde{\delta}_i,\tilde{\mcal{Z}},\|\cdot\|_2)\right)^{1/2} \\
        \leq &\, \frac{4\sqrt{2}}{\sqrt{3k}}\sum_{j=0}^{N-1}\tilde{\delta}_j(H(\delta_{j+1},\mcal{G},h))^{1/2}\\
        \leq &\, \frac{2^7}{\sqrt{3k}}\exp(\tau/2)\int_{s/4}^t H_{[]}^{1/2}(u,\mcal{G},h)\dd u \leq  \frac{k M}{8},
    \end{align*}
    where the second-last step is by Lemma 3.1 of \cite{alexander1984probability} and the last step is by (\ref{eq:wong7entropy}).
    Thus, through the decomposition of (\ref{eq:log_decomp_wong}), we note
    \begin{align*}
        \mbb{P}^*\left(\sup_{\substack{h(\theta,\theta^*)\leq t\\\theta\in\Omega_n}} \nu(\tilde{Z}_{\theta})\geq M \right) \leq &\, \mbb{P}^*\left(\sup_{h(\theta,\theta^*)\leq t} \nu_n(u_0) > (1-\tfrac{3k}{8})M\right) \\
         + &\, \sum_{j=0}^{N-1} \mbb{P}^*\left(\sup_{h(\theta,\theta^*)\leq t} \nu_n(u_{j+1}-u_j) I_{i\geq j+1}B_i >\eta_j\right) \\
         + &\, \sum_{j=0}^{N-1} \mbb{P}^*\left(\sup_{h(\theta,\theta^*)\leq t} \nu_n(\tilde{Z}_{\theta} - u_j)I_{B_j} > \eta_j\right) \\
         + &\, \mbb{P}^*\left(\sup_{h(\theta,\theta^*)\leq t} \nu_n(\tilde{Z}_{\theta} -u_N)I_{B_N} > \tfrac{k M}{8} +\eta_N\right) \\
         \leq &\, \mbb{P}_1+\mbb{P}_2 + \mbb{P}_3 + \mbb{P}_4,
    \end{align*}
    where
    \begin{align*}
    \mbb{P}_1 :=&\, |\mcal{F}_0|\sup_{h(\theta,\theta^*)\leq t} \mbb{P}(\nu_n(u_0)>(1-\tfrac{3k}{8})M), \\
    \mbb{P}_2 :=&\, \sum_{j=0}^{N-1} \prod_{l_1=0}^j |\mcal{F}_{l_1}|\prod_{l_2=0}^{j+1} |\mcal{F}_{l_2}| \sup_{h(\theta,\theta^*)\leq t}\mbb{P}(\nu_n(u_{j+1}-u_j) I_{\cup_{i\geq j+1} B_i} > \eta_j), \\
    \mbb{P}_3 :=&\,  \sum_{j=0}^{N-1} \mbb{P}^*\left(\sup_{h(\theta,\theta^*)\leq t} \nu_n(\tilde{Z}_{\theta} - u_j)I_{B_j} > \eta_j\right), \\
    \mbb{P}_4 :=&\, \mbb{P}^*\left(\sup_{h(\theta,\theta^*)\leq t} \nu_n(\tilde{Z}_{\theta} -u_N)I_{B_N} > \tfrac{k M}{8} +\eta_N\right).
    \end{align*}
    Before bounding $\mbb{P}_i$, we can quickly check (\ref{eq:low_entropy_as}). Note that Hellinger distance has an upper bound of 1, so we may assume $t\leq 1$.
    Then by definition of $s(t)$, $s(t) \leq t/2\sqrt{2} \leq t/2$, so since $H_{[]}(u,\mcal{G},h)$ is decreasing,
    $$
        H_{[]}(t,\mcal{G},h) \leq \left(\frac{1}{t-t/8}\int_{s/4}^t H_{[]}^{1/2}(u,\mcal{G},h)\dd u\right)^2 \leq \frac{3k^{5}nt^2}{2^{14}\cdot 49\exp(\tau)} \leq \frac{k^3nt^2}{2^7c_0+16k} = \frac{k}{4}\Psi(M,t^2,n).
    $$
    for any $2^7c_0+16k\leq 2^{14}\cdot49/3$, which can be easily satisfied by a wide range of choices, including the values we will use in the proof for main results.
    \begin{itemize}
        \item To bound $\mbb{P}_1$, we use (\ref{eq:bernstein-moment}). By Lemma 5 \citep{wong1995probability}, we have
        $$
        \mbb{E}|u_0|^j \leq j!2^jc_0 \|w^{1/2}(\cdot)-k_p^{1/2}(\cdot,\theta^*)\|_2^2, \,\,\forall j\geq 2,
        $$
        where $w(x):=\exp(u_0(x))k_p(x,\theta^*)$.
        Note that 
        $$
         \|w^{1/2}(\cdot)-k_p^{1/2}(\cdot,\theta^*)\|_2^2 \leq \|w^{1/2} - k_p^{1/2}(,\theta)\|_2^2 + h^2(\theta,\theta^*) \leq 2t^2.
        $$
        Thus, recall Bernstein's inequality (\ref{eq:bernstein-moment}) and letting $b:=2$, $v:=8c_0\|w^{1/2}(\cdot)-k_p^{1/2}(\cdot,\theta^*)\|_2^2$,
        \begin{align*}
            \mbb{P}_1 \leq &\, |\mcal{F}_0|\mbb{P}(\nu_n(u_0)>(1-\tfrac{3k}{8})M)\\
            \leq &\, \exp\left(H_{[]}(\delta_0,\mcal{G},h) -\Psi((1-\tfrac{3k}{8})M,t^2,n)\right)\\
            \leq &\, \exp\left(\tfrac{k}{4} \Psi(M,t^2,n) - (1-\tfrac{3k}{8})^2\Psi(M,t^2,n)\right)\\
            \leq &\, \exp\left(-(1-k)\Psi(M,t^2,n)\right),
        \end{align*}
        where we use (\ref{eq:low_entropy_as}) to bound $H_{[]}(\delta_0,\mcal{G},h)$ at the second step.
        \item 
        $$
         \mbb{P}_2 := \sum_{j=0}^{N-1} \prod_{l_1=0}^j |\mcal{F}_{l_1}|\prod_{l_2=0}^{j+1} |\mcal{F}_{l_2}| \sup_{h(\theta,\theta^*)\leq t}\mbb{P}(\nu_n(u_{j+1}-u_j) I_{\cup_{i\geq j+1} B_i} > \eta_j)
        $$
        By the definition of $B_j$, on the set $\cup_{i\geq j+1}B_j$, we have $-a_{j+1}<l_j-u_j<u_{j+1}-u_j<0$, thus
        \begin{align*}
            \Var((u_{j+1}-u_j) I_{\cup_{i\geq j+1} B_i}) \leq &\, \mbb{E}[(u_{j+1}-u_j)^2 I_{\cup_{i\geq j+1} B_i}] \\
            \leq &\, \mbb{E}[(u_{j+1}-u_j)^2] \leq  \tilde{\delta}_j^2.
        \end{align*}
        Thus, by the one-sided Bernstein's inequality for bounded random variables (\ref{eq:bernstein-bdd}),
        $$
            \mbb{P}(\nu_n(u_{j+1}-u_j) I_{\cup_{i\geq j+1} B_i} > \eta_j) \leq \exp\left(-\frac{\eta_j^2}{2(\tilde{\delta}_j^2+a_{j+1}\eta_j/3n^{1/2})}\right).
        $$
        For $j=0$, note that 
        $$
            (u_1-u_0)I_{\cup_{j\geq 1}B_j} - \mbb{E}[(u_1-u_0)I_{\cup_{j\geq 1}B_j}] \leq  \left|\mbb{E}[(u_1-u_0)I_{\cup_{j\geq 1}B_j}] \right| \leq \tilde{\delta}_0,
        $$
        hence by (\ref{eq:bernstein-bdd}) again,
        $$
            \mbb{P}(\nu_n(u_1-u_0) I_{\cup_{i\geq 1} B_i} > \eta_0) \leq \exp\left(-\frac{\eta_0^2}{2(\tilde{\delta}_0^2+\tilde{\delta}_0\eta_0/3n^{1/2})}\right).
        $$
        Recall from (\ref{eq:eta,a}), for $j=0,\dots,N-1$,
        $$
            \eta_j:= 4\tilde{\delta}_jk^{-1/2}\left(\sum_{i\leq j+1}H_{[]}(\tilde{\delta}_i,\tilde{\mcal{Z}},\|\cdot\|_2)\right)^{1/2}.
        $$
        Now, by $\eta_0\leq\sum \eta_j \leq kM/8= k^2n^{1/2}t^2/8$, we have
        \begin{align*}
            \frac{\eta_0^2}{2(\tilde{\delta}_0^2+\tilde{\delta}_0\eta_0/3n^{1/2})} \geq  &\,\frac{3\eta_0^2}{2(3\tilde{\delta}_0^2+\tilde{\delta}_0k^2t^2/8)} \\
            \geq &\, \frac{16}{6 + k^2t^2/4\tilde{\delta}_0} \cdot\frac{\sum_{j\leq 1}H(\tilde{\delta}_j,\tilde{\mcal{Z}},\|\cdot\|_2)}{k} \\
            \geq &\, 2\sum_{j\leq 1}\frac{H(\tilde{\delta}_j,\tilde{\mcal{Z}},\|\cdot\|_2)}{k},
        \end{align*}
        where the last step is due to $\tilde{\delta}_0\geq 2\sqrt{2}\exp(\tau/2)s=k^2t^2/8$.
        For $j\geq 1$, 
        \begin{align*}
            \frac{\eta_j^2}{2(\tilde{\delta}_j^2+a_{j+1}\eta_j/3n^{1/2})} \geq &\, \frac{3\eta_j^2}{6\tilde{\delta}_j^2 + 2a_{j+1}\eta_j/\sqrt{n}} \\
            \geq &\, \frac{3\eta_j^2}{8\tilde{\delta}_j^2} \geq 2\sum_{i\leq j+1}\frac{H(\tilde{\delta}_i,\tilde{\mcal{Z}},\|\cdot\|_2)}{k},
        \end{align*}
        which follows directly from the definition of $a_{j+1}$ and $\eta_j$.
        Now, 
        \begin{align}
            \mbb{P}_2 := &\, \sum_{j=0}^{N-1} \prod_{l_1=0}^j |\mcal{F}_{l_1}|\prod_{l_2=0}^{j+1} |\mcal{F}_{l_2}| \sup_{h(\theta,\theta^*)\leq t}\mbb{P}(\nu_n(u_{j+1}-u_j) I_{\cup_{i\geq j+1} B_i} > \eta_j) \nonumber\\
            \leq &\, \exp\left(2\sum_{j\leq 1}H(\tilde{\delta}_j,\tilde{\mcal{Z}},\|\cdot\|_2) - \frac{\eta_0^2}{2(\tilde{\delta}_0^2+\tilde{\delta}_0\eta_0/3n^{1/2})} \right) \nonumber\\
            +&\, \sum_{j=1}^{N-1} \exp\left(2\sum_{i\leq j+1}H(\tilde{\delta}_i,\tilde{\mcal{Z}},\|\cdot\|_2) - \frac{\eta_j^2}{2(\tilde{\delta}_j^2+a_{j+1}\eta_j/3n^{1/2})} \right)  \nonumber\\
            \leq &\, \sum_{j=0}^{N-1} \exp\left(-2\frac{1-k}{k} \sum_{i\leq j+1}H(\tilde{\delta}_i,\tilde{\mcal{Z}},\|\cdot\|_2) \right) \nonumber\\
            \leq &\, \sum_{j=0}^{N-1} \exp(-2(1-k)4^j\Psi(M,t^2,n)) \leq N(t)\exp(-(1-k)\Psi(M,t^2,n)).\label{eq:uniform_series_bound}
        \end{align}
        \item For $\mbb{P}_3$ and $\mbb{P}_4$, we show that they are both $0$. 
        $$
        \mbb{P}_3 :=  \sum_{j=0}^{N-1} \mbb{P}^* \left(\sup_{h(\theta,\theta^*)\leq t} \nu_n(\tilde{Z}_{\theta} - u_j)I_{B_j} > \eta_j\right).
        $$
        Note that by the definition of $\nu_n$
        $$
            \nu_n(\tilde{Z}_{\theta} - u_j)I_{B_j} \leq n^{-1/2}\sum_{j=1}^n (\tilde{Z}_\theta(\bmX_i)-u_j(\bmX_i))I_{B_j} + n^{1/2}\sup_{h(\theta,\theta^*)\leq t} \mbb{E}[(u_j-l_j)I_{B_j}].
        $$
        On $B_j$, $u_j-l_j\geq a_{j+1}$, thus $\mbb{P}(B_j) \leq \mbb{E}[(u_j-l_j)^2]/a_{j+1}^2 \leq \tilde{\delta}_j^2/a_{j+1}^2$,
        and 
        \begin{align*}
            \sup_{h(\theta,\theta^*)\leq t} \mbb{E}[(u_j-l_j)I_{B_j}] = &\, \sup_{h(\theta,\theta^*)\leq t} \left(\mbb{E}[(u_j-l_j)^2]\mbb{E}[I_{B_j}]\right)^{1/2} \\
            \leq  &\, \frac{\tilde{\delta}_j^2}{a_{j+1}} \leq \frac{\eta_j}{\sqrt{n}}.
        \end{align*}
        Hence $\nu_n(\tilde{Z}_{\theta} - u_j)I_{B_j} \leq \eta_j$, and $\mbb{P}_3=0$.
        \item Similarly,
        \begin{align*}
            \nu_n(\tilde{Z}_{\theta} - u_N)I_{B_N} \leq  &\, n^{1/2} \sup_{h(\theta,\theta^*)\leq t} \mbb{E}[(u_N-l_N)I_{B_N}] \\
            \leq &\, n^{1/2}\tilde{\delta}_N = 2\sqrt{2}n^{1/2}\exp(\tau/2)s(t) = \frac{kM}{8},
        \end{align*}
        therefore, $\mbb{P}_4=0$.
    \end{itemize}
    Sum $\mbb{P}_1$ through $\mbb{P}_4$ completes the proof.
\end{proof}

\begin{corollary}\label{cor:asymptotic_lemma7}
    With the same definitions as in Lemma \ref{lemma:wong7}, if in addition
    \begin{equation}\label{eq:reasonable_asymp}
        \exp(-(1-k)\Psi(M,t^2,n))\leq \frac{1}{2},
    \end{equation}
    then
    $$
        \mbb{P}^*\left(\sup_{\substack{h(\theta,\theta^*)\leq t \\ \theta\in\Omega_n}} \nu(\tilde{Z}_{\theta})\geq M \right)\leq 2\exp(-(1-k)\Psi(M,t^2,n)).
    $$
\end{corollary}
\begin{proof}
    Following the same proof as in Lemma \ref{lemma:wong7}, but (\ref{eq:uniform_series_bound}) can be instead bounded by
    \begin{align*}
        &\, \sum_{j=0}^{N-1} \exp(-2(1-k)4^j\Psi(M,t^2,n))\\
        \leq &\, \exp(-(1-k)\Psi(M,t^2,n))^2\left( \sum_{j=0}^{\infty} \exp(-(1-k)4j\Psi(M,t^2,n)) \right) \\
        \leq &\, \exp(-(1-k)\Psi(M,t^2,n)) \times \frac{1}{2} \times \frac{16}{15} \\
        \leq &\, \exp(-(1-k)\Psi(M,t^2,n)).
    \end{align*}
    due to the additional assumption (\ref{eq:reasonable_asymp}).
\end{proof}

\begin{proof}[Proof for Lemma 8]
By the choice of $\tau$ and $k$ we use later, we see that $t^2/2^{10} < s(t)$, thus (19) implies (\ref{eq:wong7entropy}), treating $\phi_n$ as $t$. 
If (\ref{eq:wong7entropy}) holds for some $t>0$, then it holds for any $\tilde{t}\geq t$. 
Thus, if (19) holds, we may apply the result of Lemma \ref{lemma:wong7} on any $\tilde{t}\geq t$, then 
\begin{align*}
    \mbb{P}^*\left(\sup_{\substack{h(\theta,\theta^*)\leq \tilde{t} \\ \theta\in\Omega_n}} \nu(\tilde{Z}_{\theta})\geq M(\tilde{t}) \right)\leq &\, (1+N(\tilde{t}))\exp(-(1-k)\Psi(M,\tilde{t}^2,n)) \\
    \leq &\, (1+N(t))\exp\left(-\frac{(1-k)k^2}{2^5c_0+4k}n\tilde{t}^2\right).
\end{align*}
From Lemma 4 \citep{wong1995probability}, 
$$
\mbb{E}\tilde{Z}_{\theta} \leq -2(1-\varrho)h^2(\theta,\theta^*)
$$
where $\varrho:= 2\exp(-\tau/2)/(1-\exp(-\tau/2))^2$.
Let $A_{\tilde{t}}:=\{\theta \in\Omega_n:\tilde{t}\leq h(\theta,\theta^*)\leq \sqrt{2}\tilde{t}\}$, we have for any $\tilde{t}>t$,
$$
\left\{\sup_{A_{\tilde{t}}}\sum_{i=1}^n \log d(\bmx_i,\theta) \geq -2(1-\varrho-k/2) n \tilde{t}^2\right\} \subseteq \left\{\sup_{A_{\tilde{t}}}\nu_n(\tilde{Z}_\theta)\geq k\sqrt{n}\tilde{t}^2\right\}.
$$
Therefore 
\begin{align}
    &\, \mbb{P}\left(\sup_{A_{\tilde{t}}}\prod_{i=1}^n d(\bmx_i,\theta) \geq \exp\left(-2(1-\varrho-k/2)n\tilde{t}^2\right) \right) \\
    \leq &\, \mbb{P}\left( \sup_{A_{\tilde{t}}}\nu_n(\tilde{Z}_\theta)\geq k\sqrt{n}\tilde{t}^2 \right) \nonumber\\
    \leq &\, (1+N(t))\exp\left(-\frac{(1-k)k^2}{2^5c_0+4k}n\tilde{t}^2\right).\nonumber
\end{align}
Let $L(t):=\min \{x\in\mbb{N}:2^{x+1}t^2>1\}$, then
\begin{align}
    &\, \mbb{P}^*\left(\sup_{A^c\cap \Omega_n}\prod_{i=1}^n d(\bmx_i,\theta) \geq \exp\left(-2(1-\varrho-k/2)nt^2\right) \right) \nonumber\\
    = &\, \sum_{j=0}^L \mbb{P}^*\left( \sup_{2^jt^2\leq h(\theta,\theta^*)< 2^{j+1}t^2} \prod_{i=1}^n d(\bmx_i,\theta) \geq \exp\left(-2(1-\varrho-k/2)nt^2\right) \right) \nonumber\\
    \leq &\, \sum_{j=0}^L \mbb{P}^*\left( \sup_{2^jt^2\leq h(\theta,\theta^*)< 2^{j+1}t^2} \prod_{i=1}^n d(\bmx_i,\theta) \geq \exp\left(-2(1-\varrho-k/2)n2^{j+1} t^2\right) \right)\nonumber \\
    \leq &\, \sum_{j=0}^L (1+N(t)) \exp\left(- \frac{(1-k)k^2}{2^5c_0+4k}n2^{j}t^2\right)\label{eq:bound_wong_thm1} \\
    \leq &\, (L(t)+1)(N(t)+1)\exp\left(- \frac{(1-k)k^2}{2^5c_0+4k}nt^2\right) .\nonumber
\end{align}
To derive the choice of constants, we follow the suggestion of \cite{wong1995probability} and choose $k=2/3$, $\exp(\tau/2)=5$, then $c_0=25(4-\tau/2)/16\approx 3.74$, $c_1=1/12$, $c_2=4/27(32c_0+8/3)\approx 1.212\times10^{-3}$ and $c=2^{-7.5}\cdot3^{-1.5}/5\approx 2^{-13}$.
Also, for the uniform bound, $N\leq \ceil{8+\log_2(t^{-1})}$ and $L=\ceil{2\log_2(t^{-1})-1}$, in which case $(1+L)(1+N)\leq 20\max\{-\log_2(t),(\log_2(t))^2\}$.

For the last result, note that $\exp(-c_2nt^2)\leq \frac{1}{2}$ implies (\ref{eq:low_entropy_as}), thus by Corollary \ref{cor:asymptotic_lemma7}, we can replace $(1+N)$ by $2$ in (\ref{eq:bound_wong_thm1}) and get
$$
 \sum_{j=0}^L 2 \exp\left(- c_2n2^{j}t^2\right) \leq 4  \exp\left(- c_2nt^2\right).
$$
Finally, relabel $t$ as $\phi_n$ to get the expression in the statement of the Lemma.
\end{proof}

\subsection{Remaining Proofs for Section 3}
\subsubsection{Proof for Lemma 9} 

\begin{proof}
    It is enough to check that a $u$-bracket for $\mcal{K}_{p'}(A)$ is also a $u\sqrt{\frac{1-p}{1-p'}}$-bracketing for $\mcal{K}_{p}(A)$, where $1>p\geq p'$.
    Suppose that $\{[L'_i,U'_i],i=1,\dots,N\}$ is a $u$-bracketing for $\mcal{K}_{p'}(A)$.
    Pick an arbitrary $\theta\in A$, there exists $i\in\{1,\dots,N\}$ such that the $u$-bracket $[L'_i,U'_i]$ forms an envelope for $k_{p'}(\bmx;\theta)$, i.e.,
    $$
    L'_i(\bmx)\leq (1-p')f(\bmx;\theta)+p'g(\bmx)\leq U'_i(\bmx),
    $$
    for which $h(L'_i,U'_i)\leq u$.
    Note that by direct multiplication and addition to the preceding inequality, we can find $L_i$ and $U_i$ such that
    $$
    L_i(\bmx):=\frac{1-p}{1-p'}L'_i(\bmx)+\frac{p-p'}{1-p'}g(\bmx)\leq k_{p}(\bmx;\theta) \leq \frac{1-p}{1-p'}U'_i(\bmx) + \frac{p-p'}{1-p'}g(\bmx)=: U_i(\bmx).
    $$
    Thus
    \begin{align*}
        h^2(L_i,U_i) = &\, \frac{1}{2}\int \left(\sqrt{L_i(\bmx)} - \sqrt{U_i(\bmx)}\right)^2\dd \bmx \\
        \leq &\, \frac{1}{2}\int \left(\sqrt{\frac{1-p}{1-p'}L'_i(\bmx)} - \sqrt{\frac{1-p}{1-p'}U'_i(\bmx)}\right)^2\dd \bmx \\
        = &\, \frac{1-p}{1-p'}\cdot \frac{1}{2}\int \left(\sqrt{L'_i(\bmx)}-\sqrt{U'_i(\bmx)}\right)^2\dd \bmx\\
        = &\, \frac{1-p}{1-p'}u^2,
    \end{align*}
    where the second inequality is by noting that for any $a\geq b\geq0,c\geq 0$, $\sqrt{a+c}-\sqrt{b+c}\leq \sqrt{a}-\sqrt{b}$.
    Thus $\{[L_i,U_i],i=1,\dots,N\}$ is a $u\sqrt{\frac{1-p}{1-p'}}$-bracketing for $\mcal{K}_{p}(A)$.
    
    For the second statement, note that $H_{[]}(u,\mcal{F})$ is decreasing as $u$ increases and $\mcal{F}(A)=\mcal{K}_0(A)$.
\end{proof}

\subsubsection{Proof for Lemma 11} 
\begin{proof}
    Recall that 
    $$
    \rho_{p,\alpha}(\theta^*,\theta) := \frac{1}{\alpha}\int \left[\left(\frac{k_p(\bmx;\theta^*)}{k_p(\bmx;\theta)}\right)^\alpha - 1\right] k_p(\bmx;\theta)\dd x.
    $$
    Note that the above expression is a special case of the $f$-divergence with convex function
    $$f(u) = \frac{1}{\alpha}\left(u^{-\alpha}-1\right),$$
    and to emphasize this fact, we will rewrite the expression as
    $$
    \rho_\alpha\left(k_p(\bmx;\theta^*) \| k_p(\bmx;\theta)\right) := \frac{1}{\alpha}\int \left[\left(\frac{k_p(\bmx;\theta^*)}{k_p(\bmx;\theta)}\right)^\alpha - 1\right] k_p(\bmx;\theta)\dd x.
    $$
    By joint convexity of $f$-divergence,
    \begin{align*}
        \rho_{p,\alpha}(\theta^*, \theta) &\, = \rho_\alpha\left.\Bigl( (1-p)f(\bmx;\theta^*)+p\,g(\bmx) \right\| (1-p)f(\bmx;\theta) + p\,g(\bmx) \Bigr) \\
        &\, \leq  (1-p)\rho_\alpha\left.\bigl(f(\bmx;\theta^*) \right\| f(\bmx;\theta)\bigr) + p\, \rho_\alpha\left.\bigl(g(\bmx) \right\| g(\bmx)\bigr) \\
        &\, = (1-p)\rho_\alpha\left.\bigl(f(\bmx;\theta^*) \right\| f(\bmx;\theta)\bigr) \\
        &\, \leq \rho_{0,\alpha}\left(\theta^*, \theta \right).
    \end{align*}
    Thus $S_{0,\alpha}(t)\subseteq S_{p,\alpha}(t)$.
\end{proof}

\subsubsection{Proof for Lemma 13}

\begin{proof}
Note that $m_n(A_n^c,\bm{X}) = m_n(A_n^c\cap \Omega_n,\bm{X}) + m_n(A_n^c\cap \Omega_n^c, \bm{X})$.
From Lemmas 8 and 9, we have  
$$
\mbb{P}^*_{\theta^*}\left(\sup_{\theta\in A_n^c\cap \Omega_n} \prod_{i=1}^n d(\bmx_i,\theta) \geq e^{-c_1 n\phi_n^2}\right) \leq C(n,\phi_n)e^{-c_2 n\phi_n^2}.
$$
Also, by
$$
m_n(A_n^c\cap \Omega_n, \bmX) \leq \sup_{\theta\in A_n^c\cap \Omega_n} \prod_{i=1}^n d(\bmx_i,\theta),
$$
we have
$$
\mbb{P}\left(m_n(A_n^c\cap \Omega_n, \bmX) \geq e^{-c_1 n\phi_n^2}\right) \leq C(n,\phi_n)e^{-c_2 n\phi_n^2}.
$$
By Markov's inequality, Fubini's Theorem and Assumption 3, we have 
\begin{align*}
\mbb{P}\left(m_n(A_n^c\cap \Omega_n^c,\bm{X})\geq e^{- c_1 n \phi_n^2}\right) \leq &\, e^{ c_1 n \phi_n^2}\iint_{A_n^c\cap \Omega_n^c} \prod_{i=1}^n d(\bmx_i;\theta)\pi_0(\theta)\dd\theta \dd \mbb{P}_{\theta^*}(\bm{X})\\
= &\, e^{c_1 n \phi_n^2}\int_{A_n^c\cap \Omega_n^c} \int\prod_{i=1}^n d(\bmx_i;\theta)\dd \mbb{P}_{\theta^*}(\bm{X}) \pi_0(\theta)\dd\theta \\
= &\, e^{c_1 n\phi_n^2} \mbb{P}_{\pi_0}(A_n^c\cap \Omega_n^c) \leq c_4 e^{- c_1 n\phi_n^2}.
\end{align*}
Thus 
$$
\mbb{P}\left(m_n(A_n^c,\bm{X})\geq 2e^{- c_1 n\phi_n^2}\right) \leq C(n,\phi_n)e^{-c_2 n\phi_n^2} + c_4 e^{-n c_1\phi_n^2}.
$$
\end{proof}

\subsubsection{Proof for Lemma 14} 

\begin{proof}
Since $f(\bmx;\theta)$ has all directional derivatives, let $\phi(\alpha) = \theta+\alpha(\theta^*-\theta)$ for some $\alpha\in(0,1)$.
Note that 
$$
    \log k_{p_n}(\bmx;\theta) - \log k_{p_n}(\bmx;\theta^*) \leq \sup_{\alpha\in(0,1)} \| \tilde{\grad}_{\theta} \log k_{p_n}(\bmx;\phi(\alpha))\|_2 \|\theta-\theta^*\|_2,
$$
where $\tilde{\grad}$ operator is defined in Definition 7. 
Note that
\begin{align*}
    \|\tilde{\grad}_\theta \log k_{p_n}(\bmx;\theta)\|_2 = & \left\|\frac{(1-p_n)\tilde{\grad}_\theta f(\bmx;\theta)}{(1-p_n)f(\bmx;\theta) + p_ng(\bmx)} \right\|_2\\
    \leq & \frac{1-p_n}{p_n}\left\|\frac{\tilde{\grad}_\theta f(\bmx;\theta)}{g(\bmx)}\right\|_2.
\end{align*}
Since by Assumption 4, $\|\tilde{\grad}_\theta f(\bmx;\theta)/g(\bmx)\|_2$ is bounded as $\|\bmw\|_2\rightarrow\infty$ in all directions, hence $\tilde{\grad} f(\bmx;\theta)/g(\bmx)$ is bounded on $\mbb{R}^d\times A$ for any compact set $A\subset\Omega$. 
Thus there exists a sequence $T_n<\infty$ such that,
$$
\sup_{x\in\mbb{R}^d}\sup_{\theta\in A_n}\left\|\frac{\tilde{\grad}_\theta f(\bmx;\theta)}{g(\bmx)}\right\|_2 \leq T_n.
$$
Thus
$$
|\log d(\bmx;\theta)| \leq \frac{1-p_n}{p_n}T_n\psi_n,
$$
and 
$$
\eta_n=\sup_{\bmx,\bmz\in\mbb{R}^d}\sup_{\theta\in A_n}\frac{d(\bmx;\theta)}{d(\bmz;\theta)} \leq e^{2(1-p_n)T_n\psi_n/p_n}.
$$
\end{proof}

\section{Theoretical Application on Regression Settings}\label{appx:proof_examples}

In this section, we will verify the assumptions in Theorem 4 for linear regression, logistic regression, and Cauchy regression.
The analysis is presented first, followed by technical lemmas.

\subsection{Linear Regression Case}
Let
$$
X_i = \theta^\top \bm{W}_i + e_i, \qquad i=1,\dots,n
$$
where $e_i$ are assumed to be i.i.d. $\mcal{N}(0,\sigma^2)$ random variables with unknown variance $\sigma^2$ and $W_i$ are $d$-dimensional covariates with $\|W_i\|_\infty\leq 1$.
The likelihood function of $X_i$ given $\theta,\sigma,W_i$ can be written as
$$
f_{\theta,\sigma}(x|\bm{w}) = (2\pi\sigma^2)^{-1/2} \exp\left(-(x-\theta^{\top}\bm{w})^2/(2\sigma^2)\right).
$$

\vspace{0.5em}
\noindent
Firstly, we shall address condition (8) in Assumption 1.  
To avoid singularity, we choose 
$$\Omega_n:=\{(\theta,\sigma) \in \mbb{R}^d: \|\theta^*-\theta\|_2\leq R_n,\, |\sigma-\sigma^*|<R_n, \sigma>1/R_n\},$$ 
where $R_n\rightarrow\infty$ is defined later.

\noindent
By Lemma \ref{lemma:gaussian_lipschitz}, we have for any $(\theta,\sigma),(\vartheta,\varsigma)\in\Omega_n$, $\|(\theta,\sigma)-(\vartheta,\varsigma)\|_\infty\leq r$,
$$
    \left|\sqrt{f_{\theta,\sigma}(x|\bm{w})} - \sqrt{f_{\vartheta,\varsigma}(x|\bm{w})}\right| \leq M^{(n)}_{\vartheta,\varsigma,r}(x,\bm{w}) \|\theta-\vartheta\|_\infty,
$$
with 
\begin{multline*}
    M^{(n)}_{\vartheta,\varsigma,r}(x,\bm{w}) = \frac{R_n^2}{2}\Bigl(d\|\bm{w}\|_2+R_n(|x|+M_n\|\bm{w}\|_2)\Bigr)(|x|+M_n\|\bmw\|_2)\\
    \left(\frac{R_n^2}{2\pi}\right)^{1/4}\exp\left(-\frac{1}{4R_n^2}\min_{\varphi\in B_r(\vartheta)}(x-\varphi^\top \bm{w})^2\right)
\end{multline*}
where $M_n = \|\theta^*\|_2 + R_n$.
Then by Lemma \ref{lemma:gaussian_M2}, $\|M^{(n)}_{\vartheta,\varsigma,r}\|^2_2\preceq M_{n}^{12}$, $\forall (\vartheta,\varsigma)\in\Omega_n$ which satisfies the requirements of Lemma 3.  
Since $\log M_{r}^{(n)}=O(\log M_n)$, the condition (8) is satisfied by taking 
$$
\phi_n=n^{-1/q}, \quad M_n = n^{k},
$$
for some $q>4$ and $k>0$ to be chosen to satisfy Assumptions 3. 

It is not hard to check that the Gaussian likelihood function satisfies the condition in Lemma 6, when $(\theta,\sigma)$ is close enough to $(\theta^*,\sigma^*)$, i.e., within $L_2$-distance of some constant $c_r$. 
Thus $S_{0,\alpha}(t_n)$ contains an $L_2$-neighbourhood around $\theta^*$ with its radius proportional to $t_n$ up to a constant depending on $(\theta^*,\sigma^*)$ and $c_r$.
In addition, Assumption 2 is satisfied by applying Lemma 7 and Remark 12. 

Moreover, for $R_n$ large, by Remark 13 
$$
\mbb{P}_{\pi_0}(\Omega_n^c) = 1-\mbb{P}_{\pi_0}(\Omega_{n}) \sim R_{n}^{d-2}e^{-R_{n}^2/2}.
$$
Recall that $M_n=R_n+\|\theta^*\|_2$ and $\phi_n=n^{-1/q}$, this gives the restriction that
$$
M_n^2 = n^{2k} \approx R_n^2 \geq 4c_1n\phi_n^2 = 4c_1n^{(q-2)/q},
$$
which can be satisfied by choosing $k  > (q-2)/2q$.
Finally, Lemma \ref{lemma:gaussian_hellinger} shows that the pair of sequences $\phi_n$ and $\psi_n$ in Proposition 3 can be chosen (up to a multiplicative constant) to diminish at the same rate when $(\theta,\sigma)$ is close to $(\theta^*,\sigma^*)$. 
Therefore $\psi_n$ can also be chosen to be $\propto \phi_n=n^{-1/q}$.
By choosing $p_n$ to converge at a rate slower than $n^{-1/q}$, e.g., $p_n=n^{-1/2q}$, then
$$
\epsilon_n = O(n^{-1/2q}),\qquad \delta_n = O(\exp(-\tfrac{1}{2}c_1 n^{1-2/q})),
$$
where $c_1$ is a fixed constant and $q>4$.

\begin{lemma}\label{lemma:gaussian_lipschitz}
Let 
$$
f_{\theta,\sigma}(x|\bm{w}) := (2\pi\sigma^2)^{-1/2}\exp(-(x-\theta^\top \bmw)^2/(2\sigma^2)),
$$
where $(\theta,\sigma)\in\Omega_n:=\{(\theta,\sigma)\in\Omega:\|\theta-\theta^*\|_2\leq R_n, |\sigma^*-\sigma|\leq R_n, \sigma >1/R_n\}$, for some $R_n>0$.
If $(\theta,\sigma),(\vartheta,\varsigma)\in\Omega_n$ with $\|(\theta,\sigma)-(\vartheta,\varsigma)\|_\infty\leq r$, then
$$
\left|\sqrt{f_{\theta,\sigma}(x|\bm{w})} - \sqrt{f_{\vartheta,\varsigma}(x|\bm{w})}\right| \leq M_{\vartheta,\varsigma,r}(x,\bmw) \|(\theta,\sigma)-(\vartheta,\varsigma)\|_\infty,
$$
where the function $M_{\vartheta,\varsigma,r}$ is given by 
\begin{multline*}
    M_{\vartheta,\varsigma,r}(x,\bmw) := \frac{R_n^2}{2}\Bigl(d\|\bmw\|_2+R_n(|x|+M_n\|\bmw\|_2)\Bigr)(|x|+M_n\|\bmw\|_2)\\
    \left(\frac{R_n^2}{2\pi}\right)^{1/4}\exp\left(-\frac{1}{4R_n^2}\min_{\varphi\in B_r(\vartheta)}(x-\varphi^\top \bmw)^2\right),
\end{multline*}
where $M_n = \|\theta^*\|_2 + R_n$.
\end{lemma}
\begin{proof}
Recall that
$$
f_{\theta,\sigma}(x|\bm{w}) = (2\pi\sigma^2)^{-1/2}\exp(-(x-\theta^\top \bmw)^2/(2\sigma^2)),
$$
Then, by Cauchy's Mean Value Theorem and the Cauchy-Schwarz inequality, there exists some $a\in(0,1)$, $\varphi=\vartheta+a(\theta-\vartheta)$, $\zeta = \varsigma+a(\sigma-\varsigma)$, such that
\begin{align*}
    &\, \left|\sqrt{f_{\theta,\sigma}(x|\bm{w})} - \sqrt{f_{\vartheta,\varsigma}(x|\bm{w})}\right| \\
    =&\, \frac{x-\varphi^\top \bmw}{2\zeta^2}f_{\varphi,\zeta}^{\frac{1}{2}}(x|\bm{w})\,(\theta-\vartheta)^\top \bmw + \frac{1}{2}\left(\frac{(x-\varphi^\top \bmw)^2}{\zeta^3}-\zeta^{-1}\right)f_{\varphi,\zeta}^{\frac{1}{2}}(x|\bm{w})\,(\sigma-\varsigma) \\
    \leq &\, \frac{1}{2\zeta^2}\left((x-\varphi^\top \bmw)d\|\bmw\|_2 + \frac{(x-\varphi^\top \bmw)^2}{\zeta}\right) f_{\varphi,\zeta}^{\frac{1}{2}}(x|\bm{w}) \| (\theta,\sigma) - (\vartheta,\varsigma)\|_\infty\\
    \leq &\, \frac{R_n^2}{2}\Bigl(d\|\bmw\|_2+R_n(|x|+M_n\|\bmw\|_2)\Bigr)(|x|+M_n\|\bmw\|_2) \\
    &\, \times \left(\frac{R_n^2}{2\pi}\right)^{1/4}\exp\left(-\frac{1}{4R_n^2}\min_{\varphi\in B_r(\vartheta)}(x-\varphi^\top \bmw)^2\right) \| (\theta,\sigma)- (\vartheta,\varsigma)\|_\infty.
\end{align*}
The last step is by noting the following upper bounds since $(\varphi,\zeta)\in\Omega_n$
$$
|x-\varphi^\top \bmw|\leq |x|+M_n\|\bmw\|_2,\quad \frac{1}{\zeta} \leq R_n.
$$
\end{proof}

\begin{lemma}\label{lemma:gaussian_moment}
    For constants $R,\alpha\in(0,M]$, $k\in\mbb{Z}_+$,
    $$
    \int_0^\infty x^k \frac{1}{\sqrt{2\pi R^2}}\exp\left(-\frac{(x-\alpha w)^2}{2R^2}\right)\dd x \leq \beta_k(w) M^k,
    $$
    where 
    $$\beta_k(w)=\frac{1}{\sqrt{\pi}}\left[ 2\sum_{j=0}^{\floor{\frac{k}{2}}} \begin{pmatrix}
    k\\2j
    \end{pmatrix} 2^{j-1}w^{k-2j}\Gamma(\tfrac{2j+1}{2}) + \sum_{j=1}^{\ceil{\frac{k}{2}}} \begin{pmatrix}
        k\\2j-1
    \end{pmatrix} 2^{\frac{2j-1}{2}-1} w^{k-2j+1}\Gamma(j)\right].$$
\end{lemma}
\begin{proof}
First, we compute the following integral
\begin{align*}
    &\, \int_0^\infty \frac{w^k}{\sqrt{2\pi R^2}}\exp\left(-\frac{w^2}{2R^2}\right)\dd w & \\
    = &\, \frac{1}{\sqrt{2\pi}}\int_0^\infty (2R^2t)^{\frac{k-1}{2}}R\exp(-t)\dd t & [t=\tfrac{w^2}{2R^2}] \\
    = &\, \frac{1}{\sqrt{\pi}}2^{k/2-1}R^k \Gamma(\tfrac{k+1}{2}).
\end{align*}
Therefore,
\begin{align*}
    &\, \int_0^\infty x^k \frac{1}{\sqrt{2\pi R^2}}\exp\left(-\frac{(x-\alpha w)^2}{2R^2}\right)\dd x & \\
    = &\, \int_{-\alpha w}^\infty (t+\alpha w)^k \frac{1}{\sqrt{2\pi R^2}}\exp\left(-\frac{t^2}{2R^2}\right)\dd t & [t=x-\alpha w] \\
    \leq &\, 2\sum_{j=0}^{\floor{\frac{k}{2}}} \begin{pmatrix}
    k\\2j
    \end{pmatrix} \int_0^\infty t^{2j} (\alpha w)^{k-2j} \frac{1}{\sqrt{2\pi R^2}} \exp\left(-\frac{t^2}{2R^2}\right)\dd t \\
    &\, +\, \sum_{j=1}^{\ceil{\frac{k}{2}}} \begin{pmatrix}
        k\\2j-1
    \end{pmatrix} \int_0^\infty t^{2j-1}(\alpha w)^{k-2j+1}\frac{1}{\sqrt{2\pi R^2}} \exp\left(-\frac{t^2}{2R^2}\right)\dd t \\
    \leq &\, \beta_k(w) M^k &
\end{align*}
where
$$
\beta_k(w) = \frac{1}{\sqrt{\pi}}\left[ 2\sum_{j=0}^{\floor{\frac{k}{2}}} \begin{pmatrix}
    k\\2j
    \end{pmatrix} 2^{j-1}w^{k-2j}\Gamma(\tfrac{2j+1}{2}) + \sum_{j=1}^{\ceil{\frac{k}{2}}} \begin{pmatrix}
        k\\2j-1
    \end{pmatrix} 2^{\frac{2j-1}{2}-1} w^{k-2j+1}\Gamma(j)\right].
$$
\end{proof}

\begin{lemma}\label{lemma:gaussian_M2}
Define 
$$\Omega_n:=\{(\theta,\sigma)\in\Omega:\|\theta-\theta^*\|_2\leq R_n, |\sigma^*-\sigma|\leq R_n, \sigma >1/R_n\},$$ 
for some $R_n>0$.
For $(\vartheta,\varsigma)\in\Omega_n$, let
\begin{multline*}
    M_{\vartheta,\varsigma,r}(x,\bm{w}) := \frac{R_n^2}{2}\Bigl(d\|\bmw\|_2+R_n(|x|+M_n\|\bmw\|_2)\Bigr)(|x|+M_n\|\bmw\|_2)\\
    \left(\frac{R_n^2}{2\pi}\right)^{1/4}\exp\left(-\frac{1}{4R_n^2}\min_{\varphi\in B_r(\vartheta)}(x-\varphi^\top \bmw)^2\right),
\end{multline*}
where $M_n = \|\theta^*\|_2+R_n$.
If $\|\bm{w}\|_\infty\leq 1$, then
$$\|M_{\vartheta,\varsigma,r}(x,\bm{w})\|_2^2 \preceq M_n^{12}.$$
\end{lemma}

\begin{proof}
$$
\|M_{\vartheta,\varsigma,r}(x,\bmw)\|_2^2 = \int M_{\vartheta,\varsigma,r}^2(x,\bmw)\dd x.
$$
Consider the value of
$$
\min_{\varphi\in B_r(\vartheta)}(x-\varphi^\top \bmw).
$$
Note that $\varphi = \vartheta+\tilde{r}\bm{v}$ for some $\tilde{r}<r$ and unit vector $\bm{v}$.
When $\bmw$ is fixed, $\varphi^\top \bmw$ takes value within the interval $(u_-,u_+)$ where
$$
u_-:= \vartheta^\top \bmw - r\|\bmw\|_2, \quad u_+ := \vartheta^\top \bmw + r\|\bmw\|_2,
$$
with the infimum and the supremum attained at $\vartheta-\frac{r\bmw}{\|\bmw\|_2}$ and $\vartheta+\frac{r\bmw}{\|\bmw\|_2}$ respectively.
Then
$$
\min_{\varphi\in B_r(\vartheta)} (x-\varphi^\top \bmw)^2 = \begin{dcases}
    (x-\vartheta^\top \bmw+r\|\bmw\|_2)^2, & x<u_-;\\
    0, & u_-\leq x\leq u_+;\\
    (x-\vartheta^\top \bmw-r\|\bmw\|_2)^2, & x>u_+;
\end{dcases}
$$
Write
$$
\int_{-\infty}^\infty M_{\vartheta,\varsigma,r}^2(x,\bmw)\dd x = \int_{-\infty}^{u_-}+\int_{u_-}^{u_+}+\int_{u_+}^\infty M_{\vartheta,\varsigma,r}^2(x,\bmw)\dd x. 
$$
Firstly, by noting that $|u_+|\leq M_n\|\bmw\|_2$
\begin{align*}
&\, \int_{u_-}^{u_+} M_{\vartheta,\varsigma,r}^2(x,\bmw)\dd x\\
= &\, \frac{1}{\sqrt{2\pi}}\frac{R_n^5}{4} \int_{u_-}^{u_+} \Bigl(d\|\bmw\|_2+R_n(|x|+M_n\|\bmw\|_2)\Bigr)^2(|x|+M_n\|\bmw\|_2)^2 \exp\left(-\frac{1}{2R_n^2}\right)\dd x \\
\leq &\, \frac{1}{\sqrt{2\pi}}\frac{M_n^5}{4}(d+2M_n^2)^2\|\bmw\|_2^2(2M_n\|\bmw\|_2)^2 2r\|\bmw\|_2 \\
\preceq & M_n^{11}.
\end{align*}
Note that $\int_{-\infty}^{u_-} M_{\vartheta,\varsigma,r}^2(x,\bmw)\dd x$ is comprised of summing integrals of form
$$
\int_{-\infty}^{u_-} \frac{1}{\sqrt{2\pi R_n^2}}|x|^k \exp\left(-\frac{(x-u_-)^2}{2R_n^2}\right)\dd x.
$$
By Lemma \ref{lemma:gaussian_moment},
\begin{align*}
    &\, \int_{-\infty}^{u_-} \frac{1}{\sqrt{2\pi R_n^2}}|x|^k \exp\left(-\frac{(x-u_-)^2}{2R_n^2}\right)\dd x \\
    \leq &\, \int_{-\infty}^{\infty} \frac{1}{\sqrt{2\pi R_n^2}}|x|^k \exp\left(-\frac{(x-u_-)^2}{2R_n^2}\right)\dd x \\
    \leq &\, 2\int_{0}^{\infty} \frac{1}{\sqrt{2\pi R_n^2}}x^k \exp\left(-\frac{(x-|u_-|)^2}{2R_n^2}\right)\dd x \\
    \preceq &\, 2\beta_k(\|\bmw\|_2)M_n^k.
\end{align*}
Since $|u_-|,|u_+|\leq M_n\|\bmw\|_2$, the integrals
$\int_{-\infty}^{u_-} M_{\vartheta,\varsigma,r}^2(x,\bmw)\dd x$ and $\int_{u_-}^{\infty} M_{\vartheta,\varsigma,r}^2(x,\bmw)\dd x$ share the same upper bound as above (see by change of variable $w=-x$).
Then $$\|M_{\vartheta,\varsigma,r}^2(x,\bmw)\|_2^2\preceq M_n^{12}.$$

\end{proof}

\begin{lemma} \label{lemma:gaussian_hellinger}
Let $\|\bmw\|_\infty\leq 1$ with nonzero entries,
$$
f_{\theta,\sigma}(x|\bm{w}) := (2\pi\sigma^2)^{-1/2}\exp\left(-\frac{(x-\theta^\top \bmw)^2}{2\sigma^2}\right).
$$
Then we have for $(\theta,\sigma)$ close to $(\theta^*,\sigma^*)$, there exists a positive constant $C$ independent of  $(\theta,\sigma)$ such that
$$
    h((\theta^*,\sigma^*),(\theta,\sigma)) \leq \phi_n \implies \|(\theta^*,\sigma^*)-(\theta,\sigma)\|_2 \leq C\phi_n.
$$
\end{lemma}
\begin{proof}
    Recall that by Cauchy's Mean Value Theorem, there exists $\varphi= \theta^*+\alpha(\theta-\theta^*)$, $\xi=\sigma^*+\alpha(\sigma-\sigma^*)$, $\alpha\in(0,1)$ such that
    \begin{align*}
        & \, f^{\frac{1}{2}}_{\theta,\sigma}(x|\bm{w}) - f^{\frac{1}{2}}_{\theta^*,\sigma^*}(x|\bm{w}) \\
        = &\,  f_{\varphi,\xi}^{\frac{1}{2}}(x|\bm{w})\left[\frac{x-\varphi^\top \bmw}{2\xi^2}(\theta-\theta^*)^\top \bmw + \frac{1}{2}\left(\frac{(x-\varphi^\top \bmw)^2}{\xi^3}-\xi^{-1}\right)(\sigma-\sigma^*)\right] .
    \end{align*}
Let $z:= (x-\varphi^\top \bmw)/\xi$, then
\begin{align*}
    h^2((\theta^*,\sigma^*),(\theta,\sigma)) := &\, \frac{1}{2}\int \left( f^{\frac{1}{2}}_{\theta,\sigma}(x|\bm{w}) - f^{\frac{1}{2}}_{\theta^*,\sigma^*}(x|\bm{w})\right)^2 \dd x \\
    =&\, \frac{1}{8\xi^2}\int f_{\varphi,\xi}(x|\bm{w})\left[z(\theta-\theta^*)^\top \bmw + \left(z^2-1\right)(\sigma-\sigma^*)\right]^2 \dd x.
\end{align*}
Note that in the above integral with respect to $x$, $z$ is a standard normal random variable and $f_{\varphi,\xi}(x|\bm{w}) \propto f_z(z)$.
There exists a constant $\lambda$ depending only on the (maximum) distance between $(\theta,\sigma)$ and $(\theta^*,\sigma^*)$ such that
$$
  h^2((\theta^*,\sigma^*),(\theta,\sigma)) \geq \lambda \|(\theta,\sigma)-(\theta^*,\sigma^*)\|_2^2.
$$
\end{proof}

\subsection{Logistic Regression Case}

Consider the following linear logistic regression model,
$$
X_i \sim \text{Bern}(p_i), \qquad p_i =\Phi(\theta^{\top}\bm{W}_i), \quad \theta,\bm{W}_i\in\mbb{R}^d,
$$
where $\Phi(z)=(1+e^{-z})^{-1}$.
Again, we assume the covariates satisfy $\|\bm{W}_i\|_\infty\leq 1$, then
$$
f_\theta(x|\bm{w}) = \Phi(\theta^\top \bm{w})^x \Phi(-\theta^\top \bm{w})^{1-x}.
$$
\noindent 
Note that 
$$
\left|\sqrt{f_\theta(x|\bm{w})}-\sqrt{f_\vartheta(x|\bm{w})}\right|\leq \frac{1}{2}\|\bm{w}\|_2\|\theta-\vartheta\|_2.
$$
So we can apply Lemma 2 with $M^{(n)}_{\vartheta,r}(x,\bm{w}) = \frac{1}{2}\|\bm{w}\|_2$. 
Hence $\|M^{(n)}_{\vartheta,r}(x,\bm{w})\|_2^2=O(d)$ is bounded above by a uniform constant since $d$ is fixed.
Thus,
Assumption 1 can be satisfied with 
$$
\phi_n=n^{-1/q}, \quad R_n=n^{k},
$$
for some $q>4$ and $k>0$ to be chosen to satisfy Assumption 3. 
Again, it is not hard to check Lemma 6 holds for this likelihood and thus Assumption 2 is satisfied by applying Lemma 7 and Remark 12. 
Moreover, for $R_n$ large, by Remark 13, 
$$
\mbb{P}_{\pi_0}(\Omega_n^c) = 1-\mbb{P}_{\pi_0}(\Omega_{n}) \sim R_{n}^{d-2}e^{-R_{n}^2/2}.
$$
Recall that $M_n=R_n+\|\theta^*\|_2$ and $\phi_n=n^{-1/q}$, this gives the restriction again that
$$
M_n^2 = n^{2k} \approx R_n^2 \geq 4c_1n\phi_n^2 = 4c_1n^{(q-2)/q},
$$
which can be satisfied by choosing $k>(q-2)/2q$, similar to the previous example.

Finally, Lemma \ref{lemma:logistic_hellinger} shows that the pair of sequences $\phi_n$ and $\psi_n$ in Proposition 4 can be chosen (up to a multiplicative constant) to diminish at the same rate when $\theta$ is close to $\theta^*$. 
Therefore $\psi_n$ can also be chosen to be $\propto \phi_n=n^{-1/q}$.
By choosing $p_n$ to converge at a rate slower than $n^{-1/q}$, e.g., $p_n=n^{-1/2q}$, then
$$
\epsilon_n = O(n^{-1/2q}),\qquad \delta_n =O(\exp(-\tfrac{1}{2}c_1 n^{1-2/q})),
$$
where $c_1$ is a fixed constant and $q>4$.

\begin{lemma} \label{lemma:logistic_hellinger}
Let $\|\bmw\|_\infty\leq 1$ with nonzero entries,
$$
f_\theta(x|\bm{w}) := \Phi(\theta^\top \bmw)^x \Phi(-\theta^\top \bmw)^{1-x},
$$
where $\Phi(x)=(1+e^{-x})^{-1}$.
Then we have for $\theta$ close to $\theta^*$, there exists a positive constant $C$ independent of  $\theta$ such that
$$
    h(\theta^*,\theta) \leq \phi_n \implies \|\theta^*-\theta\|_2 \leq C\phi_n.
$$
\end{lemma}
\begin{proof}
Note that
$$
\grad_\theta f^{\frac{1}{2}}_{\theta}(x|\bm{w}) = \left(\frac{x}{2}\Phi(\theta^\top \bmw)^{\frac{x}{2}}\Phi(-\theta^\top \bmw)^{\frac{3-x}{2}} - \frac{1-x}{2}\Phi(\theta^\top \bmw)^{\frac{x+2}{2}}\Phi(-\theta^{\top}\bmw)^{\frac{1-x}{2}}\right) \bmw.
$$
Then there exists $\varphi = \theta^*+\alpha(\theta-\theta^*)$, $\alpha\in(0,1)$ such that
$$
\sum_{y=0}^1 \left(f^{\frac{1}{2}}_{\theta^*}(x|\bm{w})-f^{\frac{1}{2}}_\theta(x|\bm{w})\right)^2 = \frac{1}{4}\Phi(\varphi^\top \bmw)\Phi(-\varphi^\top \bmw) \left((\theta-\theta^*)^\top \bmw\right)^2 .
$$
Now, by Assumption $\|\bmw\|_\infty\leq 1$ with nonzero entries
$$
h^2(\theta^*,\theta) = \frac{1}{8}\Phi(\varphi^\top \bmw)\Phi(-\varphi^\top \bmw)  \left((\theta-\theta^*)^\top \bmw\right)^2 .
$$
Since $\theta$ is assumed to be close to $\theta^*$, $\varphi$ takes value in a compact set, and $\bmw\|_\infty\leq1$, hence $\exists C>0$ such that
$$
h(\theta^*,\theta) \geq C\|\theta^*-\theta\|_2.
$$
\end{proof}

\subsection{Cauchy Regression Case}

Consider the following Cauchy regression model,
$$
X_i = \theta^\top \bm{W}_i + e_i,
$$
where $e_i\sim \text{Cauchy}(0,\gamma)$ are Cauchy random variables with unknown dispersion parameter $\gamma$.
Assume that the covariates $\bm{W}_i$ satisfy $\|\bm{W}_i\|_\infty\leq1$.
Then the likelihood is given by
$$
f_{\theta,\gamma}(x|\bm{w}) = \frac{\sqrt{\gamma}}{\pi}\left(1+\frac{(x-\theta^\top \bm{w})^2}{\gamma}\right)^{-1}.
$$
\noindent 
By Cauchy's Mean Value Theorem, we can show (as in Lemma \ref{lemma:cauchy_lipschitz} that for any $(\theta,\gamma_1),(\vartheta,\gamma_2)\in\Omega_n:=\{(\vartheta,\gamma)\in\Omega:\|\vartheta-\theta^*\|_2\leq R_n, |\gamma-\gamma^*|<R_n, \gamma>R_n^{-1}]\}$, such that $\|(\theta,\gamma_1)-(\vartheta,\gamma_2)\|_\infty\leq r$,
$$
\left|\sqrt{f_{\theta,\gamma_1}(x|\bm{w})}-\sqrt{f_{\vartheta,\gamma_2}(x|\bm{w})}\right|\leq M^{(n)}_{\vartheta,\gamma_2,r}(x,\bm{w})\|(\theta,\gamma_1)-(\vartheta,\gamma_2)\|_\infty,
$$
where
\begin{multline*}
     M^{(n)}_{\vartheta,\gamma_2,r}(x,\bm{w}) =\frac{d}{\sqrt{\pi}}\left(1+\min_{\varphi\in B_{r}(\vartheta)}(x-\varphi^{\top}\bm{w})^2/R_n\right)^{-\frac{3}{2}}\\
     \cdot \left[R_n^{\frac{3}{4}}(|x|+M_n\|\bm{w}\|_2)\|\bm{w}\|_2+\frac{1}{4}R_n^{\frac{3}{4}}+R_n^{\frac{7}{4}}(|x|+M_n\|\bm{w}\|_2)^2\right],
\end{multline*}
where $M_n = \|\theta^*\|_2 + R_n$, $B_r(\vartheta)$ is the $L_2$-ball around $\vartheta$ and by Lemma \ref{lemma:cauchy_M2}, $\|M_{\vartheta,r}\|_2^2 \preceq M_n^9$. 
Thus, similar to the first example, $\log M_r^{(n)}=O(\log M_n)$ and Assumption 1 can be satisfied by taking 
$$
\phi_n=n^{-1/q}, \quad M_n = n^{k},
$$
for some $q>4$ and $k>0$ to be chosen to satisfy Assumption 3. 
It is not hard to check that Lemma 6 holds for this likelihood and thus Assumption 2 is satisfied by applying Lemma 7 and Remark 12. 
Assumption 3 can be checked using the same procedure as in the linear regression example, which again gives $k>\tfrac{1}{4}$. 

Finally, Lemma \ref{lemma:cauchy_hellinger} shows that the pair of sequences $\phi_n$ and $\psi_n$ in Proposition 3 can be chosen (up to a multiplicative constant) to diminish at the same rate when $\theta$ is close to $\theta^*$. 
Therefore $\psi_n$ can also be chosen to be $\propto \phi_n=n^{-1/q}$.
By choosing $p_n$ to converge at a rate slower than $n^{-1/q}$, e.g., $p_n=n^{-1/2q}$, then
$$
\epsilon_n \sim n^{-1/2q},\qquad \delta_n\sim \exp(-\tfrac{1}{2}c_1 n^{1-2/q}),
$$
where $c_1$ is a fixed constant, and $q>4$.

\begin{lemma}\label{lemma:cauchy_lipschitz}
Let 
$$
f_{\theta,\gamma}(x|\bm{w}) := \frac{\sqrt{\gamma}}{\pi}(1+(x-\theta^\top \bmw)^2/\gamma)^{-1} ,
$$
where $(\theta,\gamma)\in\Omega_n:=\{(\theta,\gamma)\in\Omega:\|\theta-\theta^*\|_2\leq R_n, |\gamma-\gamma^*|<R_n, \gamma>R_n^{-1}]\}$, $R_n>0$.
If $\|(\theta,\gamma_1)-(\vartheta,\gamma_2)\|_\infty\leq r$ then
$$
\left|\sqrt{f_{\theta,\gamma_1}(x|\bm{w})} - \sqrt{f_{\vartheta,\gamma_2}(x|\bm{w})}\right| \leq M_{\vartheta,\gamma_2,r}(x,\bmw) \|(\theta,\gamma_1)-(\vartheta,\gamma_2)\|_\infty,
$$
where the function $M_{\vartheta,\gamma_2,r}$ is given by 
\begin{multline*}
     M_{\vartheta,\gamma_2,r}(x,\bmw) :=\frac{d}{\sqrt{\pi}}\left(1+\min_{\varphi\in B_{r}(\vartheta)}(x-\varphi^{\top}\bmw)^2/R_n\right)^{-\frac{3}{2}}\\
     \cdot \left[R_n^{\frac{3}{4}}(|x|+M_n\|\bmw\|_2)\|\bmw\|_2+\frac{1}{4}R_n^{\frac{3}{4}}+R_n^{\frac{7}{4}}(|x|+M_n\|\bmw\|_2)^2\right],
\end{multline*}
where $M_n = \|\theta^*\|_2 + R_n$, $B_r(\vartheta)$ is the $L_2$-ball around $\vartheta$.
\end{lemma}
\begin{proof}
$$
\begin{dcases}
\grad_{\theta} \sqrt{f_{\theta,\gamma}(x|\bm{w})} = \gamma^{\frac{1}{4}}\pi^{-\frac{1}{2}}\left(1+(x-\theta^\top \bmw)^2/\gamma\right)^{-\frac{3}{2}}\frac{x-\theta^\top \bmw}{\gamma}\bmw,\\
\grad_{\gamma} \sqrt{f_{\theta,\gamma}(x|\bm{w})} = \frac{1}{4}\gamma^{-\frac{7}{4}} \pi^{-\frac{1}{2}}\left(1+(x-\theta^\top \bmw)^2/\gamma\right)^{-\frac{3}{2}}\left[\gamma + 3(x-\theta^\top \bmw)^2\right].
\end{dcases}
$$
By Cauchy's Mean Value Theorem, there exists $\varphi=\vartheta+\alpha(\theta-\vartheta)$, $\xi = \gamma_2 + \alpha(\gamma_1-\gamma_2)$ for some $\alpha\in(0,1)$ such that
\begin{align*}
    &\, \left|\sqrt{f_{\theta,\gamma_1}(x|\bm{w})} - \sqrt{f_{\vartheta,\gamma_2}(x|\bm{w})}\right| \\
    = &\,  \xi^{\frac{1}{4}}\pi^{-\frac{1}{2}}\left(1+(x-\varphi^\top \bmw)^2/\xi\right)^{-\frac{3}{2}}\frac{x-\varphi^\top \bmw}{\xi} (\theta-\vartheta)^\top \bmw  \\
     &\quad + \frac{1}{4}\xi^{-\frac{7}{4}} \pi^{-\frac{1}{2}}\left(1+(x-\varphi^\top \bmw)^2/\xi\right)^{-\frac{3}{2}}\left[\xi + 3(x-\varphi^\top \bmw)^2\right]   (\gamma_1-\gamma_2) \\
     \leq &\, \frac{d}{\sqrt{\pi}}\left(1+\min_{\varphi\in B_{r}(\vartheta)}(x-\varphi^{\top}\bmw)^2/R_n\right)^{-\frac{3}{2}}\cdot \|(\theta,\gamma_1)-(\vartheta,\gamma_2)\|_\infty\\
     &\quad \cdot \left[R_n^{\frac{3}{4}}(|x|+M_n\|\bmw\|_2)\|\bmw\|_2+\frac{1}{4}R_n^{\frac{3}{4}}+R_n^{\frac{7}{4}}(|x|+M_n\|\bmw\|_2)^2\right].
\end{align*}
\end{proof}

\begin{lemma}\label{lemma:cauchy_M2}
For $(\theta,\gamma)\in\Omega_n:=\{(\theta,\gamma)\in\Omega:\|\theta-\theta^*\|_2\leq R_n, \gamma\in[R_n^{-1},R_n]\}$, $R_n>0$, let
\begin{multline*}
     M_{\vartheta,\gamma_2,r}(x,\bmw) :=\frac{d}{\sqrt{\pi}}\left(1+\min_{\varphi\in B_{r}(\vartheta)}(x-\varphi^{\top}\bmw)^2/R_n\right)^{-\frac{3}{2}}\\
     \cdot \left[R_n^{\frac{3}{4}}(|x|+M_n\|\bmw\|_2)\|\bmw\|_2+\frac{1}{4}R_n^{\frac{3}{4}}+R_n^{\frac{7}{4}}(|x|+M_n\|\bmw\|_2)^2\right],  
\end{multline*}
where $M_n = \|\theta^*\|_2 + R_n$, $B_r(\vartheta)$ is the $L_2$-ball around $\vartheta$.
If $\|\bmw\|_\infty\leq 1$, then
$$\|M_{\vartheta,r}(x,\bmw)\|_2^2 \preceq M_n^9.$$
\end{lemma}

\begin{proof}
Similar to Lemma \ref{lemma:gaussian_M2}, we consider the value of $\min_{\varphi\in B_{r}(\vartheta)}(x-\varphi^{\top}\bmw)^2$ in different configurations of $(x,\bmw,\vartheta)$.
Since $\varphi = \vartheta+\tilde{r}v$ for some $\tilde{r}<r$ and some unit vector $v$, when $x$ is fixed, $\varphi^\top \bmw$ takes value within the interval $(u_-, u_+)$
where
$$u_- := \vartheta^\top \bmw-r\|\bmw\|_2,\qquad u_+ := \vartheta^\top \bmw+r\|\bmw\|_2,$$
with the infimum and supremum attained at $\vartheta-\frac{r\bmw}{\|\bmw\|_2}$ and $\vartheta+\frac{r\bmw}{\|\bmw\|_2}$ respectively.
Therefore $(x-\varphi^{\top}\bmw)^2$ is minimized by choosing $\varphi=\vartheta-r\bmw$ when $x<u_-$ and $\varphi=\vartheta+r\bmw$ when $x>u_+$. 
When $x\in(u_-,u_+)$, $(x-\varphi^{\top}\bmw)^2$ has minimum zero.
For simplicity, let 
$$
S(x,\bmw) = \left[R_n^{\frac{3}{4}}(|x|+M_n\|\bmw\|_2)\|\bmw\|_2+\frac{1}{4}R_n^{\frac{3}{4}}+R_n^{\frac{7}{4}}(|x|+M_n\|\bmw\|_2)^2\right]^2.
$$
Again, the integral $\|M_{\vartheta,\gamma_2,r}(x,\bmw)\|_2^2$ can be evaluated separately in three intervals
$\int_{-\infty}^{u_-}$,$\int_{u_-}^{u_+}$ and $\int_{u_+}^\infty$.
By noting that $\max\{|u_-|,|u_+|\}\leq (\|\vartheta\|_2+r)\|\bmw\|_2\leq M_n\|\bmw\|_2$,
\begin{equation*}
    \int_{u_-}^{u_+}\frac{1}{\pi} \left(1+\min_{\varphi\in B_{r}(\vartheta)}(x-\varphi^{\top}\bmw)^2/R_n\right)^{-3} S(x,\bmw) \dd x \preceq M_n^{\frac{17}{2}} \|\bmw\|_2^5 \preceq M_n^{\frac{17}{2}}.
\end{equation*}
Moreover, since the integrand is non-negative, for $k\leq 4$,
\begin{align*}
    &\, \int_{-\infty}^{u_-} \frac{1}{\pi} \left(1+\min_{\varphi\in B_{r}(\vartheta)}(x-\varphi^{\top}\bmw)^2/R_n\right)^{-3} |x|^k \dd x \\
    \leq &\, \int_{-\infty}^{\infty} \frac{1}{\pi} \left(1+(x-u_-)^2/R_n\right)^{-3} |x|^k\dd x \\
    = &\, \int_0^{\infty} \frac{1}{\pi} \left(1+(x-u_-)^2/R_n\right)^{-3} x^k \dd x + \int_0^{\infty} \frac{1}{\pi} \left(1+(x+u_-)^2/R_n\right)^{-3} x^k \dd x  \\
    \preceq &\, M_n^{k+\frac{1}{2}}.
\end{align*}
The last step is by computing the integral
\begin{align*}
    &\, \int_0^{\infty} \frac{1}{\pi} \left(1+(x- a)^2/R_n\right)^{-3} x^k \dd x & \\
    =&\, \int_0^{\infty} \frac{1}{\pi} \left(1+t^2/R_n\right)^{-3} (t+a)^k \dd t  & [t=y-a] \\
    \leq &\, 2\sum_{j=0}^{\floor{\frac{k}{2}}} \binom{k}{2j} \int_0^\infty t^{2j} a^{k-2j} \frac{1}{\pi}(1+t^2/R_n)^{-3}\dd t \\
    &\quad + \sum_{j=1}^{\ceil{\frac{k}{2}}}\binom{k}{2j-1} \int_0^\infty t^{2j-1}a^{k-2j-1}\frac{1}{\pi}(1+t^2/R_n)^{-3}\dd t \\
    \preceq &\, \max\{a,R_n\}^{k+\frac{1}{2}},
\end{align*}
while noting 
$$
\int_0^\infty \frac{1}{\pi} t^k(1+t^2/R_n)^{-3}\dd t \leq R_n^{\frac{k+1}{2}}.
$$
Therefore, 
$$
 \int_{-\infty}^{u_-} \frac{1}{\pi}\left(1+\min_{\varphi\in B_{r}(\vartheta)}(x-\varphi^{\top}\bmw)^2/R_n\right)^{-3} S(x,\bmw) \dd x \preceq M_n^9.
$$
The same upper bound applies to
$$
 \int_{u_+}^{\infty} \frac{1}{\pi}\left(1+\min_{\varphi\in B_{r}(\vartheta)}(x-\varphi^{\top}\bmw)^2/R_n\right)^{-3} S(x,\bmw) \dd x.
$$
Hence,
$$
\|M_{\vartheta,\gamma_2,r}(x,\bmw)\|_2^2  \preceq M_n^9.
$$
\end{proof}

\begin{lemma} \label{lemma:cauchy_hellinger}
Let $\|\bmw\|_\infty\leq 1$ with nonzero entries,
$$
f_{\theta,\gamma}(x|\bm{w}) := \frac{\sqrt{\gamma}}{\pi}(1+(x-\theta^\top \bmw)^2/\gamma)^{-1} .
$$
Then we have for $(\theta,\gamma)$ close to $(\theta^*,\gamma^*)$, there exists a positive constant $C$ independent of  $(\theta,\gamma)$ such that
$$
    h((\theta^*,\gamma^*),(\theta,\gamma)) \leq \phi_n \implies \|(\theta^*,\gamma^*)-(\theta,\gamma))\|_2 \leq C\phi_n.
$$
\end{lemma}
\begin{proof}
$$
  h^2((\theta^*,\gamma^*),(\theta,\gamma)) = \frac{1}{2}\int  \left(\sqrt{\frac{\sqrt{\gamma^*}}{\pi(1+(x-\theta^{*\top}\bmw)^2/\gamma^*)}} - \sqrt{\frac{\sqrt{\gamma}}{\pi(1+(x-\theta^{\top}\bmw)^2/\gamma)}}\right)^2\dd x.      
$$
Recall the grad of $\sqrt{f_{\theta,\gamma}}$ computed in Lemma \ref{lemma:cauchy_lipschitz}.
$$
\begin{dcases}
    \grad_{\theta} \sqrt{f_{\theta,\gamma}(x|\bm{w})} = \gamma^{\frac{1}{4}}\pi^{-\frac{1}{2}}\left(1+(x-\theta^\top \bmw)^2/\gamma\right)^{-\frac{3}{2}}\frac{x-\theta^\top \bmw}{\gamma} \bmw, \\
    \grad_{\gamma} \sqrt{f_{\theta,\gamma}(x|\bm{w})} = \frac{1}{4}\gamma^{-\frac{7}{4}} \pi^{-\frac{1}{2}}\left(1+(x-\theta^\top \bmw)^2/\gamma\right)^{-\frac{3}{2}}\left[\gamma + 3(x-\theta^\top \bmw)^2\right].
\end{dcases}
$$
By Cauchy's Mean Value Theorem, there exists $\varphi = \theta^* + \alpha(\theta-\theta^*)$, $\xi=\gamma^* + \alpha(\gamma-\gamma^*)$ for some $\alpha\in(0,1)$ such that
\begin{align*}
    &\, \int \left(\sqrt{\frac{\sqrt{\gamma^*}}{\pi(1+(x-\theta^{*\top}\bmw)^2/\gamma^*)}} - \sqrt{\frac{\sqrt{\gamma}}{\pi(1+(x-\theta^{\top}\bmw)^2/\gamma)}}\right)^2\dd x \\
    = &\, \int \frac{\sqrt{\xi}}{\pi}\left(1+\frac{(x-\varphi^\top \bmw)^2}{\xi}\right)^{-3}\left[\frac{x-\varphi^\top \bmw}{\xi}(\theta-\theta^*)^\top \bmw + \frac{\gamma-\gamma^*}{4\xi} + \frac{3}{\xi^2}(x-\varphi^\top \bmw)^2(\gamma-\gamma^*)   \right]^2 \dd x \\
    \succeq &\, \left((\theta-\theta^*)^\top \bmw\right)^2 +(\gamma-\gamma^*)^2.
\end{align*}
The last step is due to the value of integral with respect to $\frac{x-\varphi^\top \bmw}{\xi}$ is only dependent on $\xi$ which is `close' to $\gamma^*$ and thus has a finite lower and upper bound.
Thus, since $\|\bmw\|_\infty\leq 1$ with nonzero entries,
\begin{equation*}
    h^2((\theta^*,\gamma^*),(\theta,\gamma)) \succeq  \|(\theta,\gamma)-(\theta^*,\gamma^*)\|_2^2.
\end{equation*}
\end{proof}

\subsection{Quantile Regression Case}

In quantile regression, one would like to estimate the $\tau$-th quantile, denoted $q_{\tau}(X_i|\bm{W}_i)$ of the random variable $X_i$ conditioned on the covariates $\bm{W}_i$.
As shown in \cite{koenker1978regression}, solving the quantile regression problem is equivalent to solving the following optimisation problem:
$$
 \min_{\theta} \sum_{i} Q_{\tau}(x_i- \theta^\top \bm{w}_i), \quad Q_{\tau}(u) = \frac{1}{2}(|u|+(2\tau-1)u). 
$$
This is equivalent to maximising the likelihood
$$
\max_{\theta} \prod_{i} f_{\tau}(x_i;\bm{w}_i,\theta),
$$
where the probability density function $f_{\tau}(x;\bmw_i,\theta)$ is given by 
$$
f^{\tau}_{\theta}(x|\bm{w}) = \tau(1-\tau)\exp\left(-Q_{\tau}(x-\theta^\top \bm{w})\right),
$$
which corresponds to the asymmetric Laplace distribution.
Please refer to \cite{yu2001bayesian} for more details on Bayesian quantile regression.
First, we check Assumption 1
by using Lemma 2.
We showed that, by Lemma \ref{lemma:quantile_lipschitz}, for any $\theta,\vartheta\in \Omega_n:=\{\theta\in \Omega:\|\theta-\theta^*\|_2\leq R_n\}$ such that $\|\theta-\vartheta\|_\infty\leq r$,
$$
\left| \sqrt{f^{\tau}_{\theta}(x|\bm{w})} - \sqrt{f^{\tau}_{\vartheta}(x|\bm{w})} \right| \leq M_{\vartheta,r}(x,\bm{w})\|\theta-\vartheta\|_\infty,
$$
where $\tau^*=\max\{\tau,1-\tau\}$, $M_n = R_n+\|\theta^*\|_2$ and 
$$
M_{\vartheta,r}(x,\bm{w}) = \frac{d}{2}\exp\left(-\frac{1-\tau^*}{2}\min_{\varphi\in B_r(\vartheta)}|x-\varphi^\top \bm{w}| \right) \|\bm{w}\|_2.
$$
Then by Lemma \ref{lemma:quantile_M2}, assuming $\|\bmw\|_\infty\leq 1$,
$$
\|M_{\vartheta,r}\|_2^2 = \frac{d^2}{2(1-\tau^*)}\|\bm{w}\|_2^2 + \frac{d^2r}{2}\|\bm{w}\|_2^3=O(1).
$$
Thus Assumption 1 can be satisfied with the choice
$
    \phi_n=n^{-1/q}, R_n=n^{k}
$
for some $q>4$ and $k>0$ to be chosen for Assumption 3.
It is not hard to check that the likelihood function satisfies the condition in Lemma 6 
when $\theta$ is close enough to $\theta^*$, i.e., within $L_2$-distance of some constant $c_r$.
Thus $S_{0,\alpha}(t_n)$ contains an $L_2$-neighbourhood around $\theta^*$ with its radius proportional to $t_n$ up to a constant depending on $\theta^*$ and $c_r$.
Assumption 2 
 is satisfied by applying Lemma 7 
 and Remark 12. 
For Assumption 3
, we consider for large $R_n$ and apply Remark 13,
$$
\mbb{P}_{\pi_0}(\Omega_n^c) = 1-\mbb{P}_{\pi_0}(\Omega_{n}) \sim R_{n}^{d-2}e^{-R_{n}^2/2}.
$$
We need
$$
R_n^2=n^{2k} \geq 4c_1n\phi_n^2 = 4c_1n^{(q-2)/q},
$$
which can be satisfied by choosing $k  > (q-2)/2q$, for instance, $q=6$, $k>\tfrac{1}{6}$.

Finally, Lemma \ref{lemma:quantile_hellinger} also shows that $\phi_n$ and $\psi_n$ in Proposition 3 
can be chosen to diminish at the same rate.
Thus, we can choose $\psi_n\propto n^{-1/q}$ and $p_n=n^{-1/2q}$ which gives
$$
\epsilon_n \sim n^{-1/2q},\qquad \delta_n\sim \exp(-\tfrac{1}{2}c_1 n^{1-2/q}),
$$
where $c_1$ is a fixed constant and $q>4$.

\begin{lemma}\label{lemma:quantile_lipschitz}
Let 
$$
f^{\tau}_{\theta}(x|\bm{w}) := \tau(1-\tau)\exp\left(-Q_{\tau}(x-\theta^\top \bmw)\right) ,
$$
where $\theta\in\Omega_n:=\{\theta\in\Omega:\|\theta-\theta^*\|_2\leq R_n\}$, for some $R_n>0$,
$$
Q_{\tau}(u) := \begin{dcases}
    (\tau-1)u, & u\leq 0,\\
    \tau u, & u>0.
\end{dcases}
$$
If $\theta,\vartheta\in\Omega_n$ with $\|\theta-\vartheta\|_\infty\leq r$, then
$$
\left| \sqrt{f^{\tau}_{\theta}(x|\bm{w})} - \sqrt{f^{\tau}_{\vartheta}(x|\bm{w})} \right| \leq M_{\vartheta,r}(x,\bmw)\|\theta-\vartheta\|_\infty,
$$
where the function $M_{\vartheta,r}$ is given by 
$$
M_{\vartheta,r}(x,\bmw) := \frac{1}{2}\exp\left(-\frac{1-\tau^*}{2}\min_{\varphi\in B_r(\vartheta)}|x-\varphi^\top \bmw| \right) \|\bmw\|_2  .
$$
where $\tau^*=\max\{\tau,1-\tau\}$.
\end{lemma}
\begin{proof}
Note that
$$
\grad_{\theta}\sqrt{f^{\tau}_{\theta}(x|\bm{w})}=\begin{dcases}
    -\sqrt{\tau(1-\tau)} (1-\tau) \exp\left(-\frac{\tau-1}{2}(x-\theta^\top \bmw)\right)\bmw  , & x\leq\theta^\top \bmw,\\
    -\sqrt{\tau(1-\tau)} \tau \exp\left(-\frac{\tau}{2}(x-\theta^\top \bmw)\right)\bmw , & x>\theta^\top \bmw.
\end{dcases}
$$
Using the same argument as in Lemma 20, recall the $\tilde{\grad}$ operator from Definition 9 for the supremum of all directional derivatives. 
Let $\varphi=\theta+\alpha(\vartheta-\theta)$, then
\begin{align*}
    \sqrt{f^{\tau}_{\theta}(x|\bm{w})} - \sqrt{f^{\tau}_{\vartheta}(x|\bm{w})} \leq &\, \sup_{\alpha\in(0,1)} \left\|\tilde{\grad}_{\theta}\sqrt{f_\varphi^\tau(x|\bm{w})}\right\|_2 \|\theta-\vartheta\|_2 \\
    \leq &\,  \sup_{\varphi\in B_r(\vartheta)} \frac{d\tau^*}{2}\exp\left(- \frac{1-\tau^*}{2}\left|x-\varphi^\top \bmw\right|\right)\|\bmw\|_2 \|\theta-\vartheta\|_\infty  .
\end{align*}
Thus, we can take 
$$
M_{\vartheta,r}(x,\bmw) = \frac{d}{2}\exp\left(- \frac{1-\tau^*}{2}\min_{\varphi\in B_r(\vartheta)}|x-\varphi^\top \bmw|\right)\|\bmw\|_2  .
$$
\end{proof}

\begin{lemma}\label{lemma:quantile_M2}
Define 
$$\Omega_n:=\{\theta\in\Omega:\|\theta-\theta^*\|_2\leq R_n\},$$ 
for some $R_n>0$.
For $\vartheta\in\Omega_n$, let
$$
M_{\vartheta,r}(x,\bmw) := \frac{d}{2}\exp\left(-\frac{1-\tau^*}{2}\min_{\varphi\in B_r(\vartheta)}|x-\varphi^\top \bmw| \right) \|\bmw\|_2  ,
$$
where $\tau^*=\max\{\tau,1-\tau\}$.
If $\|\bmw\|_\infty\leq 1$, then
$$\|M_{\vartheta,r}(x,\bmw)\|_2^2 =\frac{d^2}{2(1-\tau^*)}\|\bmw\|_2^2 + \frac{d^2r}{2}\|\bmw\|_2^3=O(1).$$    
\end{lemma}
\begin{proof}
$$
\|M_{\vartheta,r}\|_2^2 = \frac{d^2}{4} \|\bmw\|_2^2\int \exp\left(-(1-\tau^*)\min_{\varphi\in B_r(\vartheta)}|x-\varphi^\top \bmw|\right) \dd x.
$$
The integral can be treated similarly to the Gaussian or the Cauchy case by considering the value of $\min_{\varphi\in B_r(\vartheta)}|x-\varphi^\top \bmw|$ in different cases.
Define
$$
u_- := \vartheta^\top \bmw - r\|\bmw\|_2, \quad u_+ := \vartheta^\top \bmw+r\|\bmw\|_2.
$$
Then $u_-\leq\varphi^\top \bmw\leq u_+$ and thus
$$
\min_{\varphi\in B_r(\vartheta)}|x-\varphi^\top \bmw| = \begin{dcases}
    u_- - x, & x\leq u_-, \\
    0, & u_- < x\leq u_+, \\
    x-u_+, & x>u_+.
\end{dcases}
$$
Now, we just compute separately $\int_{-\infty}^{u_-}$, $\int_{u_-}^{u_+}$ and $\int_{u_+}^\infty$.
\begin{align*}
    &\, \int_{-\infty}^{u_-} \exp\left(-(1-\tau^*)\min_{\varphi\in B_r(\vartheta)}|x-\varphi^\top \bmw|\right) \dd x \\
    = &\, \int_{-\infty}^{u_-} \exp\left(-(1-\tau^*)(u_- - x)\right) \dd x\\
    = &\, \frac{1}{1-\tau^*},
\end{align*}
\begin{align*}
    &\, \int_{u_-}^{u_+} \exp\left(-(1-\tau^*)\min_{\varphi\in B_r(\vartheta)}|x-\varphi^\top \bmw|\right) \dd x \\
    =&\, u_+ - u_- = 2r\|\bmw\|_2,
\end{align*}
and
\begin{align*}
    &\, \int_{u_+}^{\infty} \exp\left(-(1-\tau^*)\min_{\varphi\in B_r(\vartheta)}|x-\varphi^\top \bmw|\right) \dd x \\
     = &\, \int_{u_+}^{\infty} \exp\left(-(1-\tau^*)(x - u_+)\right) \dd x \\
    = &\, \frac{1}{1-\tau^*}.
\end{align*}
Thus
\begin{align*}
    \|M_{\vartheta,r}\|_2^2 = &\, \frac{d^2}{4}  \|\bmw\|_2^2\left[\frac{2}{1-\tau^*} + 2r\|\bmw\|_2\right]\\
    = &\, \frac{d^2}{2(1-\tau^*)}\|\bmw\|_2^2 + \frac{d^2r}{2}\|\bmw\|_2^3.
\end{align*}
Since $\|\bmw\|_\infty\leq1$, the above expression does not depend on the radius of the space $\Omega_n$ and is hence of constant order.
\end{proof}

\begin{lemma}\label{lemma:quantile_hellinger}
Let
$$
f^{\tau}_{\theta}(x|\bm{w}) := \tau(1-\tau)\exp\left(-Q_{\tau}(x-\theta^\top \bmw)\right) ,
$$
where $\theta\in\Omega$.
If $\|\bmw\|_\infty\leq 1$ with nonzero entries, then for $\theta$ close enough to $\theta^*$, there exists a positive constant $C$ independent of $\theta$ such that
$$
h(\theta^*,\theta)\leq \phi_n \implies \|\theta^*-\theta\|_2 \leq C\phi_n.
$$
\end{lemma}
\begin{proof}
Recall that
$$
h^2(\theta,\theta^*) = \frac{1}{2}\int \left(\sqrt{f_{\theta}^\tau(x|\bm{w})} - \sqrt{f_{\theta^*}^\tau(x|\bm{w})}\right)^2\dd x,
$$
and
$$
\grad_{\theta}\sqrt{f_\theta^\tau(x|\bm{w})}= \begin{dcases}
    -\sqrt{\tau(1-\tau)} \frac{1-\tau}{2}\exp\left(-\frac{\tau-1}{2}(x-\theta^\top \bmw)\right) \bmw, & x\leq \theta^\top \bmw,\\
    \sqrt{\tau(1-\tau)} \frac{\tau}{2}\exp\left(-\frac{\tau}{2}(x-\theta^\top \bmw)\right) \bmw, & x>\theta^\top \bmw.
\end{dcases}
$$
Consider the integrand in the following two cases:
\begin{enumerate}
    \item[(i)] When $x\leq \theta^{*\top}\bmw\leq \theta^\top \bmw$, there exists $\varphi_1=\theta+a_1(\theta^*-\theta)$ for some constant $a_1\in(0,1)$ such that
    $$
    \sqrt{f_{\theta}^\tau(x|\bm{w})} - \sqrt{f_{\theta^*}^\tau(x|\bm{w})} =\sqrt{\tau(1-\tau)}\frac{\tau-1}{2}\exp\left(-\frac{\tau-1}{2}(x-\varphi_1^\top \bmw)\right)\bmw^\top (\theta-\theta^*) .
    $$
    Thus
    \begin{align*}
        &\,  \int_{-\infty}^{\theta^{*\top}\bmw} \left(\sqrt{f_{\theta}^\tau(x|\bm{w})} - \sqrt{f_{\theta^*}^\tau(x|\bm{w})}\right)^2 \dd x \\
        = &\, \int_{-\infty}^{\theta^{*\top}\bmw} \frac{\tau(1-\tau)^3}{4}\exp\left(-(\tau-1)(x-\varphi_1^\top \bmw)\right)(\bmw^\top(\theta-\theta^*))^2 \dd x \\
        = &\, \frac{\tau(1-\tau)^2}{4}(\bmw^\top(\theta-\theta^*))^2 \exp(-(\tau-1)(\theta^*-\varphi_1)^\top \bmw) .
    \end{align*}
    \item[(ii)] When $\theta^\top \bmw\leq \theta^{*\top}\bmw\leq x$,
    there exists $\varphi_2=\theta+a_2(\theta^*-\theta)$ for some constant $a_2\in(0,1)$ such that
    $$
    \sqrt{f_{\theta}^\tau(x|\bm{w})} - \sqrt{f_{\theta^*}^\tau(x|\bm{w})} = \sqrt{\tau(1-\tau)}\frac{\tau}{2}\exp\left(-\frac{\tau}{2}(x-\varphi_2^\top \bmw)\right)x^\top (\theta-\theta^*) .
    $$
    Thus
    \begin{align*}
        &\,  \int_{\theta^{\top}\bmw}^\infty \left(\sqrt{f_{\theta}^\tau(x|\bm{w})} - \sqrt{f_{\theta^*}^\tau(x|\bm{w})}\right)^2 \dd x \\
        = &\, \int_{\theta^{\top}\bmw}^\infty \frac{\tau^3(1-\tau)}{4}\exp\left(-\tau(x-\varphi_2^\top \bmw)\right)(\bmw^\top(\theta-\theta^*))^2 \dd x \\
        = &\, \frac{\tau^2(1-\tau)}{4}(\bmw^\top(\theta-\theta^*))^2 \exp(-\tau(\theta^*-\varphi_2)^\top \bmw) .
    \end{align*}
\end{enumerate}
Since the integrand in $h(\theta^*,\theta)$ is non-negative, the integral is lower-bounded by the integral computed only on the regions described in the above two cases.
When $\theta$ is close to $\theta^*$ and $\|\bmw\|_\infty\leq 1$ with nonzero entries,
\begin{align*}
    h^2(\theta^*,\theta)\geq &\, \frac{1}{2}\frac{\tau^*(1-\tau^*)^2}{4} (\bmw^\top(\theta-\theta^*))^2\exp(-\tau^*a_3(\theta^*-\theta)^\top \bmw) \\
     \succeq &\, \|\theta-\theta^*\|_2^2,
\end{align*}
where $\tau^*=\max\{\tau,1-\tau\}$, $a_3 = \max\{1-a_1,1-a_2\}$.
The last step is to apply the covariance SVD argument as in the end of Lemma \ref{lemma:gaussian_hellinger}.
\end{proof}

\section{}
\subsection{Alternative Proof for Posterior Contraction}
The proof strategy for our main result is roughly divided into two parts: 
\begin{enumerate}
    \item Bounding the likelihood ratio within a compact neighbourhood around $\theta^*$;
    \item Bounding the posterior mass of the set complement of that neighbourhood.
\end{enumerate}
The second component of our analysis involves establishing the contraction rate of the posterior distribution, following the work of \cite{wong1995probability,shen2001rates}.
Posterior contraction rates were also investigated independently in \cite{ghosal2000convergence}, and \cite{ghosal2017fundamentals} presents results that extend beyond the i.i.d. observation setting.
In what follows, we briefly discuss the assumptions and conclusions of our work when the proof strategy of Theorem 8.11 in \cite{ghosal2017fundamentals} is adopted instead.
Please note that all the constants in this section are independent of the constants in the main text, even when the same symbols are used, although some are functionally equivalent.

\begin{definition}[Test]
Let $(\mbb{X},\mcal{X})$ be a measurable space. 
A test function $\varphi$ defined on $(\mbb{X},\mcal{X})$ is a measurable function $\varphi:\mbb{X}\rightarrow[0,1]$.
\end{definition}

Let $h^2(\theta,\theta')$ denote the squared Hellinger distance between two likelihood functions indexed by $\theta$ and $\theta'$, i.e., $f(\cdot;\theta)$ and $f(\cdot;\theta')$.
We are not concerned with the contaminated likelihood $k_p$ since the assumptions to be stated below all hold for any $p>0$ whenever they hold for $p=0$, i.e., for $f$.
\begin{assumption}[Basic Testing Assumption]\label{as:basic_test}
Let $\xi,K>0$ be some universal constants.
For every $n\in\mbb{N}$, $\phi>0$ and $\theta$ with $h(\theta,\theta^*)>\phi$, there exists a test $\varphi_n:\mbb{X}^n\rightarrow[0,1]$ such that
\begin{equation}
    \mbb{E}_{\theta^*}[\varphi_n]\leq \exp(-Kn\phi^2), \qquad \sup_{h(\theta',\theta)\leq \xi\phi} \mbb{E}_{\theta'}[1-\varphi_n]\leq \exp(-Kn\phi^2).
\end{equation}
\end{assumption}
By Proposition D.8 of \citep{ghosal2017fundamentals}, the above assumption is always satisfied for any likelihood model with universal constants $\xi=1/2$, $K=1/8$.

Recall that $\Omega$ denotes the parameter space of interest. 
Let $B_{KL}(\theta,\phi):=\{\theta'\in\Omega: KL(\theta\|\theta')<\phi\}$ denote the $\phi$-neighbourhood of $\theta$ in KL-divergence.
\begin{assumption}\label{as:contraction}
There exists a partition of $\Omega$, namely $\Omega_n\cup \Omega_n^c$ such that for constants $\phi_n,\bar{\phi}_n\geq n^{-1/2}$, and every sufficiently large $j$,
\begin{enumerate}
    \item[(i)] \begin{equation*}
        \frac{\mbb{P}_{\pi_0}(\theta: j\phi_n<h(\theta,\theta^*)\leq 2j\phi_n)}{\mbb{P}_{\pi_0}(B_{KL}(\theta^*,\phi_n))}\leq e^{Kj^2n\phi_n/2},
    \end{equation*}
    \item[(ii)] \begin{equation*}
        \sup_{\phi\geq\phi_n}\log N(\xi\phi, \{f(\cdot;\theta): \theta\in\Omega_n, h(\theta,\theta^*)\leq 2\phi\}) \leq n\phi_n^2,
    \end{equation*}
    where $N$ denotes the covering number under the Hellinger distance.
    \item[(iii)] There exists some constant $C>0$ such that
    \begin{equation*}
        \frac{\mbb{P}_{\pi_0}(\Omega_n^c)}{\mbb{P}_{\pi_0}(B_{KL}(\theta^*,\bar{\phi}_n))} = o\left(e^{-D_nn\bar{\phi}_n^2}\right),
    \end{equation*}
    for some $D_n\rightarrow\infty$.
\end{enumerate}
\end{assumption}

Define the neighbourhood set $A_n\subset\Omega$ by $A_n(r):=\{\theta\in\Omega:h(\theta,\theta^*)\geq r\phi_n\}$. Note that the size of the neighbourhood $A_n$ depends on an additional value $r$.
\begin{theorem}[Theorem 8.11 \citep{ghosal2017fundamentals}]\label{thm:ghosal}
Under Assumptions \ref{as:basic_test} and \ref{as:contraction},
$$
\mbb{P}_{\pi_n}(A_n(M_n)|\bm{X})\rightarrow0 \quad \text{ in $\mbb{P}_{\theta^*}^{(n)}$-probability}
$$
for any arbitrarily slow $M_n\rightarrow\infty$.

More specifically, there exists some constant $C_1>0$ such that for any $n$ large enough,
\begin{multline}\label{eq:main-ghosal}
    \mbb{P}_{\theta^*}^{(n)}\left(\mbb{P}_{\pi_n}(A_n(M_n)|\bm{X}) > e^{-\tfrac{1}{2}(KM_n^2-1)n\phi_n^2} + e^{-(\tfrac{1}{4}KM_n^2-D)n\phi_n^2} +  C_1e^{-\tfrac{1}{2}(D_n-2D)n\bar{\phi}_n^2}\right) \\
    < \frac{e^{-\tfrac{1}{2}(KM_n^2-1)n\phi_n^2}}{1-e^{-KM_n^2n\phi_n^2}} +\sum_{j\geq1}e^{-(\tfrac{1}{4}KM_n^2(2j^2-1)-D)n\phi_n^2} + e^{-\tfrac{1}{2}(D_n-2D)n\bar{\phi}_n^2} + \frac{2}{D}
\end{multline}
holds for any $D>0$.
\end{theorem}
\begin{proof}[Sketch Proof.]
First, by Assumption \ref{as:basic_test}, (ii) of Assumption \ref{as:contraction} and Theorem D.5 of \citep{ghosal2017fundamentals}, for given $M_n>1$, there exist tests $\varphi_n$ such that for every $j\in\mbb{N}$,
\begin{equation}\label{eq:exp_test}
   \mbb{E}_{\theta^*}[\varphi_n]\leq e^{n\phi_n^2}\frac{e^{-KM_n^2n\phi_n^2}}{1-e^{-KM_n^2n\phi_n^2}}, \qquad \sup_{\theta\in\Omega_n: h(\theta,\theta^*)>Mj\phi_n} \mbb{E}_{\theta}[1-\varphi_n] \leq e^{-KM_n^2j^2n\phi_n^2}. 
\end{equation}
As in the proof of Lemma 8, we split $\Omega_n$ into countably many shells 
$$
\mcal{S}_{n,j}:=\{\theta\in\Omega_n:M_n\\j\phi_n<h(\theta,\theta^*)\leq M_n(j+1)\phi_n\},
$$
thus $\cup_{j\geq1}\mcal{S}_{n,j}=\Omega_n\cap A_n(M_n)$. 
For notational simplicity, we will abbreviate $A_n(M_n)$ to just $A_n$.
Define the event $E_n\subset\mbb{X}^n$ such that 
$$
    \int_\Omega\prod_{i=1}^n \frac{k_p(X_i;\theta)}{k_p(X_i;\theta^*)}\pi_0(\theta)\dd \theta \geq e^{-2Dn\phi_n^2}\mbb{P}_{\pi_0}(B_{KL}(\theta^*,\phi_n)).
$$
Consider the same decomposition of the posterior mass $$
\mbb{P}_{\pi_n}(A_n|\bm{X}) \leq \mbb{P}_{\pi_n}(A_n\cap\Omega_n|\bm{X})+\mbb{P}_{\pi_n}(\Omega_n^c|\bm{X}).
$$
The first term may be bounded by
\begin{align*}
    &\,\mbb{P}_{\pi_n}(A_n\cap\Omega_n|\bm{X}) \leq \frac{\int_{A_n\cap\Omega_n}\prod_{i=1}^n \frac{k_p(X_i;\theta)}{k_p(X_i;\theta^*)}\pi_0(\theta)\dd \theta}{\int_\Omega\prod_{i=1}^n \frac{k_p(X_i;\theta)}{k_p(X_i;\theta^*)}\pi_0(\theta)\dd \theta}\\
    \leq &\, \mbb{I}_{E_n^c} + \mbb{I}_{E_n} \frac{\int_{A_n\cap\Omega_n}\prod_{i=1}^n \frac{k_p(X_i;\theta)}{k_p(X_i;\theta^*)}\pi_0(\theta)\dd \theta}{\int_\Omega\prod_{i=1}^n \frac{k_p(X_i;\theta)}{k_p(X_i;\theta^*)}\pi_0(\theta)\dd \theta}(\varphi_n+1-\varphi_n) \\
    \leq &\, \mbb{I}_{E_n^c} + \varphi_n + \frac{\int_{A_n\cap\Omega_n}\prod_{i=1}^n \frac{k_p(X_i;\theta)}{k_p(X_i;\theta^*)}\pi_0(\theta)\dd \theta(1-\varphi_n)}{e^{-2Dn\phi_n^2}B_{KL}(\theta^*,\phi_n)}.
\end{align*}
\begin{enumerate}
    \item By Theorem 6.26 in \cite{ghosal2017fundamentals}, the event $E_n^c$ holds with probability at most $1/D$ under $\mbb{P}_{\theta^*}^n$, contributing to one portion of the $1/D$ in the failure rate.
    \item The expectation of $\varphi_n$ under $\mbb{P}^n_{\theta^*}$ is bounded by applying (\ref{eq:exp_test}), which may be turned into a probabilistic bound in $\mbb{P}^n_{\theta^*}$ by applying Markov's inequality, contributing to the first term in the inner inequality of (\ref{eq:main-ghosal}) and the first term in the failure probability.
    \item The integral in the final term may be split into integrals on $\cup_{j\geq1}\mcal{S}_{n,j}$. Thus, the expectation of the last term is bounded using (i) in Assumption \ref{as:contraction}. Applying Markov's inequality to get the infinite-sum terms in (\ref{eq:main-ghosal}).
\end{enumerate}
Similarly, $\mbb{P}_{\pi_n}(\Omega_n^c|\bm{X})$ may be bounded by
\begin{align*}
    \mbb{P}_{\pi_n}(\Omega_n^c|\bm{X}) \leq \mbb{I}_{E_n^c} + \frac{\mbb{P}_{\pi_n}(\Omega_n^c|\bm{X})}{e^{-2Dn\bar{\phi}_n^2}\mbb{P}_{\pi_0}(B_{KL}(\theta^*,\bar{\phi}_n))}.
\end{align*}
\begin{enumerate}
    \item Again, by Theorem 6.26 in \cite{ghosal2017fundamentals}, the event $E_n^c$ holds with probability at most $1/D$ under $\mbb{P}_{\theta^*}^n$, contributing to the other portion of the $1/D$ in the failure rate.
    \item By (iii) of Assumption \ref{as:contraction}, the second term above is uniformly bounded for any $\bm{X}$, contributing to the third term in the inner inequality and failure rate of (\ref{eq:main-ghosal}).
\end{enumerate}

\end{proof}

\begin{remark}[Practicality of Assumptions]
The assumptions (i-iii) in Assumption \ref{as:contraction} replace Assumptions 2,1,3 in the main text respectively.
\begin{itemize}
    \item Assumption \ref{as:contraction} (i) is implied by the following condition: 

There exists constant a $C_2>0$ such that for constants $\phi_n\leq n^{-1/2}$, 
\begin{equation}\label{eq:prior_mass_centre}
    \mbb{P}_{\pi_0}(B_{KL}(\theta^*,\phi_n))\geq e^{-C_2n\phi_n^2},
\end{equation}
which is weaker than our Assumption 2 \citep{shen2001rates}.

\item Assumption \ref{as:contraction} (ii) may be checked through Lemma 2 with a slightly altered argument to bound the covering number (in Hellinger distance) instead of the bracketing number. 

\item By assuming (\ref{eq:prior_mass_centre}), Assumption \ref{as:contraction} (iii) may be checked through Lemma 7 if we take the prior $\pi_0$ to be standard Gaussian and $\Omega_n$ to be an expanding $L_2$-ball. The radius of $\Omega_n$ needs to grow polynomially with $n$ for $D_n-2D>0$, considering that $D$ is also growing with $n$.
\end{itemize}
Overall, the assumptions are somewhat equivalent to those made in the main text, with (\ref{eq:prior_mass_centre}) potentially easier to verify than our Assumption 2 in certain models.
\end{remark}

\begin{remark}
    Note that for (\ref{eq:main-ghosal}) to be meaningful, it is necessary that $KM_n^2>4D$ and $D_n>2D$ for the quantities to contract with $n$ and $M_n\phi_n\rightarrow0$ for the set $A_n$ to contract.
    Since $\phi_n$ contracts to 0 no faster than $n^{-1/2}$ by definition, $M_n$ cannot grow faster than $O(n^{1/2})$. 
    Consequently, $D$ cannot grow faster than $O(n)$, implying that the failure probability associated with the upper bound on the posterior mass can be shown to decay only at a polynomial rate in $n$.
    To prove the contraction rate in the strong sense (for almost sure convergence), a stronger version of (\ref{eq:prior_mass_centre}) is required, and we refer the reader to Theorem 8.9 \citep{ghosal2017fundamentals} for more details.
    However, the failure rate still decreases polynomially with $n$ when invoking Theorem 8.9 \citep{ghosal2017fundamentals}.
\end{remark}

\begin{remark}
    The minimal $n$ for Theorem \ref{thm:ghosal} to hold depends on the sequence of $M_n$ and the prior $\pi_0$, which can be significantly smaller than that of Theorem 5.
\end{remark}

\subsubsection{Potential Extensions}
In this paper, our primary focus has been on finite-dimensional models with i.i.d. observations.
However, posterior contraction results exist for a broader range of setups, including independent but not identical observations, misspecified models, and nonparametric models \citep{ghosal2017fundamentals}.

For our proof strategy to remain applicable in these extended frameworks, it is essential to identify a suitable contamination density $g(\bmx)$ that has a heavier tail than the likelihood function $f(\bmx;\theta)$, where $\theta$ is potentially infinite.
Our Assumption 4 requires $g(\bmx)$ to have a heavier tail than $\grad_{\theta}f(\bmx;\theta)$ for any $\theta$ to better characterise the contraction rate of $\epsilon_n$.
However, this might not be appropriate in the infinite-dimensional setting.
Instead,  it is more suitable to characterise the contraction rate of $d(\bmx;\theta)=\tfrac{k_p(\bmx;\theta)}{k_p(\bmx;\theta^*_0)}$, where $\theta^*_0$ denotes the minimiser, within the parameter space $\Omega$, of the KL divergence from $k_p(\cdot;\theta)$ to the true data-generating measure $p^*$, which coincides with $\theta^*$ if the model is correctly specified.
In most cases, the assumptions required to establish posterior contraction remain analogous to those employed in the finite-dimensional i.i.d. case, i.e., a lower bound on prior mass of a shrinking neighbourhood $A_n$ around $\theta^*_0$, an upper bound on prior mass of the complement of an expanding sieve $\Omega_n$, and an upper bound for the complexity (in terms of covering number or bracketing number) of the space of likelihood functions indexed by the sieve $\Omega_n$.
For a comprehensive exposition of these results, we refer the reader to Chapter 8 of \citet{ghosal2017fundamentals}.

\section{Additional Discussions}

\subsection{For Strictly Adjacent Dataset}\label{appx:adjacency}

Definition 1 is perhaps the most adopted definition for neighbouring datasets in the differential privacy literature, where the two datasets are explicitly required to have the same cardinality.

In some places, it is alternatively defined as one dataset being a proper subset of the other with strictly one fewer element, for instance, in \cite{dwork2008differential}.
To avoid ambiguity, we call two datasets $\bmX$ and $\bmZ$ \textit{strictly adjacent} if $\bmZ=\bmX\cup\{\bmz\}$ and $\bmz\notin\bmX$ (or the same statement with $\bmZ$ and $\bmX$ swapped holds).
Our result translates easily to the strictly adjacent case.

\begin{definition}
    Two datasets $\bm{Z}$ and $\bm{X}$ are \textit{strictly adjacent} if and only if $\exists \bmx\notin \bm{Z}$ such that $\bmX=\bm{Z}\cup\{\bmx\}$.
\end{definition}

\begin{remark}
Using the notation above, $\bm{X} = \bm{Z}\cup\{\bmx\}$, so $\bm{Z}$ and $\bm{X}$ are a pair of neighbouring datasets.
In this case, (4) becomes either 
$$
\mbb{P}_{\pi_n}(S|\bmX) \leq \sqrt{\eta_n}\left[\sqrt{\eta_n} + \frac{m_{n-1}(A_n^c, \bm{Z})}{m_n(\Omega,\bm{X})}\right] \mbb{P}_{\pi_{n-1}}(S|\bm{Z}) + \frac{m_{n-1}(A_n^c, \bm{Z})}{m_n(\Omega,\bm{X})}. 
$$
or
$$
\mbb{P}_{\pi_{n-1}}(S|\bm{Z}) \leq \sqrt{\eta_n}\left[\sqrt{\eta_n} + \frac{m_{n}(A_n^c, \bm{X})}{m_{n-1}(\Omega,\bm{Z})}\right] \mbb{P}_{\pi_{n}}(S|\bm{X}) + \frac{m_{n}(A_n^c, \bm{X})}{m_{n-1}(\Omega,\bm{Z})}. 
$$
\end{remark}

\begin{corollary} 
    Under Assumptions 1-5, for any pair of strictly adjacent datasets $\bmX$ and $\bmZ$, regardless of which one is larger.
    With slight abuse of notation, let $\pi_*$ denote the appropriate posterior distribution $\pi_n$ or $\pi_{n-1}$ depending on the size of the given dataset, we have
    $$
        \mbb{P}\Bigl(\mbb{P}_{\pi_{*}}(S|\bm{X}) \leq \exp(\epsilon_{n-1})\mbb{P}_{\pi_{*}}(S|\bm{Z}) + \delta_{n-1} \Bigr) \geq 1- Ce^{-nc_5 \phi_{n-1}^2}, 
    $$
    where  $\epsilon_n$ is defined with respect to $\tfrac{1}{2}\varepsilon_n$ and $\delta_n$, while $\varepsilon_n$,$\delta_n$, $t_n$, $c_5$ are defined as in Theorem 5. 
\end{corollary}
\begin{proof}
    Let $n=\max\{|\bmX|,|\bmZ|\}$.
    Recall that the term $m_n(\Omega,\cdot)$ only depends on the sequence $t_n$ and $m_n(A_n^c,\cdot)$ only depends on $\phi_n$.
    Denote the lagged sequences $Lt_n=t_{n-1}$,  $L\phi_n=\phi_{n-1}$, and the scaled sequences $\tilde{t}_n = \tfrac{n}{n-1}t_n$, $\tilde{\phi}_n =\sqrt{\tfrac{n}{n-1}}\phi_n$.
    For instance, if we want to evaluate $\tfrac{m_{n-1}(A_n^c, \bm{Z})}{m_n(\Omega,\bm{X})}$, we can evaluate the $m_n(\Omega,\bm{X})$ on $Lt_n$ but  $m_{n-1}(A_n^c, \bm{Z})$ on $\tilde{\phi}_{n-1}$, then the indices in the obtained result are synchronized to $n-1$.

    Since only one sequence needs to be lagged, without loss of generality, we can always synchronise to $n-1$ and thus the rest follows.
\end{proof}

\subsection{Privacy in Non-identifiable Models} \label{appx:identifiability}
In the paper, we mainly addressed the case when the model is identifiable.
\begin{definition}[Identifiability]
A set of likelihood functions $f(\bmx;\theta):\mcal{X}\times\Omega\rightarrow \mbb{R}_{\geq0}$ indexed by $\theta\in\Omega$ is \textit{identifiable} if and only if any likelihood function in the set is \textit{distinguishable} from the rest of the set.
We also call the models that adhere to that set of likelihood functions \textit{identifiable}.
\end{definition}

The "Identifiability Assumption" does not change the level of differential privacy measured in ($\epsilon$,$\delta$).
However, to determine the convergence rate for ($\epsilon$,$\delta$) in Theorem 5 requires sufficient knowledge of the model, including how the density functions are indexed. 

In other words, it is not crucial to have an identifiable parameterisation in order for the model to be differentially private.
In fact, under suitable regularity conditions, the level of privacy in terms of $\epsilon$ or ($\epsilon,\delta$) is intrinsic to the model, regardless of how its parameter space and likelihood function take form.
The level of privacy is preserved by applying a suitable transformation (not necessarily bijective) to the parameter space and the prior.
Hence, the resulting model can be considered as a reparameterisation of the original model.

To see this invariance, we first consider a simple case where the parameter space is transformed by a diffeomorphism (a differentiable bijective mapping), where the change of measure can be easily deduced.
\begin{proposition}\label{prop:transformation}
Let $(\Omega_\psi,\Sigma_\psi,\dd\psi)$ and $(\Omega_\theta,\Sigma_\theta,\dd\theta)$ be two measure spaces and a diffeomorphism $t:\Omega_\psi\rightarrow\Omega_\theta$ such that $t^{-1}$ is well-defined and differentiable. 
Then the differential privacy level of model $f_\psi(\bmx;\psi)$ parameterised by $\Omega_\psi$ with prior $\pi_0(\psi)$ is the same as the privacy level of model $f_\theta(\bmx;\theta)$ parameterised by $\Omega_\theta$ with prior $\tilde{\pi}_0(\theta)$,
$$
f_\theta(\bmx;\theta) = f_\psi(\bmx;t^{-1}(\theta))\cdot |J_{t^{-1}}(\theta)|, \qquad \tilde{\pi}_0(\theta) = \pi_0(t^{-1}(\psi))\cdot |J_{t^{-1}}(\theta)|
$$
where $J_{t^{-1}}(\theta)$ is the Jacobian of the inverse function.
\end{proposition}
\begin{proof}
Recall that $f_\psi(\bmx;\psi)$ and $\pi_0(\psi)$ are the likelihood and the prior with respect to the parameter space $\Omega_\psi$, and
$$
f_\theta(\bmx;\theta) = f_\psi(\bmx;t^{-1}(\theta))\cdot |J_{t^{-1}}(\theta)|, \qquad \tilde{\pi}_0(\theta) = \pi_0(t^{-1}(\psi))\cdot |J_{t^{-1}}(\theta)|,
$$
where $J_{t^{-1}}(\theta)$ is the Jacobian of the inverse function.
Note that $f_\theta$ and $\tilde{\pi}_0$ are derived from $f_\psi$, $\pi_0$ through a change of measure under mapping $t$.

For any quintuples $(S,\epsilon,\delta,\bmX,\bmZ)$, where $S\in \Sigma_\theta$, $\epsilon,\delta\geq0$, $\bmX,\bmZ$ are neighbouring datasets, such that
\begin{align*}
   &\, \mbb{P}_{\pi_\theta}(S|\bmX) \leq  e^\epsilon \mbb{P}_{\pi_\theta}(S|\bmZ) +\delta\\
   \implies & \, \mbb{P}_{\pi_\psi}(t^{-1}(S)|\bmX) \leq  e^\epsilon \mbb{P}_{\pi_\psi}(t^{-1}(S)|\bmZ) +\delta,
\end{align*}
and vice versa, since the integral is preserved through the change of measure.
Therefore, a pair of $(\epsilon,\delta)$ holds for a model parameterised in $\Omega_\psi$ if and only if it holds for that model parameterised in $\Omega_\theta$ through the diffeomorphism $t$.
\end{proof}

Similarly, it is also possible to link a non-identifiable model in $\Omega_\psi$ to the identifiable version of itself in $\Omega_\theta$ while preserving the privacy level.
To begin with, we need to reduce the redundancy in the non-identifiable space $\Omega_\psi$.
Consider the following relation $\sim$ such that
$$
\psi_1 \sim \psi_2 \iff \nexists A\subset \mcal{X} , \mu(A)>0 \text{ such that } \forall \bmx\in A, \,f(\bmx;\psi_1)\neq f(\bmx;\psi_2).
$$
where $\mu$ is the base measure for the measure space $\Omega_{\psi}$.
That is, $\psi_1$ and $\psi_2$ lead to two likelihood functions that are equal almost everywhere.
Since equal a.e. could create an issue in the DP inequality as the worst case could land on the discontinuity, we consider $f(\bmx;\psi)$ to be continuous in $\bmx$.
In which case, $\psi_1\sim\psi_2$ if and only if they represent the same likelihood function on the whole of $\mcal{X}$.

Denote $\Omega_\theta:=\quot{\Omega_{\psi}}{\sim}$ and consider the quotient map
$$
\begin{array}{cccc}
  t: & \Omega_\psi & \rightarrow & \Omega_\theta, \\
     & \psi & \mapsto & [\psi].
\end{array}
$$
If there is a way to choose every representative $[\psi]$ such that $t$ is measurable, then we have the following proposition.
\begin{proposition}\label{prop:non-id-quotient}
Let $(\Omega_\psi,\Sigma_\psi,\dd\psi)$ be a measure space such that the model with likelihood function $f_\psi(x;\psi)$ and prior $\pi_0(\psi)$ is non-identifiable.
Suppose that the quotient map $t:\Omega_\psi\rightarrow\Omega_\theta\subset\Omega_\psi$ that reparameterises the model onto the measure space $(\Omega_\theta,\Sigma_\theta,\dd\theta)$ is measurable, then the model parameterised in $\Omega_\theta$ has the same level of differential privacy as the model parameterised in $\Omega_\psi$.
\end{proposition}
\begin{proof}
    We will prove that if $t$ is measurable,
    $$
    (\epsilon,\delta) \text{ holds for } \Omega_\psi \iff (\epsilon,\delta) \text{ holds for } \Omega_\theta
    $$
    under the same construction of the quotient map $t$, pushforward measure $t_*(\mbb{P}_{\pi_0})$ and its density $\tilde{\pi}_0$ with respect to $\dd\theta$ as in the previous proof.
    Since the pushforward measure and the design of the quotient ensure that any integrals are preserved, the $\implies$ direction is trivial.

    For the $\impliedby$ direction, assume that $(\epsilon,\delta)$ holds for $\Omega_\theta$, hence for any $S\in\Sigma_\theta$, $\bmX,\bmZ$
    $$
        \mbb{P}_{\pi_\theta}(S|\bmX)\leq e^\epsilon \mbb{P}_{\pi_\theta}(S|\bmZ) + \delta,
    $$
    equivalently, for any $S\in\Sigma_\theta$, $\bmX,\bmZ$,
    $$
        \mbb{P}_{\pi_\psi}(t^{-1}(S)|\bmX)\leq e^\epsilon \mbb{P}_{\pi_\psi}(t^{-1}(S)|\bmZ) + \delta.
    $$
    Note the following identities
    \begin{align*}
        \mbb{P}_{\pi_\psi}(A|\bmX) = &\, \int_A \pi_\psi(\psi|\bmX)\dd \psi \\
        =&\, \int_A \frac{\pi_\psi(\psi|\bmX)}{\pi_\psi(\psi|\bmZ)}\pi_\psi(\psi|\bmZ)\dd\psi \\
        =&\, \int_A \left(\frac{\pi_\psi(\psi|\bmX)}{\pi_\psi(\psi|\bmZ)} - e^\epsilon\right)\pi_\psi(\psi|\bmZ)\dd\psi + e^\epsilon\mbb{P}_{\pi_\psi}(A|\bmZ).
    \end{align*}
    For $(\epsilon,\delta)$ to hold for $\Omega_\psi$, we need
    $$
        \int_A \left(\frac{\pi_\psi(\psi|\bmX)}{\pi_\psi(\psi|\bmZ)} - e^\epsilon\right)\pi_\psi(\psi|\bmZ)\dd\psi \leq \delta.
    $$
    Suppose there exists $B\subset A$ such that
    $$
        \frac{\pi_\psi(\psi|\bmX)}{\pi_\psi(\psi|\bmZ)} < e^\epsilon, \quad \forall \psi\in B,
    $$
    then
    $$
        \mbb{P}_{\pi_\psi}(A\setminus B|\bmX) - e^\epsilon \mbb{P}_{\pi_\psi}(A\setminus B|\bmZ) \geq \mbb{P}_{\pi_\psi}(A|\bmX) - e^\epsilon \mbb{P}_{\pi_\psi}(A|\bmZ).
    $$
    Without loss of generality, it is now sufficient to check that $(\epsilon,\delta)$ holds for all sets $A\in\Sigma_\psi$ such that 
    \begin{equation}\label{eq:ratio_lowerbound}
        \frac{\pi_\psi(\psi|\bmX)}{\pi_\psi(\psi|\bmZ)} \geq e^\epsilon, \quad \forall \psi\in A,
    \end{equation}
    since any other sets require a smaller $\delta$ for the differential privacy inequality to hold.
    
    For any set $A\in \Sigma_\psi$ such that (\ref{eq:ratio_lowerbound}) holds, let $S=t(A)$, note that any $\psi\in t^{-1}(S)\supseteq S$ also satisfies (\ref{eq:ratio_lowerbound}),
    $$
    \mbb{P}_{\pi_\psi}(t^{-1}(S)|\bmX) = \int_{t^{-1}(S)}\underbrace{\left(\frac{\pi_\psi(\psi|\bmX)}{\pi_\psi(\psi|\bmZ)} - e^\epsilon\right)}_{\geq 0} \pi_{\psi}(\psi|\bmZ)\dd \psi + e^\epsilon \mbb{P}_{\pi_\psi}(t^{-1}(S)|\bmZ).
    $$
    Since $A\subset t^{-1}(S)$, 
    $$
        \delta\geq \mbb{P}_{\pi_\psi}(t^{-1}(S)|\bmX) - e^\epsilon \mbb{P}_{\pi_\psi}(t^{-1}(S)|\bmZ) \geq \mbb{P}_{\pi_\psi}(A|\bmX) - e^\epsilon \mbb{P}_{\pi_\psi}(A|\bmZ).
    $$
    and hence the $\impliedby$ direction is proved.
\end{proof}
\begin{remark}
    The forward invariance in differential privacy level through mapping $t$ essentially follows from the post-processing property \citep{dwork2014algorithmic}, though note that the mapping has to be measurable for the equations to be well-defined.
\end{remark}

\section{Extra Figures}
Figure \ref{fig:dim_compare} is a replot of Figure 2 in the main text, where now the $\epsilon$ is plotted in log scale, and it is easier to see the polynomial decay in $\epsilon$ with respect to $n$.
\begin{figure}[t]
    \centering
    \begin{subfigure}{0.8\linewidth}
        \centering
        \caption{}\label{subfig:lr_appx}
        \includegraphics[width=\linewidth]{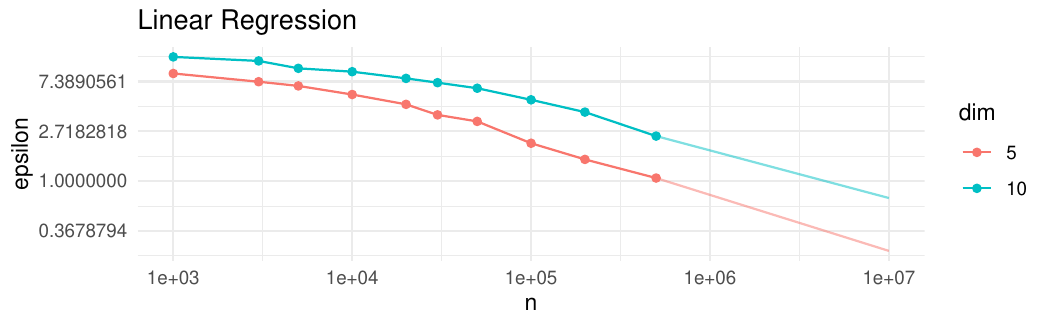}
    \end{subfigure} \hfill
    \begin{subfigure}{0.8\linewidth}
        \centering
        \caption{}\label{subfig:lgr_appx}
        \includegraphics[width=\linewidth]{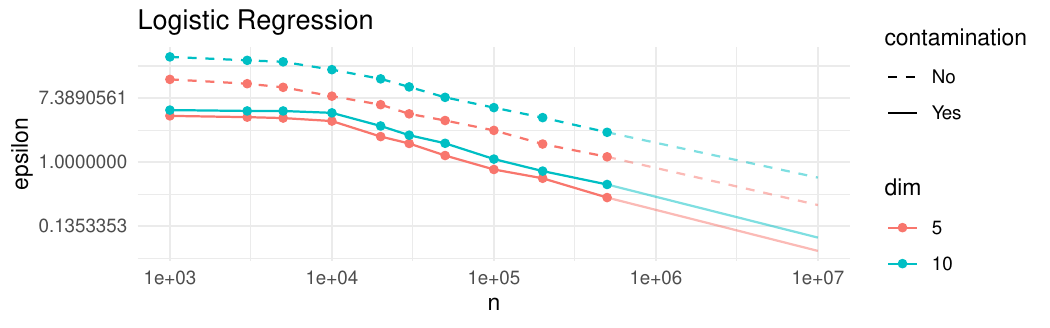}
    \end{subfigure}
    \begin{subfigure}{0.8\linewidth}
        \centering
        \caption{}\label{subfig:cr_appx}
        \includegraphics[width=\linewidth]{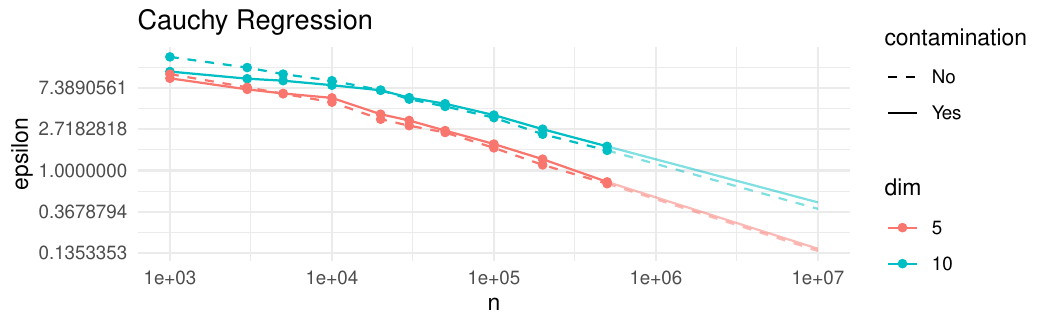}
    \end{subfigure}
    \caption{Figure 2 plotted in log-scale. Plots of estimated $\epsilon_n$ under the setting of Linear Regression, Logistic Regression and Cauchy Regression as the dataset size $n$ varies. The pale lines without nodes are extrapolated estimations.}
    \label{fig:dim_compare}
\end{figure}

\end{appendix}

\begin{acks}[Acknowledgments]
We would like to thank all the reviewers and the associate editor for reviewing and providing constructive comments!
\end{acks}
\begin{funding}
All authors were supported by the EPSRC research grant "Pooling INference and COmbining Distributions Exactly: A Bayesian approach (PINCODE)", reference (EP/X028100/1, EP/X028119/1, EP/X028712/1, EP/X027872/1). 

LA, HD, MP and GOR were also supported by the UKRI grant, "On intelligenCE And Networks (OCEAN)", reference (EP/Y014650/1).

GOR was also supported by EPSRC grants Bayes for Health (R018561), CoSInES (R034710), and EP/V009478/1.
\end{funding}

\begin{supplement}
\stitle{Supplement to "Privacy Guarantees in Posterior Sampling under Contamination"}

\sdescription{The supplementary material contains
the proofs not included in the appendix and additional technical lemmas for the example section.}
\end{supplement}


\bibliographystyle{imsart-nameyear} 
\bibliography{bibtex}       
\nocite{ghosal1997review}
\nocite{bennett1962probability}

\end{document}